\documentclass[11pt]{article}

\usepackage[margin=1in]{geometry}
\usepackage[colorlinks,citecolor=blue,urlcolor=blue]{hyperref}      
\usepackage{url}            
\usepackage{booktabs,multirow}       
\usepackage{amsmath, amsfonts, amssymb, amsthm, mathtools,mathrsfs}       
\usepackage{authblk}

\usepackage{graphicx,caption,subcaption}

\usepackage{multicol,longtable,multirow,booktabs}

\usepackage[shortlabels]{enumitem}
\usepackage[authoryear]{natbib}

\graphicspath{{../figures/gab-optimality/}, {../../figures/gab-optimality/}}

\newcommand{\E}{\mathbb{E}}
\newcommand{\R}{\mathbb{R}}

\newcommand{\Dcal}{\mathcal{D}}

\newcommand{\Fcal}{\mathcal{F}}

\newcommand{\Kcal}{\mathcal{K}}
\newcommand{\Lcal}{\mathcal{L}}

\newcommand{\Scal}{\mathcal{S}}

\newcommand{\bb}[1]{\boldsymbol{#1}}
\newcommand{\bbhat}[1]{\widehat{\boldsymbol{#1}}}

\newcommand{\bbtilde}[1]{\widetilde{\boldsymbol{#1}}}
\newcommand{\ind}[1]{\boldsymbol{1}_{#1} }

\newcommand{\tr}{^{\intercal}}
\newcommand{\norm}[1]{\left\Vert #1 \right\Vert}
\newcommand{\inner}[1]{\left\langle #1 \right\rangle}

\newcommand{\normdist}{\mathcal{N}}

\newcommand{\var}{\mathrm{Var}}
\newcommand{\prob}{\mathbb{P}}

\DeclareMathOperator*{\argmin}{arg\,min}

\newcommand{\IF}{\mathrm{IF}}

\theoremstyle{plain}

\newtheorem{theorem}{Theorem}
\newtheorem{proposition}{Proposition}
\newtheorem{corollary}{Corollary}
\newtheorem*{informalthm}{Theorem (Informal)}

\theoremstyle{definition}

\newtheorem{remark}{Remark}

\numberwithin{equation}{section}

\title{Universally Optimal Robustness-Efficiency Tradeoffs for a General Class of Minimum Divergence Estimators}
\date{}
\author[1]{Subhrajyoty Roy}
\affil[1]{Washington University in St. Louis, St. Louis, USA}

\author[2]{Supratik Basu}
\affil[2]{Duke University, Durham, USA}

\author[3]{Abhik Ghosh}
\author[3]{Ayanendranath Basu}
\affil[3]{Indian Statistical Institute, Kolkata, India}

\begin{document}

\maketitle

\begin{abstract}
    Balancing the efficiency of an estimator under ideal conditions against its robustness under contamination remains a central challenge in robust statistics. While minimum divergence methods offer a flexible alternative to traditional M-estimation, choosing the appropriate discrepancy measure has historically relied on heuristic or empirical justifications. This manuscript introduces a rigorous optimality criterion for this selection process. By investigating the comprehensive Generalized Alpha-Beta Divergence (GABD) family, we explicitly characterize the Pareto frontier dictating the lowest possible asymptotic variance for any strictly enforced asymptotic breakdown point. Our main theoretical results establish that the estimator achieving this mathematical optimum invariably falls within the extended $(\phi, \gamma)$-divergence class. Crucially, the derived optimal tuning parameter, $\phi^*$, given other parameters, depends solely on the desired breakdown threshold and is entirely invariant to both the assumed parametric model and the exact nature of the data contamination. Supported by comprehensive derivations of asymptotic normality, influence functions, and breakdown thresholds for both continuous and discrete settings, this work offers a unified, theoretical resolution to the long-standing problem of optimal divergence selection in robust inference.
\end{abstract}


\section{Introduction}\label{sec:intro}

A fundamental tension in statistical inference is the trade-off between estimation efficiency at the true model and robustness against data contamination. Since the inception of robust statistics, the general consensus, barring scattered exceptions, is that the robustness of an estimator must be purchased at an associated cost of efficiency loss, which may be substantial depending on the desired level of robustness. Thus, a natural question emerges: whether one can construct an estimator that attains maximal efficiency while guaranteeing a prescribed level of robustness.  This paper aims to demonstrate that this question can be answered positively within a significantly rich class of minimum divergence estimators, and hints at a universality principle in achieving this optimal balance.

A framework to answer this question was developed through the seminal work of~\cite{huber1964robust}, utilizing a minimax approach to bound the asymptotic bias of an estimator under contamination. Although this led to the emergence of Huber's $\psi$-function and a rich mathematical theory of M-estimation, the results were limited to location and scale models. To extend this further, \cite{hampel1974influence} introduced the infinitesimal approach, seeking maximal efficiency given a bounded influence function. This was also later extended to the regression setup~\citep{yohai1997optimal}; readers are referred to~\cite{avella2015robust} for a comprehensive summary of these approaches. However, the influence function, being a local measure of sensitivity, has limited capacity to distinguish between robust and nonrobust estimators. \cite{lindsay1994efficiency} and \cite{basu1994minimum} demonstrated the existence of a whole class of estimators that have the same (first-order) influence functions but differ strikingly in their robustness properties. A global measure of robustness, most notably the breakdown point~\citep{Huber_Donoho_1983}, is therefore appropriate to quantify the intrinsic reliability of an estimator. Yet, as pointed out by \cite{huber2009nonoptimality}, historical attempts to strictly maximize the breakdown point (e.g., \cite{rousseeuw1984least, rousseeuw1984robust, lopuhaa1992highly}) often yielded highly unstable or non-smooth estimators, rendering them challenging to be implemented in practice.

Due to its natural appeal, the minimum divergence estimation framework has gained considerable prominence over the past few decades as an alternative to the traditional M-estimation approach. Beyond classical statistics, this framework has been adapted in Bayesian computation~\citep{bhattacharya2019bayesian, miller2019robust}, generative modelling~\citep{nowozin2016_fGAN}, distributionally robust optimization~\citep{namkoong2016DRO}, robust neural networks~\citep{amid2019robust, werner2024global, ghosh2026provably}, etc. The minimum divergence estimator is defined as the value of the model parameter that minimizes a discrepancy measure $d(\cdot, \cdot)$ between the statistical model and the observed data; further exposition on this framework is provided in Section~\ref{sec:mgabde-defn}. The choice of the divergence measure $d$, thus, encodes the robustness-efficiency trade-off behavior; see~\cite{amari2010information} for a geometric perspective. Consequently, the field has witnessed significant activity aimed at constructing and identifying useful divergences that strike a balance between efficiency and robustness; this has led to the development of divergences such as the $f$-divergence family~\citep{csiszar1967information}, the power divergence (PD) family~\citep{cressie1984multinomial}, the Hellinger divergence~\citep{tamura1986minimum, toma2007minimum}, the density power divergence (DPD) family~\citep{basu1998robust, ghosh2013robust}, the logarithmic density power divergence (LDPD) family~\citep{jones2001comparison,fujisawa2008robust}, the S-divergence (SD) family~\citep{ghosh2017generalized}, the logarithmic S-divergence (LSD) family~\citep{maji2016logarithmic}, the alpha-beta (AB) divergence family~\citep{cichocki2010families}, the bridge divergence family~\citep{kuchibhotla2019statistical}, $\Kcal$-divergences~\citep{sorek2022kdiv}, etc. While each proposal achieves robustness guarantees in terms of bounded influence functions or high breakdown points, none of them emerged as a universally optimal choice. The selection of divergence to be used has remained ad-hoc, largely guided by analytical tractability, optimization constraints, and empirical performance rather than by a fundamental optimality principle.

\begin{figure}[htbp]
    \centering
    \includegraphics[width=0.7\linewidth]{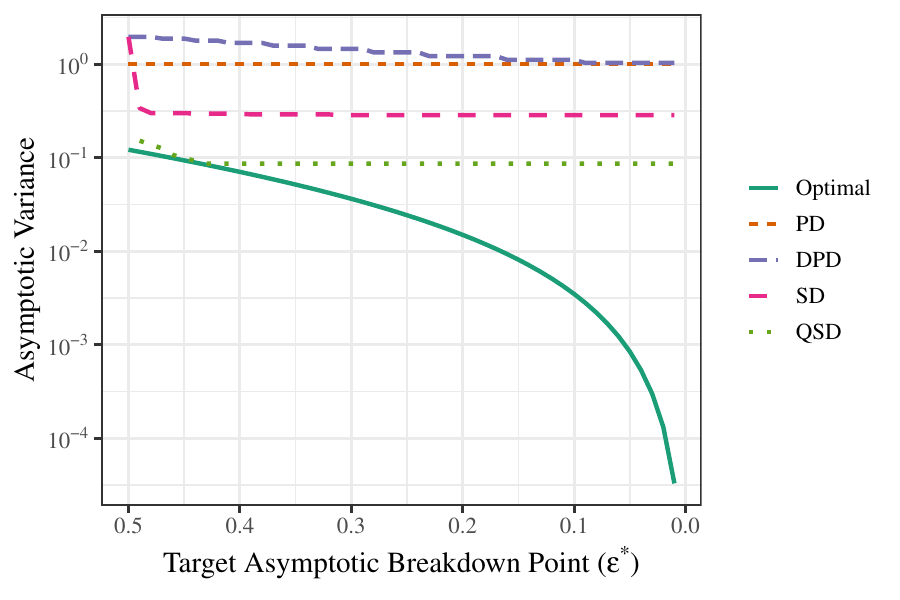}
    \caption{Best possible asymptotic variance given a breakdown constraint for different families of minimum divergence estimators (over all hyperparameter choices), in an exponential distribution scale estimation setting, vide Section~\ref{appendix:frontier-fig} of supplement. Here, additionally, we show similar breakdown-variance trade-offs for the minimum PD, DPD, SD, and QSD (quadratic S-divergence, defined as in~\eqref{eqn:jhhb-div} with $\phi = 2$) estimators.}
    \label{fig:frontier-motivation}
\end{figure}

This work establishes this optimality principle by employing a minimax approach. A decisive step toward structural unification was recently taken by~\cite{roy2025characterization}, who introduced a very general class of statistical divergences, called the ``Generalized Alpha-Beta Divergence'' (GABD) class, encompassing virtually all of the above families as special or limiting cases. Within this very general class of minimum divergence estimators, we obtain the Pareto Frontier as illustrated through Figure~\ref{fig:frontier-motivation} (underlying set-ups being described in Supplement~\ref{appendix:frontier-fig}), which maps the minimum achievable worst-case variance for any prescribed breakdown point, and contrasts it with several existing minimum divergence estimators. Remarkably, our results establish that this variance-minimizing estimator, subject to a global robustness constraint, always belongs to an extended $(\phi,\gamma)$-divergence family~\citep{jones2001comparison}; the detailed form of the divergence is provided in Eq.~\eqref{eqn:jhhb-div}. Prior works have often oscillated arbitrarily between specific boundary cases of this extended $(\phi,\gamma)$-divergence family, such as the density power divergence ($\phi = 1, \beta = 1$) or the logarithmic density power divergence ($\phi \to 0+, \beta = 1$). Unfortunately, this ad-hoc choice may be suboptimal, as shown in our results. Furthermore, the optimal tuning parameter $\phi^\ast$ for this family is universal in scope: it is completely agnostic to the underlying parametric model and the structure of the contaminating distribution, as a consequence of a functional inequality; more details will be elaborated in Section~\ref{sec:optimality}. An informal version of our main result on minimax optimality is presented below, whereas the technical version is deferred to Theorems~\ref{thm:gab-optimality-discrete} and~\ref{thm:gab-optimality-beta0}.

\begin{informalthm}
    Consider any i.i.d. sample $X_1, \dots, X_n \sim g$; the true density $g$ is unknown and is modelled by a parametric family $f_{\bb{\theta}}$ for $\bb{\theta} \in \Theta$, both dominated by a command measure $\mu$. Suppose $g = f_{\bb{\theta}^g}$ for some $\bb{\theta}^g \in \Theta$. Suppose, $v^{(\alpha,\beta),\psi}(g)$ denotes the asymptotic variance and $\epsilon^{(\alpha,\beta),\psi}(g)$ denotes the asymptotic breakdown point (see Eq.~\eqref{eqn:bp-working-defn}) of the minimum generalized alpha-beta divergence estimator (MGABDE) as defined in Eq.~\eqref{eqn:gab-defn-1}. Then, for any exogenously given $\epsilon^\ast \in (0, 1/2]$ and $\alpha > 0, \beta \geq 0$, under suitable regularity conditions,
    \begin{equation*}
        \sup_{g : \int g^{\alpha+\beta}\mathrm{d}\mu < \infty } v^{(\alpha,\beta),\psi^\ast}(g)  = \inf_{\psi \in S} \sup_{g: \int g^{\alpha+\beta}\mathrm{d}\mu < \infty } v^{(\alpha,\beta), \psi}(g),
    \end{equation*}
    \noindent where,
    \begin{equation*}
        S = \left\{ \psi: \inf_{g : \int g^{\alpha+\beta}\mathrm{d}\mu < \infty } \epsilon^{(\alpha,\beta),\psi}(g) \geq \epsilon^\ast \right\}, \
        \psi^\ast(x) = \frac{(x^{\phi^\ast} - 1)}{\phi^\ast},
    \end{equation*}
    \noindent and,
    \begin{equation*}
        \phi^\ast =
        \begin{cases}
            \log(1+\beta/\alpha)/(\beta\log(1-\epsilon^\ast)), & \text{ if } \beta > 0, \\
            -1/(\alpha \log(1-\epsilon^\ast)),                 & \text{ if } \beta = 0.
        \end{cases}
    \end{equation*}
\end{informalthm}

The remainder of the paper is organized as follows. Section~\ref{sec:prelims} outlines the generalized alpha-beta divergence family and the core principles of minimum divergence estimation. Sections~\ref{sec:asymp-normality}--\ref{sec:bp-analysis} build the theoretical foundation by establishing asymptotic normality and deriving key robustness metrics, specifically via influence function and asymptotic breakdown point analyses. Section~\ref{sec:optimality} then presents the main optimality results, discussing their broad implications and providing additional theoretical insights. We conclude in Section~\ref{sec:empirical} with two empirical studies designed to corroborate our theoretical findings. All proofs and extended technical details are deferred to the supplementary material.

\textbf{Notation.} Before diving deep into the mathematical intricacies, we introduce some notation that we will use throughout this paper. For a density $f$ dominated by a measure $\mu$, we define its $L^p$-norm as $\Vert f\Vert_p = \left( \int f^{p}d\mu \right)^{1/p}$, for any $p > 0$, provided that the integral exists and is finite. If the integral does not converge, we define $\Vert f\Vert_p = \infty$. Let $\Lcal^p$ denote the collection of densities $f$ such that $\Vert f\Vert_p$ is finite. In the same spirit, we also define a $(p, q)$-inner product between the densities $f$ and $g$ as $\inner{f, g}_{p, q} = \int f^p g^q d\mu$, for any $p, q \in \R$. Note that this is not a proper inner product in the mathematical sense since it does not satisfy the commutative property. For an interval $I$, $C^k(I)$ denotes the class of all functions that are $k$-times continuously differentiable on $I$ for $k \geq 1$, and the set of all continuous functions on $I$ for $k = 0$. Given any two sequences $a_n$ and $b_n$, we write $a_n \asymp b_n$ if they are asymptotically equivalent, i.e., there exist constants $c, C > 0$ satisfying $ca_n < b_n < Ca_n$ for all sufficiently large $n$. For a sequence of random variables $X_n$, we use the notation $X_n \to_\prob X$ to denote convergence in probability, $X_n \to_d X$ to denote convergence in distribution, and $X_n \to_d F$ when $X$ follows the distribution function $F$. The notation $\ind{A}$ denotes the indicator function of a set $A$. Let $\nabla$ and $\nabla^2$ be the gradient and the Hessian operator, respectively.

\section{Preliminaries}\label{sec:prelims}

\subsection{Generalized Alpha-Beta (GAB) Divergence Family}

Inspired by the pioneering work on DPD by~\cite{basu1998robust}, multiple attempts have been made to produce newer classes of statistical divergences aimed at balancing robustness and efficiency in estimation. Among these, some prominent classes of divergences are the SD class, introduced by~\cite{ghosh2017generalized}, the family of LSDs by~\cite{maji2016logarithmic}, and the AB divergence family by~\cite{cichocki2011generalized}. In an attempt to connect all these divergences into a superclass of statistical divergence, \cite{roy2025characterization} introduced the class of generalized alpha-beta (GAB) divergences. Let $\alpha, \beta \in \R$ be two hyperparameters and $\psi: [0, \infty] \to \R$ be the generating function for this family. Then for any two densities $f$ and $g$ dominated by a common dominating measure $\mu$, the GAB divergence is defined as
\begin{equation}
    d_{\mathrm{GAB}}^{(\alpha,\beta),\psi}(f, g) = \dfrac{1}{\beta(\alpha+\beta)}\psi(\norm{f}_{\alpha+\beta}^{\alpha+\beta}) + \dfrac{1}{\alpha(\alpha+\beta)}\psi(\norm{g}_{\alpha+\beta}^{\alpha+\beta}) - \dfrac{1}{\alpha\beta} \psi(\inner{f,g}_{\alpha,\beta}),
    \label{eqn:gab-defn-1}
\end{equation}
\noindent where $\alpha, \beta$ and $(\alpha+\beta)$ are all nonzero. For the edge cases where either one of them is equal to zero, the form of the GAB divergence may be obtained by taking the corresponding limits,
\begin{align}
    d_{\mathrm{GAB}}^{(\alpha,\beta),\psi}(f, g) = \begin{cases}
    \alpha^{-2} \left( \alpha \psi'(\norm{f}_\alpha^\alpha) \int f^\alpha \log(f/g) - \psi(\norm{f}_\alpha^\alpha) + \psi(\norm{g}_\alpha^\alpha) \right) & \text{ if } \alpha \neq 0, \beta = 0, \\
    \beta^{-2} (\beta  \psi'(\norm{g}_\beta^\beta) \int g^\beta \log(g/f) - \psi(\norm{g}_\beta^\beta) + \psi(\norm{f}_\beta^\beta))                      & \text{ if } \alpha = 0, \beta \neq 0, \\
    \alpha^{-2} \left( \psi'(1) \int \log(f^\alpha/g^\alpha) + \psi(\norm{f/g}_\alpha^\alpha) - \psi(1) \right)                                           & \text{ if } \alpha = -\beta \neq 0    \\
    \psi'(1) \int (\log(f) - \log(g))^2/2       & \text{ if } \alpha = \beta = 0.
    \end{cases}
    \label{eqn:gab-defn-edge}
\end{align}
\noindent \cite{roy2025characterization} demonstrated that when $\alpha, \beta \neq 0$ and $(\alpha+\beta) \notin \{0, 1\}$, then a necessary and sufficient condition for the above form of the GAB divergence to be a valid statistical divergence is that $\psi \in C^1((0, \infty))$ and $\Psi(x) := \psi(e^x)$ is strictly increasing and convex. Even in the edge cases when either $\alpha = 0$ or $\beta = 0$ or $(\alpha + \beta) \in \{0, 1\}$, it is necessary for $\psi$ to be strictly increasing in some subset of $(0, \infty)$. Therefore, for the rest of this paper, unless otherwise specified, we shall assume that the generating function $\psi$ characterizing the GAB divergence is strictly increasing and $\Psi(x) := \psi(e^x)$ is convex.

A particularly interesting special case of the GAB divergence is produced with $\psi(x) = (x^\phi - 1)/\phi$, which leads to the form
\begin{equation}
    d_{\mathrm{GAB}}^{(\alpha,\beta, \phi)}(f, g) = \dfrac{1}{\phi\beta(\alpha+\beta)} \norm{f}_{\alpha+\beta}^{(\alpha+\beta)\phi} + \frac{1}{\phi\alpha(\alpha+\beta)} \norm{g}_{\alpha+\beta}^{(\alpha+\beta)\phi} - \frac{1}{\alpha\beta\phi} \inner{f, g}_{\alpha,\beta}^{\phi}.
    \label{eqn:jhhb-div}
\end{equation}
\noindent When $\beta = 1$, this becomes equivalent to the $(\phi, \gamma)$-divergence family of~\cite{jones2001comparison}, where $\gamma$ is a resymbolization of the hyperparameter $\alpha$. If $\phi = 1$, then it leads to the SD~\citep{ghosh2017generalized} and as $\phi \to 0+$, it becomes same as the LSD~\citep{maji2016logarithmic}.


\subsection{Minimum Generalized Alpha-Beta Divergence Estimation}\label{sec:mgabde-defn}

In a parametric inference problem, given the true distribution $G$ with density $g$ and a model family of distributions $\mathcal{F} = \{F_{\bb{\theta}}: \bb{\theta} \in \bb{\Theta} \subset \R^p \}$ with corresponding densities $f_{\bb{\theta}}$, the minimum generalized alpha-beta divergence (MGABD) functional is given by
\begin{equation}
    \bbhat{\theta}^{(\alpha,\beta),\psi}(G) := \argmin_{\bb{\theta} \in \bb{\Theta}} d_{GAB}^{(\alpha, \beta),\psi}(f_{\bb{\theta}}, g),
    \label{eqn:bp-mgabd-functional}
\end{equation}
\noindent provided such a minimum exists. For the cases with $\alpha\beta(\alpha+\beta) \neq 0$, note that the second term in the right-hand side of Eq.~\eqref{eqn:gab-defn-1} involving $\norm{g}_{\alpha+\beta}$ does not contribute to the minimization problem, hence a more refined objective function is given by
\begin{equation}
    H^{(\alpha,\beta),\psi}(\bb{\theta})
    := \dfrac{1}{\beta(\alpha+\beta)} \psi\left(\norm{f_{\bb{\theta}}}_{\alpha+\beta}^{\alpha+\beta}\right) - \frac{1}{\alpha\beta} \psi\left( \inner{ f_{\bb{\theta}}, g}_{\alpha,\beta} \right).
    \label{eqn:H-theta-popn}
\end{equation}
\noindent If $\alpha\beta(\alpha+\beta) = 0$, a similar reduction is possible for the objective function. However, in practical problems, a statistician only observes a sample $X_1, X_2, \dots, X_n$ generated according to the probability distribution $g$. Based on the values of the hyperparameters and the setup, different choices of estimation procedures are then available at the disposal of the statistician.
\begin{enumerate}
    \item If $\beta = 1$, then the term involving $(\alpha,\beta)$-inner product between $f_{\bb{\theta}}$ and $g$ can be well-approximated using a Monte-Carlo integration approach, i.e., $\int f_{\bb{\theta}}^\alpha g \approx n^{-1}\sum_{i=1}^n f_{\bb{\theta}}^\alpha(X_i)$. In this case, to obtain the minimum GAB divergence estimator (MGABDE), one may choose to minimize the empirical objective function
          \begin{equation}
              H_n^{(\alpha,1),\psi}(\bb{\theta})
              := \dfrac{1}{\alpha+1} \psi\left( \int f_{\bb{\theta}}^{\alpha+1} \right) - \frac{1}{\alpha} \psi\left( \frac{1}{n}\sum_{i=1}^n f_{\bb{\theta}}^\alpha(X_i)  \right).
              \label{eqn:H-theta-beta1-empirical}
          \end{equation}
          \noindent The population counterpart $H^{(\alpha,1),\psi}(\bb{\theta})$ is obtained by substituting $\beta = 1$ in Eq.~\eqref{eqn:H-theta-popn}.
    \item If the underlying random variables are discrete, say with a common support $\chi$, then one may define $r_n(x) = n^{-1}\sum_{i=1}^n \bb{1}(X_i = x)$ for each $x \in \chi$, the relative frequency of $x$. Then, MGABDE can be obtained by minimizing the objective function
          \begin{equation}
              H_{n,disc}^{(\alpha,\beta),\psi}(\bb{\theta})
              := \dfrac{1}{\beta(\alpha+\beta)} \psi\left( \sum_{x \in \chi} f_{\bb{\theta}}^{\alpha+\beta}(x) \right) - \frac{1}{\alpha\beta} \psi\left( \sum_{x \in \chi} f_{\bb{\theta}}^\alpha(x) r_n^\beta(x)  \right).
              \label{eqn:H-theta-discrete-empirical}
          \end{equation}
          \noindent where the subscript $disc$ denotes the discrete setup.
    \item For the most general case when the underlying random variables are continuous, and $\beta \neq 1$, there is an immediate challenge in estimation. Since the data are discrete but the model is continuous, this creates difficulty in accurately representing the $(\alpha,\beta)$-inner product between $f_\theta$ and $g$; see~\cite{ghosh2017minimum} for details. The Basu-Lindsay approach~\citep{basu1994minimum} suggests that a nonparametric estimator $\hat{g}_n$ needs to be constructed first, possibly as $\hat{g}_n(x) = n^{-1}\sum_{i=1}^n W(x, X_i, b_n)$, where $W(x, X_i, b_n)$ is an appropriate kernel function with a sequence of bandwidths $b_n$ that decays with sample size at an appropriate rate. Next, a kernel-integrated ``smoothed'' model is generated as $\tilde{f}_{\bb{\theta}}(x) = \int W(x, y, b_n) f_{\bb{\theta}}(y)dy$. Finally, the MGABDE is defined as the minimizer of GAB divergence between $\tilde{f}_{\bb{\theta}}$ and $\hat{g}_n$ instead. Therefore, the objective function becomes
          \begin{equation}
              \tilde{H}_{n}^{(\alpha,\beta),\psi}(\bb{\theta})
              := \dfrac{1}{\beta(\alpha+\beta)} \psi\left( \int \tilde{f}_{\bb{\theta}}^{\alpha+\beta} \right) - \frac{1}{\alpha\beta} \psi\left( \int \tilde{f}_{\bb{\theta}}^\alpha \hat{g}_n^\beta  \right).
              \label{eqn:H-theta-BL-empirical}
          \end{equation}
\end{enumerate}
\noindent Eq.~\eqref{eqn:H-theta-beta1-empirical}-\eqref{eqn:H-theta-BL-empirical} are valid objective functions only when $\alpha\beta(\alpha+\beta) \neq 0$. If either $\alpha = 0$ or $\beta = 0$ but not both, the corresponding objective functions for the discrete and the continuous cases can be obtained by taking corresponding limits $\alpha \to 0+$ or $\beta \to 0+$, and substituting appropriate expressions as in Eq.~\eqref{eqn:gab-defn-edge}.

\begin{remark}\label{remark:affine-psi}
    A key invariance property of the MGABDE is that any affine transformation of the $\psi$-function leads to the same estimator, i.e., if $\psi_1(x) = a\psi(x) + b$ for some $a > 0$ and $b \in \R$, then $d_{GAB}^{(\alpha,\beta),\psi_1}(f,g) = a d_{GAB}^{(\alpha,\beta),\psi}$ and hence $\bbhat{\theta}^{(\alpha,\beta),\psi_1}(G) = \bbhat{\theta}^{(\alpha,\beta),\psi}(G)$. This suggests that without the loss of any generality, we may assume $\psi(1) = 0$ and $\psi'(1) = 1$.
\end{remark}

\section{Asymptotic Normality of MGABDE}\label{sec:asymp-normality}

Let $\bb{\theta}^g$ be the true minimizer of $H^{(\alpha,\beta),\psi}(\theta)$ given in Eq.~\eqref{eqn:H-theta-popn}. Then, in line with the works on the minimum S-divergence estimator (MSDE)~\citep{ghosh2017generalized} and minimum logarithmic S-divergence estimator (MLSDE)~\citep{maji2016logarithmic}, one expects that for any ``nice'' function $\psi$, the corresponding MGABDE $\bbhat{\theta}^{(\alpha,\beta),\psi}_n$ should be $\sqrt{n}$-consistent and asymptotically normally distributed as the sample size $n \to \infty$. This is indeed true. But because of different approaches to obtain the MGABDE by minimizing slightly different objective functions as described in Section~\ref{sec:mgabde-defn}, we shall present a generic result first. This generic result translates the asymptotic properties of the objective function into the desired asymptotic properties of the corresponding MGABDE. Before formally stating the result, it is worthwhile to mention the standard regularity conditions on the model family $f_{\bb{\theta}}$ which are typically assumed for various theoretical results concerning minimum divergence estimators, see~\cite{basu1998robust, Lehmann2006tpe} for relevant discussions on these assumptions.

\begin{enumerate}[label = (A\arabic*), ref = (A\arabic*)]
    \item \label{assum:identifiable} The parametric model family $\Fcal$ is identifiable, i.e., for any two $\bb{\theta}_1, \bb{\theta}_2 \in \bb{\Theta}$, if $f_{\bb{\theta}_1} = f_{\bb{\theta}_2}$ almost everywhere, then $\bb{\theta}_1 = \bb{\theta}_2$.
    \item \label{assum:support} The densities $f_{\bb{\theta}}$ have a support $\chi \subseteq \R$ which is independent of $\bb{\theta}$. The true density $g$ also has the same support $\chi$.
    \item \label{assum:f-diff} There is an open subset $\omega$ of the parameter space $\bb{\Theta}$ such that $f_{\bb{\theta}}$ is three times differentiable with respect to $\bb{\theta}$ for every $\bb{\theta} \in \omega$, and the third derivative is continuous.
    \item \label{assum:f-integral-diff} The integrals $\int f_{\bb{\theta}}^{\alpha+\beta}$ and $\int f_{\bb{\theta}}^\alpha g^{\beta}$ are three times differentiable and the derivative can be taken under the integral sign.
    \item\label{assum:H-third-diff} For almost all $x$, there exists functions $M_{jkl}(x), M_{jk,l}(x)$ and $M_{j,k,l}(x)$ that are uniformly bounded in expectation with respect to $g$ and $f_{\bb{\theta}}$ for all $\bb{\theta} \in \bb{\Theta}$, and they satisfy the inequalities
          \begin{align*}
              \vert f_{\bb{\theta}}^{\alpha+\beta - 1}u_{jkl\bb{\theta}}(x) \vert                                       & < M_{jkl}(x),   \\
              \vert f_{\bb{\theta}}^{\alpha+\beta - 1}u_{jk\bb{\theta}}(x) u_{l\bb{\theta}}(x) \vert                    & < M_{jk,l}(x),  \\
              \vert f_{\bb{\theta}}^{\alpha+\beta - 1}u_{j\bb{\theta}}(x) u_{k\bb{\theta}}(x) u_{l\bb{\theta}}(x) \vert & < M_{j,k,l}(x),
          \end{align*}
          \noindent for all $j,k$ and $l$. Here, $u_{\bb{\theta}}(x) = \nabla_{\bb{\theta}}\log(f_{\bb{\theta}}(x))$ denotes the score function for the parametric density $f_{\bb{\theta}}$, $u_{j\bb{\theta}}(x)$ denotes the $j$-th element of $u_{\bb{\theta}}(x)$, $u_{jk\bb{\theta}}(x)$ denotes the $(j,k)$-th element of $\nabla^2_{\bb{\theta}} \log(f_{\bb{\theta}}(x))$ and $u_{jkl\bb{\theta}}(x)$ denotes $(j,k,l)$-th element of the tensor $\nabla^3_{\bb{\theta}} \log(f_{\bb{\theta}}(x))$.
\end{enumerate}
\noindent In addition to these regularity conditions, we now present the assumptions governing the asymptotic behavior of the objective function. Later in Sections~\ref{sec:normality-beta1}-\ref{sec:normality-continuous}, it will be demonstrated how these assumptions may be weakened for the specific estimation processes described in Section~\ref{sec:mgabde-defn}. Here, we make slight abuse of the notation by using $H_n^{(\alpha,\beta),\psi}(\bb{\theta})$ which may be either $H_{n}^{(\alpha,1),\psi}(\bb{\theta})$ as in Eq.~\eqref{eqn:H-theta-beta1-empirical} or $H_{n,disc}^{(\alpha,\beta),\psi}(\bb{\theta})$ as in Eq.~\eqref{eqn:H-theta-discrete-empirical} or $\tilde{H}_{n}^{(\alpha,\beta),\psi}(\bb{\theta})$ as in Eq.~\eqref{eqn:H-theta-BL-empirical}.

\begin{enumerate}[label = (A\arabic*), ref = (A\arabic*)]
    \setcounter{enumi}{5}
    \item \label{assum:H-first-diff} As $n\to \infty$, $\sqrt{n} \nabla H_n^{(\alpha,\beta),\psi}(\bb{\theta})\vert_{\bb{\theta} = \bb{\theta}^g} \to_d \bb{Z}$ where $\bb{Z}$ follows $p$-variate normal distribution with mean $\bb{0}_p$ and variance-covariance matrix $\bb{K}(\bb{\theta}^g)$.
    \item \label{assum:H-second-diff} As $n\to \infty$, $\nabla^2 H_n^{(\alpha,\beta),\psi}(\bb{\theta})\vert_{\bb{\theta} = \bb{\theta}^g} \to_{\prob} \bb{J}(\bb{\theta}^g)$, a $p\times p$ positive-definite matrix.
\end{enumerate}

\begin{proposition}\label{prop:generic-normality}
    Let $\psi \in C^2((0, \infty))$ and Assumptions~\ref{assum:identifiable}-\ref{assum:H-second-diff} hold. Then, there exists a sequence of MGABDE $\bbhat{\theta}_n^{(\alpha, \beta), \psi}$ which are solutions to the estimating equations $\nabla H_n^{(\alpha,\beta),\psi}(\bb{\theta}) = \bb{0}_p$ such that as $n \to \infty$,
    \begin{enumerate}
        \item $\bbhat{\theta}_n^{(\alpha, \beta), \psi}$ is consistent for $\bb{\theta}^g$.
        \item $n^{1/2} \bb{K}({\bb{\theta}^g})^{-1/2} \bb{J}(\bb{\theta}^g)  (\bbhat{\theta}_n^{(\alpha, \beta), \psi} - \bb{\theta}^g)$ converges in distribution to a $p$-dimensional standard normal random variable.
    \end{enumerate}
\end{proposition}

As Proposition~\ref{prop:generic-normality} shows, the key elements of the asymptotic variance of the MGABDE are the matrices $\bb{K}(\bb{\theta}^g)$ and $\bb{J}(\bb{\theta^g})$. Therefore, at this point, it is helpful to introduce some notation that will pave the way to represent them in a compact manner. Let the information function be denoted by $i_{\bb{\theta}} = -\nabla u_{\bb{\theta}}$. Also, let
\begin{align}
    P_{\alpha,\beta} = \int f_{\bb{\theta}}^\alpha g^\beta u_{\bb{\theta}}, \
    Q_{\alpha,\beta} = \int f_{\bb{\theta}}^\alpha g^\beta u_{\bb{\theta}}u_{\bb{\theta}}\tr, \
    R_{\alpha,\beta} = \int f_{\bb{\theta}}^\alpha g^\beta i_{\bb{\theta}}, \label{eqn:P-Q-R-defn} \\
    \hat{P}_{\alpha,\beta} = \int f_{\bb{\theta}}^\alpha \hat{g}_n^\beta u_{\bb{\theta}}, \
    \hat{Q}_{\alpha,\beta} = \int f_{\bb{\theta}}^\alpha \hat{g}_n^\beta u_{\bb{\theta}}u_{\bb{\theta}}\tr, \
    \hat{R}_{\alpha,\beta} = \int f_{\bb{\theta}}^\alpha \hat{g}_n^\beta i_{\bb{\theta}},
    \label{eqn:P-Q-R-hat-defn}
\end{align}
\noindent and let $\tilde{P}_{\alpha,\beta}, \tilde{Q}_{\alpha,\beta}, \tilde{R}_{\alpha,\beta}$ denote the corresponding integrals where the model density $f_{\bb{\theta}}$ is replaced by its kernel-convoluted counterpart $\tilde{f}_{\bb{\theta}}$. Only for the special case of $\beta = 1$, we can redefine
\begin{equation}
    \hat{P}_{\alpha,1}
    = \E_{n}(f_{\bb{\theta}}^\alpha(X) u_{\bb{\theta}}(X)),
    \hat{Q}_{\alpha,1} = \E_{n}(f_{\bb{\theta}}^\alpha(X) u_{\bb{\theta}}(X) u_{\bb{\theta}}(X)\tr),
    \hat{R}_{\alpha,1} = \E_{n}(f_{\bb{\theta}}^\alpha(X) i_{\bb{\theta}}(X)),
    \label{eqn:P-Q-R-hat-defn-beta1}
\end{equation}
\noindent where $\E_n$ denotes the expectation operator with respect to the empirical CDF of the observed sample $X_1, \dots, X_n$. Here, we approximate the cross-integrals involving $f_{\bb{\theta}}$ and $\hat{g}_n$ by corresponding functionals of the empirical distribution via Monte-Carlo expectation, without resorting to kernel smoothing.


These notations help us to express the following two key quantities related to the asymptotic properties of MGABDE in a compact way;
\begin{equation}
    \bb{K}^{(\alpha,\beta),\psi}(\bb{\theta}) = (\psi'(\inner{f_{\bb{\theta}}, g}_{\alpha,\beta}))^2 (Q_{2\alpha, 2\beta - 1} - P_{\alpha,\beta}P_{\alpha,\beta}\tr),
    \label{eqn:K-theta-generic}
\end{equation}
\noindent and,
\begin{multline}
    \bb{J}^{(\alpha,\beta),\psi}(\bb{\theta})
    = (\alpha + \beta)   \psi''(\norm{f_{\bb{\theta}}}_{\alpha+\beta}^{\alpha+\beta})  P_{\alpha+\beta,0}P_{\alpha+\beta,0}\tr - \alpha  \psi''( \inner{f_{\bb{\theta}}, g}_{\alpha,\beta}  ) P_{\alpha,\beta}P_{\alpha,\beta}\tr  \\
    + \psi'(\norm{f_\theta}_{\alpha+\beta}^{\alpha+\beta}) ((\alpha+\beta) Q_{\alpha+\beta, 0} - R_{\alpha+\beta, 0}) - \psi'(\inner{f_{\bb{\theta}}, g}_{\alpha,\beta} ) (\alpha Q_{\alpha,\beta} - R_{\alpha,\beta}).
    \label{eqn:J-theta-generic}
\end{multline}
\noindent These $p \times p$ matrices help characterize the expression of asymptotic variance and influence function of MGABDE. Similar expressions are ubiquitous in the derivation of asymptotic properties of various minimum divergence estimators; see~\cite{basu2011statistical} for several related results. When the true density $g$ belongs to the model family $\Fcal$, we have the best-fitting parameter $\bb{\theta}^g$ satisfying $g = f_{\bb{\theta}^g}$. In this case, the above expressions for $\bb{J}$ and $\bb{K}$ matrix for $\bb{\theta} = \bb{\theta}^g$ simplifies to
\begin{align}
    \bb{K}^{(\alpha,\beta),\psi}(\bb{\theta}^g)
    & = (\psi'(\norm{g}_{\alpha+\beta}^{\alpha+\beta}))^2 (Q_{0,  2 (\alpha+\beta) - 1} - P_{0, \alpha + \beta}P_{0, \alpha +\beta}\tr),
    \label{eqn:K-theta-generic-bestfit}          \\
    \bb{J}^{(\alpha,\beta),\psi}(\bb{\theta}^g) &
    = \beta \left[ \psi''(\norm{g}_{\alpha+\beta}^{\alpha+\beta})  P_{0, \alpha+\beta}P_{0, \alpha+\beta}\tr
    +  \psi'(\norm{g}_{\alpha+\beta}^{\alpha+\beta})  Q_{0, \alpha+\beta} \right].
    \label{eqn:J-theta-generic-bestfit}
\end{align}
\noindent The expressions of $\bb{K}^{(\alpha,\beta),\psi}(\bb{\theta})$ and $\bb{J}^{(\alpha,\beta),\psi}(\bb{\theta})$ as given in Eq.~\eqref{eqn:K-theta-generic}-\eqref{eqn:J-theta-generic} are valid only if $\beta \neq 0$. If $\beta = 0$, one must consider appropriate limits as $\beta \to 0+$, which produces the modified expressions of these matrices as
\begin{multline}
    \bb{K}^{(\alpha,0),\psi}(\bb{\theta})
    = \left( \psi''(\norm{f_{\bb{\theta}}}_\alpha^\alpha ) \right)^2 \left( \inner{f_{\bb{\theta}}, g}_{2\alpha,-1} - \norm{f_{\bb{\theta}}}_\alpha^{2\alpha}  \right) P_{\alpha,0}P_{\alpha, 0}\tr
    + \left( \psi'(\norm{f_{\bb{\theta}}}_\alpha^\alpha) \right)^2 \tilde{Q}_\alpha \\
    + 2 \psi''(\norm{f_{\bb{\theta}}}_\alpha^\alpha) \psi'(\norm{f_{\bb{\theta}}}_\alpha^\alpha) (P_{2\alpha,-1} - P_{\alpha,0}\norm{f_{\bb{\theta}}}_\alpha^\alpha ) P_{\alpha, 0}\tr,
    \label{eqn:K-theta-beta0}
\end{multline}
\noindent where $\tilde{Q}_{\alpha} = (Q_{2\alpha, -1} - P_{\alpha,0}P_{\alpha,0}\tr)$, and,
\begin{multline}
    \bb{J}^{(\alpha,0),\psi}(\bb{\theta})
    = \psi'''(\norm{f_{\bb{\theta}}}_\alpha^\alpha) P_{\alpha,0}P_{\alpha,0}\tr \left( \int f_{\bb{\theta}}^\alpha \log(f_{\bb{\theta}}^\alpha / g^\alpha) - 1 \right)\\
    + (Q_{\alpha,0} - \alpha^{-1} R_{\alpha,0}) \left[ \psi''(\norm{f_{\bb{\theta}}}_\alpha^\alpha) \left( \int f_{\bb{\theta}}^\alpha \log(f_{\bb{\theta}}^\alpha / g^\alpha) - 1 \right) + \psi'(\norm{f_{\bb{\theta}}}_\alpha^\alpha) \right].
    \label{eqn:J-theta-beta0}
\end{multline}
\noindent For the special case when the true density $g$ belongs to the model family $\Fcal$ such that $g = f_{\bb{\theta}^g}$, the $\bb{J}$-matrix simplifies greatly, reducing to
\begin{multline}
    \bb{J}^{(\alpha,0),\psi}(\bb{\theta}^g)
    = -\psi'''(\norm{f_{\bb{\theta}}}_\alpha^\alpha) P_{\alpha,0}P_{\alpha,0}\tr \\
    + \left(Q_{\alpha,0} - \alpha^{-1} R_{\alpha,0} \right) \left( \psi'(\norm{f_{\bb{\theta}}}_\alpha^\alpha)  - \psi''(\norm{f_{\bb{\theta}}}_\alpha^\alpha) \right).
    \label{eqn:J-theta-beta0-bestfit}
\end{multline}

\subsection{Asymptotic Normality of MGABDE for $\beta = 1$}\label{sec:normality-beta1}

As illustrated before through Eq.~\eqref{eqn:H-theta-beta1-empirical}, when $\beta = 1$, the MGABDE can be computed without the need to obtain an estimate of the data density $g$, and thus, the estimation process remains free of major nuisances like the choice of the nonparametric estimate $\hat{g}$, bandwidth parameter selection, etc. As a result, one obtains the result of the asymptotic behavior of such an estimator quite easily through an application of Proposition~\ref{prop:generic-normality} and standard classical limit theorems. It turns out that, Assumptions~\ref{assum:H-first-diff}-\ref{assum:H-second-diff} may be replaced by a single weaker assumption instead, as detailed below.
\begin{enumerate}[label = (A\arabic*a), ref = (A\arabic*a)]
    \setcounter{enumi}{5}
    \item\label{assum:beta1-J-pd} The matrix $\bb{J}^{(\alpha,1),\psi}(\bb{\theta^g})$ as given in Eq.~\eqref{eqn:J-theta-generic} is positive definite.
\end{enumerate}

\begin{corollary}\label{thm:normality-beta-1}
    Let $\psi \in C^2((0, \infty))$, and Assumptions~\ref{assum:identifiable}-\ref{assum:H-third-diff} and \ref{assum:beta1-J-pd} hold. Then, there exists a sequence of MGABDE $\bbhat{\theta}_n^{(\alpha,1),\psi}$ which are solutions to the estimating equations $\nabla H_n^{(\alpha,1),\psi}(\bb{\theta}) = \bb{0}_p$ such that as $n \to \infty$,
    \begin{enumerate}
        \item $\bbhat{\theta}_n^{(\alpha,1),\psi}$ is consistent for $\bb{\theta}^g$.
        \item $n^{1/2} (\bb{K}^{(\alpha,1),\psi}({\bb{\theta}^g}))^{-1/2} \bb{J}^{(\alpha, 1),\psi}(\bb{\theta}^g)  (\bbhat{\theta}_n^{(\alpha, 1), \psi} - \bb{\theta}^g)$ converges in distribution to a $p$-dimensional standard normal random variable, where $\bb{K}^{(\alpha,1),\psi}({\bb{\theta}^g})$ and $\bb{J}^{(\alpha,1),\psi}(\bb{\theta}^g)$ are as given in Eq.~\eqref{eqn:K-theta-generic} and~\eqref{eqn:J-theta-generic} but evaluated at $\beta = 1$ and $\bb{\theta} = \bb{\theta}^g$.
    \end{enumerate}
\end{corollary}

Corollary~\ref{thm:normality-beta-1} generalizes several results existing in the minimum divergence literature. For instance, Theorem 2.2 of~\cite{basu1998robust} for DPD ($\psi(x) = x$), Section 3.2 of~\cite{jones2001comparison} for $(\phi, \gamma)$-divergence family ($\psi(x) = \phi^{-1}x^\phi$), Theorem 4 of~\cite{broniatowski2012decomposable} for LDPD ($\psi(x) = \log(x)$), Theorems 3.1 and 3.2 of~\cite{kuchibhotla2019statistical} for bridge divergences ($\psi(x) = \log(\lambda + (1-\lambda)x)$), are special cases of Corollary~\ref{thm:normality-beta-1}.

\subsection{Asymptotic Normality under the Discrete Setup}\label{sec:normality-discrete}

When the underlying random variable $X$ is discrete, one can make use of the relative frequencies $r_n(x) = n^{-1}\sum_{i=1}^n \bb{1}(X_i = x)$ to obtain the MGABDE as illustrated through Eq.~\eqref{eqn:H-theta-discrete-empirical}. Moving along the lines of work by~\cite{lindsay1994efficiency} and~\cite{basu1994minimum}, we shift our focus to represent the objective function through the lens of $\delta_n(x) = r_n(x) / f_{\bb{\theta}}(x)$. This scaled frequency serves as an empirical approximation of $\delta_g(x) = g(x) / f_{\bb{\theta}}(x)$. Let us now express the gradient of the empirical objective function through these scaled frequencies as
\begin{equation}
    \nabla H_{n,disc}^{(\alpha,\beta),\psi}(\bb{\theta})
    = \frac{1}{\beta}\psi'(  \norm{f_{\bb{\theta}}}_{\alpha+\beta}^{\alpha+\beta} ) P_{\alpha+\beta,0} \\
    - \frac{1}{\beta} \psi'( \inner{  f_{\bb{\theta}},  \delta_n}_{\alpha+\beta,\beta}) \hat{P}_{\alpha,\beta}
    \label{eqn:discrete-H-diff}
\end{equation}
\noindent Here, $\hat{P}_{\alpha,\beta}$ is as given in Eq.~\eqref{eqn:P-Q-R-hat-defn}. Note that the first term is free of the data $X_1, \dots, X_n$, and hence we will only focus on the second term to obtain its asymptotic distribution. Under certain regularity conditions, an extension of Lemma 4 of~\cite{ghosh2015asymptotic} established that $\sqrt{n}\hat{P}_{\alpha,\beta}$ asymptotically follows a normal distribution. Furthermore, by a similar strategy, it can be shown that $\inner{f_{\bb{\theta}}, \delta_n}_{\alpha+\beta, \beta}$ converges in probability to $\inner{f_{\bb{\theta}}, \delta_g}_{\alpha+\beta, \beta}$ as $n \to \infty$ for $\beta \neq 0$. If $\psi'$ is continuous, then one can appeal to the continuous mapping theorem to establish the asymptotic distribution of $\sqrt{n}H_{n}^{(\alpha,\beta),\psi}(\bb{\theta})$ as needed by Assumption~\ref{assum:H-first-diff}. Assumption~\ref{assum:H-second-diff} can be established similarly. Thus, in view of Proposition~\ref{prop:generic-normality}, one obtains the asymptotic distribution of the MGABDE $\bbhat{\theta}^{(\alpha,\beta),\psi}$ under the discrete setup. However, before we make these intuitions precise through a formal statement, it is beneficial to clarify the assumptions underlying the result. These assumptions correspond exactly to the assumptions (SA4)-(SA7) of~\cite{ghosh2015asymptotic}.

\begin{enumerate}[label = (DA\arabic*), ref = (DA\arabic*)]
    \item\label{assum:discrete-J-pd} The matrix $\bb{J}^{(\alpha,\beta),\psi}(\bb{\theta^g})$ as given in Eq.~\eqref{eqn:J-theta-generic} is positive definite.
    \item\label{assum:discrete-hellinger-bound} For any $\bb{\theta} \in \bb{\Theta}$ and for all $j, k = 1, \dots, p$, the terms
          \begin{align*}
               & \sum_{x} g^{1/2}(x) f_{\bb{\theta}}^{\alpha + \beta - 1}(x), \
              \sum_{x} g^{1/2}(x) f_{\bb{\theta}}^{\alpha + \beta - 1}(x) \vert u_{j\bb{\theta}}(x)\vert,                   \\
               & \sum_{x} g^{1/2}(x) f_{\bb{\theta}}^{\alpha + \beta - 1}(x) \vert u_{jk\bb{\theta}}(x)\vert, \text{ and, }
              \sum_{x} g^{1/2}(x) f_{\bb{\theta}}^{\alpha + \beta - 1}(x) \vert u_{j\bb{\theta}}(x)\vert \vert u_{k\bb{\theta}}(x)\vert,
          \end{align*}
          \noindent are summable and bounded.
    \item\label{assum:discrete-ratio-bound} $\delta_g^{(\beta - 1)}(x)$ is uniformly bounded for all $\bb{\theta} \in \bb{\Theta}$.
\end{enumerate}

\begin{corollary}\label{thm:normality-discrete}
    Let $\psi \in C^2((0, \infty))$ and $\beta \notin \{0, 1\}$, and the model family $\Fcal$ consists of only discrete distributions with probability mass functions given by $f_{\bb{\theta}}$ for $\bb{\theta} \in \bb{\Theta}$. Additionally, we grant Assumptions~\ref{assum:identifiable}-\ref{assum:H-third-diff} and \ref{assum:discrete-J-pd}-\ref{assum:discrete-ratio-bound}. Then, there exists a sequence of MGABDE $\bbhat{\theta}_n^{(\alpha,\beta),\psi}$ which are solutions to $\nabla H_{n,disc}^{(\alpha,\beta),\psi}(\bb{\theta}) = \bb{0}_p$ such that as $n \to \infty$,
    \begin{enumerate}
        \item $\bbhat{\theta}_n^{(\alpha,\beta),\psi}$ is consistent for $\bb{\theta}^g$, the best fitting parameter.
        \item $n^{1/2} \beta^{-1} \bb{K}^{(\alpha,\beta),\psi}({\bb{\theta}^g})^{-1/2} \bb{J}^{(\alpha,\beta),\psi}(\bb{\theta}^g)  (\bbhat{\theta}_n^{(\alpha, \beta), \psi} - \bb{\theta}^g)$ converges in distribution to a $p$-dimensional standard normal random variable, where $\bb{K}^{(\alpha,\beta),\psi}({\bb{\theta}^g})$ and $\bb{J}^{(\alpha,\beta),\psi}(\bb{\theta}^g)$ are as given in Eq.~\eqref{eqn:K-theta-generic}-\eqref{eqn:J-theta-generic} but evaluated at $\bb{\theta} = \bb{\theta}^g$.
    \end{enumerate}
\end{corollary}

When $\beta = 0$, we require slight modifications of these assumptions.

\begin{enumerate}[label = (DA\arabic*a), ref = (DA\arabic*a)]
    \item\label{assum:discrete-J-pd-beta0} The matrix $\bb{J}^{(\alpha,0),\psi}(\bb{\theta^g})$ as given in Eq.~\eqref{eqn:J-theta-beta0} is positive definite.
    \item\label{assum:discrete-hellinger-bound-beta0} In addition to Assumption~\ref{assum:discrete-hellinger-bound}, the provided series are still absolutely summable when an additional logarithmic factor, such that $\log(f_{\bb{\theta}}(x))$ or $\log(g(x))$ is multiplied to each summand.
\end{enumerate}

\begin{corollary}\label{thm:normality-discrete-beta0}
    Let $\psi \in C^3((0, \infty))$, $\beta = 0$, and consider a setup similar to Corollary~\ref{thm:normality-discrete}. Under Assumptions~\ref{assum:identifiable}-\ref{assum:H-third-diff} and \ref{assum:discrete-J-pd-beta0}, \ref{assum:discrete-hellinger-bound-beta0}, \ref{assum:discrete-ratio-bound}, there exists a sequence of MGABDE $\bbhat{\theta}_n^{(\alpha,0),\psi}$ which are solutions to $\nabla H_{n,disc}^{(\alpha,0),\psi}(\bb{\theta}) = \bb{0}_p$ such that as $n \to \infty$,
    \begin{enumerate}
        \item $\bbhat{\theta}_n^{(\alpha,\beta),\psi}$ is consistent for $\bb{\theta}^g$, the best fitting parameter.
        \item $n^{1/2} \bb{K}^{(\alpha,0),\psi}({\bb{\theta}^g})^{-1/2} \bb{J}^{(\alpha,0),\psi}(\bb{\theta}^g)  (\bbhat{\theta}_n^{(\alpha, 0), \psi} - \bb{\theta}^g)$ converges in distribution to a $p$-dimensional standard normal random variable where $\bb{K}^{(\alpha,0),\psi}(\bb{\theta}^g)$ and $\bb{J}^{(\alpha,0),\psi}(\bb{\theta}^g)$ is as given in Eq.~\eqref{eqn:K-theta-beta0}-\eqref{eqn:J-theta-beta0} evaluated at $\bb{\theta} = \bb{\theta}^g$.
    \end{enumerate}
\end{corollary}

\subsection{Asymptotic Normality under Continuous Models: the Basu-Lindsay Approach}\label{sec:normality-continuous}

When $\beta \neq 1$ and one is dealing with a model family $\Fcal$ consisting of continuous distributions for the underlying random variable $X$, then one must use a nonparametric estimate $\hat{g}_n$ as a plug-in to minimize the empirical objective function arising from the Generalized Alpha-Beta divergence. One interesting proposal by~\cite{basu1994minimum} is to consider a modified objective function $\tilde{H}_n^{(\alpha,\beta),\psi}(\bb{\theta})$ as in Eq.~\eqref{eqn:H-theta-BL-empirical}. For the convenience of the reader, we restate the notations: $W(x, y, h)$ is a kernel function with bandwidth parameter $h$ such that $W(x,y,h) < M(h)$ for some function $M(h)$, and for almost all $y$, the expectation $\E(W^2(X, y, h))$ exists and is finite under $X \sim g$ or $X \sim f_{\bb{\theta}}$ for any $\bb{\theta} \in \Theta$. The standard choices of kernel functions (e.g., Gaussian kernel, Epanechnikov kernel, etc.) satisfy these boundedness constraints. We take the nonparametric estimate of $g$ as
\begin{equation*}
    \hat{g}_n(y) = \frac{1}{n}\sum_{i=1}^n W(X_i, y, h),
\end{equation*}
\noindent and let us denote
\begin{equation}
    \tilde{f}_{\bb{\theta}}(y) = \int W(x, y, h) f_{\bb{\theta}}(y) dy, \text{ and, }
    \tilde{g}(y) = \int W(x, y, h)g(y)dy.
    \label{eqn:kernel-convoluted-model}
\end{equation}
\noindent Adapting the notations given in Eq.~\eqref{eqn:P-Q-R-defn}, we also define
\begin{equation}
    \tilde{P}_{\alpha,\beta} = \int \tilde{f}_{\bb{\theta}}^\alpha \tilde{g}^\beta \tilde{u}_{\bb{\theta}}, \
    \tilde{Q}_{\alpha,\beta} = \int \tilde{f}_{\bb{\theta}}^\alpha \tilde{g}^\beta \tilde{u}_{\bb{\theta}} \tilde{u}_{\bb{\theta}}\tr, \
    \tilde{R}_{\alpha,\beta} = \int \tilde{f}_{\bb{\theta}}^\alpha \tilde{g}^\beta \tilde{i}_{\bb{\theta}}, \label{eqn:P-Q-R-defn-continuous}
\end{equation}
\noindent and their empirical counterparts by $\hat{\tilde{P}}_{\alpha,\beta}, \hat{\tilde{Q}}_{\alpha,\beta}$ and $\hat{\tilde{R}}_{\alpha,\beta}$ where the true convoluted density $\tilde{g}$ is replaced by the nonparametric estimate $\hat{g}_n$. As before, we first present the simplifying assumptions for this specific case, before diving deep into the technical details of the asymptotic normality of the MGABDE.

\begin{enumerate}[label = (CA\arabic*), ref = (CA\arabic*)]
    \item\label{assum:continuous-J-pd} The matrix $\bbtilde{J}^{(\alpha,\beta),\psi}(\bb{\theta^g})$ is positive definite. Here, $\bbtilde{J}^{(\alpha,\beta),\psi}(\bb{\theta^g})$ denotes a $p \times p$ matrix with the same expression as $\bb{J}^{(\alpha,\beta),\psi}(\bb{\theta^g})$ as in Eq.~\eqref{eqn:J-theta-generic} but with $f_{\bb{\theta}}$ replaced by $\tilde{f}_{\bb{\theta}}$ and $g$ replaced by $\tilde{g}$.
    \item\label{assum:continuous-hellinger-bound} The cross-integral terms
          $\int \hat{g}_n^{1/2} \tilde{f}_{\bb{\theta}}^{\alpha + \beta - 1}$, $\int \hat{g}_n^{1/2} \tilde{f}_{\bb{\theta}}^{\alpha + \beta - 1} \vert \tilde{u}_{j\bb{\theta}}\vert$, $\int \hat{g}_n^{1/2}\tilde{f}_{\bb{\theta}}^{\alpha + \beta - 1} \vert \tilde{u}_{jk\bb{\theta}}\vert$, and $\int \hat{g}_n^{1/2} \tilde{f}_{\bb{\theta}}^{\alpha + \beta - 1} \vert \tilde{u}_{j\bb{\theta}}\vert \vert \tilde{u}_{k\bb{\theta}}\vert$ are bounded for all $j, k = 1, \dots, p$ and for any $\bb{\theta} \in \bb{\Theta}$. Here, $\tilde{u}_{j\bb{\theta}}(x)$ denotes the $j$-th element of $\tilde{u}_{\bb{\theta}}(x) := \nabla_{\bb{\theta}} \log(\tilde{f}_{\bb{\theta}}(x))$ and $\tilde{u}_{jk\bb{\theta}}(x)$ denotes the $(j,k)$-th element of $\nabla^2_{\bb{\theta}} \log(\tilde{f}_{\bb{\theta}}(x))$.
    \item\label{assum:continuous-ratio-bound} $\tilde{g}^{\beta - 1}(x) / \tilde{f}_{\bb{\theta}}^{\beta - 1}(x)$ is uniformly bounded for all $\bb{\theta} \in \bb{\Theta}$.
\end{enumerate}

\begin{corollary}\label{thm:normality-continuous}
    Let $\psi \in C^2((0, \infty))$ and $\beta \notin \{0, 1\}$, and the model family $\Fcal$ consists of continuous distributions with density functions $f_{\bb{\theta}}$ for $\bb{\theta} \in \bb{\Theta}$. Suppose that the Assumptions~\ref{assum:identifiable}-\ref{assum:H-third-diff} are true when $g$ and $f_{\bb{\theta}}$ are replaced by $\tilde{g}$ and $\tilde{f}_{\bb{\theta}}$ respectively. Moreover, assume that Assumptions~\ref{assum:continuous-J-pd}-\ref{assum:continuous-ratio-bound} hold. Then, there exists a sequence of MGABDE $\bbhat{\theta}_n^{(\alpha,\beta),\psi}$ which are solutions to the estimating equations $\nabla \tilde{H}_{n}^{(\alpha,\beta),\psi}(\bb{\theta}) = \bb{0}_p$ such that as $n \to \infty$,
    \begin{enumerate}
        \item $\bbhat{\theta}_n^{(\alpha,\beta),\psi}$ is consistent for $\bb{\theta}^g$, the best fitting parameter.
        \item $n^{1/2} \beta^{-1} (\bbtilde{K}^{(\alpha,\beta),\psi}({\bb{\theta}^g}))^{-1/2} \bbtilde{J}^{(\alpha,\beta),\psi}(\bb{\theta}^g)  (\bbhat{\theta}_n^{(\alpha, \beta), \psi} - \bb{\theta}^g)$ converges in distribution to a $p$-dimensional standard normal random variable. Here,
              \begin{equation*}
                  \bbtilde{K}^{(\alpha,\beta),\psi}(\bb{\theta})
                  = \left( \psi'\left( \langle\tilde{f}_{\bb{\theta}}, \tilde{g}\rangle_{\alpha,\beta}\right) \right)^2 \left( \int \tilde{f}_{\bb{\theta}}^{2\alpha}(y)\tilde{g}^{2\beta - 2}(y) \tilde{u}_{\bb{\theta}}(y) \tilde{u}_{\bb{\theta}}\tr (y) \nu(y) dy - \tilde{P}_{\alpha,\beta} \tilde{P}_{\alpha,\beta}\tr \right),
              \end{equation*}
              \noindent where $\nu(y) = \E_g(W(X, y, h)) = \int W^2(x,y,h) g(x)dx$ and $\tilde{P}_{\alpha,\beta}$ are as given in Eq.~\eqref{eqn:P-Q-R-defn-continuous}, and, the $p \times p$ matrix $\bbtilde{J}^{(\alpha,\beta),\psi}(\bb{\theta}^g)$ is given by the same expression as $\bb{J}^{(\alpha,\beta),\psi}(\bb{\theta}^g)$ in Eq.~\eqref{eqn:J-theta-generic} but with $g$ replaced by $\tilde{g}$ and ${f}_{\bb{\theta}}$ replaced by $\tilde{f}_{\bb{\theta}}$ and evaluated at $\bb{\theta} = \bb{\theta}^g$.
    \end{enumerate}
\end{corollary}
\noindent As done previously, Assumptions~\ref{assum:continuous-J-pd} and~\ref{assum:continuous-hellinger-bound} should be modified for the case $\beta = 0$ as follows.

\begin{enumerate}[label = (CA\arabic*a), ref = (CA\arabic*a)]
    \item\label{assum:continuous-J-pd-beta0} The matrix $\bbtilde{J}^{(\alpha,0),\psi}(\bb{\theta^g})$ is positive definite. Here, $\bbtilde{J}^{(\alpha,0),\psi}(\bb{\theta^g})$ denotes a $p \times p$ matrix with the same expression as $\bb{J}^{(\alpha,0),\psi}(\bb{\theta^g})$ as in Eq.~\eqref{eqn:J-theta-beta0} but with $f_{\bb{\theta}}$ replaced by $\tilde{f}_{\bb{\theta}}$ and $g$ replaced by $\tilde{g}$.
    \item\label{assum:continuous-hellinger-bound-beta0} In addition to Assumption~\ref{assum:continuous-hellinger-bound}, the integrals remain bounded even when the integrands are multiplied by a logarithmic factor such as $\log(\tilde{f}_{\bb{\theta}})$ and $\log(\tilde{g})$.
\end{enumerate}

\begin{corollary}\label{thm:normality-continuous-beta0}
    Let $\psi \in C^3((0, \infty))$, $\beta = 0$, and consider a setup similar to that of Corollary~\ref{thm:normality-continuous}. Additionally, suppose that Assumptions~\ref{assum:identifiable}-\ref{assum:H-third-diff} are true when $g$ and $f_{\bb{\theta}}$ are replaced by $\tilde{g}$ and $\tilde{f}_{\bb{\theta}}$ respectively, and also Assumptions~\ref{assum:continuous-J-pd-beta0}, \ref{assum:continuous-hellinger-bound-beta0} and \ref{assum:continuous-ratio-bound} hold. Then, there exists a sequence of MGABDE $\bbhat{\theta}_n^{(\alpha,0),\psi}$ which are solutions to the estimating equations $\nabla \tilde{H}_{n}^{(\alpha,0),\psi}(\bb{\theta}) = \bb{0}_p$ such that as $n \to \infty$,
    \begin{enumerate}
        \item $\bbhat{\theta}_n^{(\alpha,0),\psi}$ is consistent for $\bb{\theta}^g$, the best fitting parameter.
        \item $n^{1/2} (\bbtilde{K}^{(\alpha,0),\psi}({\bb{\theta}^g}))^{-1/2} \bbtilde{J}^{(\alpha,0),\psi}(\bb{\theta}^g)  (\bbhat{\theta}_n^{(\alpha,0), \psi} - \bb{\theta}^g)$ converges in distribution to a $p$-dimensional standard normal random variable. Here, $\bbtilde{K}^{(\alpha,0),\psi}(\bb{\theta}^g)$ and $\bbtilde{J}^{(\alpha,0),\psi}(\bb{\theta}^g)$ are as given in Eq.~\eqref{eqn:K-theta-beta0} and~\eqref{eqn:J-theta-beta0} evaluated at the best fitting parameter $\bb{\theta}^g$ with the exceptions that $g$ is replaced by $\tilde{g}$, ${f}_{\bb{\theta}}$ is replaced by $\tilde{f}_{\bb{\theta}}$ and $Q_{2\alpha, -1}$ is replaced by
              \begin{equation*}
                  \tilde{Q}_{2\alpha,-1} = \int \tilde{f}_{\bb{\theta}}^{2\alpha}(y)\tilde{g}^{-2} \tilde{u}_{\bb{\theta}}(y) \tilde{u}_{\bb{\theta}}\tr (y) \nu(y) dy,
              \end{equation*}
              \noindent where $\nu(y) = \E_g(W^2(X, y, h)) = \int W^2(x,y,h) g(x)dx$.
    \end{enumerate}
\end{corollary}


\section{Influence Function Analysis}\label{sec:influence-function}

As a local measure of robustness, the influence function (or influence curve) has found a broad reach in the robust statistics community; see~\cite{hampel1974influence, hampel2005robust} and the references therein. Given Huber's classical contamination model
\begin{equation}
    G_{\epsilon, m} =  (1 - \epsilon) G + \epsilon K_m,
    \label{eqn:contam-model}
\end{equation}
\noindent where $G$ is the true distribution and $K_m$ is the contaminating distribution, if the MGABD functional is denoted by $\bbhat{\theta}^{(\alpha,\beta),\psi}$ as in~\eqref{eqn:bp-mgabd-functional}, then its influence function is defined as the limit
\begin{equation*}
    \IF(y; G, \bbhat{\theta}^{(\alpha,\beta),\psi}) :=  \lim_{\epsilon \to 0+} \dfrac{\bbhat{\theta}^{(\alpha,\beta),\psi}((1-\epsilon)G + \epsilon \Delta_y ) - \bbhat{\theta}^{(\alpha,\beta),\psi}(G) }{\epsilon},
\end{equation*}
\noindent where $\Delta_y$ denotes the degenerate distribution at $y \in \R$. Influence function has found its usage even beyond robust statistics, for instance, in semiparametric inference~\citep{bickel1993efficient}, in double machine learning and causal inference~\citep{chernozhukov2018doubleml}, in explainable artificial intelligence~\citep{koh2017understanding}, in differential privacy~\citep{dwork2009differential}, etc., as a key tool in estimating the sensitivity of estimators. We study the influence function of the MGABD functional in this section for a comprehensive understanding of this estimator.

As in Section~\ref{sec:asymp-normality}, we first present a generic result under a stronger set of assumptions, and later replace them with weaker assumptions for the three different estimation strategies.

\begin{proposition}\label{prop:influence-function}
    Suppose that the Assumptions~\ref{assum:identifiable}-\ref{assum:H-second-diff} hold. Furthermore, assume that the MGABD functional satisfying an estimating equation of the form $\nabla H^{(\alpha,\beta),\psi}(\bb{\theta}) = \bb{0}$ as in Eq.~\eqref{eqn:H-theta-popn} is Fisher consistent. Then, its influence function admits the form
    \begin{equation}
        \IF(y; G, \bbhat{\theta}^{(\alpha,\beta),\psi})
        = \beta \bb{J}^{-1}(\bb{\theta}^g) N^{(\alpha,\beta),\psi}(y; \bb{\theta}^g),
        \label{eqn:influence-function-generic}
    \end{equation}
    \noindent where
    \begin{multline}
        N^{(\alpha,\beta),\psi}(y; \bb{\theta}) = \psi''(\inner{f_{\bb{\theta}}, g}_{\alpha,\beta}) P_{\alpha,\beta} (f_{\bb{\theta}}^\alpha(y) g^{\beta - 1}(y) - \inner{f_{\bb{\theta}}, g}_{\alpha,\beta}) \\
        + \psi'(\inner{f_{\bb{\theta}}, g}_{\alpha,\beta}) (f_{\bb{\theta}}^\alpha(y) g^{\beta - 1}(y) u_{\bb{\theta}}(y) - P_{\alpha,\beta}).
        \label{eqn:N-generic}
    \end{multline}
\end{proposition}

Now, we present the specific results for three different estimation strategies described in Section~\ref{sec:mgabde-defn}. The derivation of the influence function of the MGABD functional corresponding to the cases $\beta = 1$ and the discrete random variable setup follows immediately from Proposition~\ref{prop:influence-function}. The formal results are as follows.

\begin{corollary}\label{cor:influence-function-beta1}
    Let $\beta = 1$. Then, under the same set of assumptions as Corollary~\ref{thm:normality-beta-1}, the MGABD functional satisfying an estimating equation of the form $\nabla H^{(\alpha,1),\psi}(\bb{\theta}) = 0$ as in Eq.~\eqref{eqn:H-theta-popn} has an influence function given by
    \begin{equation}
        \IF(y; G, \bbhat{\theta}^{(\alpha,1),\psi})
        = \left[ \bb{J}^{(\alpha,1),\psi}(\bb{\theta}^g)\right]^{-1} N^{(\alpha,1),\psi}(y; \bb{\theta}^g)
        \label{eqn:influence-function-beta1}
    \end{equation}
    \noindent where $\bb{J}^{(\alpha,1),\psi}(\bb{\theta}^g)$ and $N^{(\alpha,1),\psi}(y; \bb{\theta}^g)$ are as given in Eq.~\eqref{eqn:J-theta-generic} and~\eqref{eqn:N-generic} respectively but with $\beta = 1$.
\end{corollary}

\begin{corollary}\label{cor:influence-function-discrete}
    Let $\beta \neq 1$ and the model family $\Fcal$ contains only discrete distributions with probability mass functions given by $f_{\bb{\theta}}$. Then, under the same set of assumptions as Corollary~\ref{thm:normality-discrete}, the MGABD functional satisfying an estimating equation of the form $\nabla H_{disc}^{(\alpha,\beta),\psi}(\bb{\theta}) = 0$ as in Eq.~\eqref{eqn:H-theta-popn} yields an influence function given by
    \begin{equation}
        \IF(y; G, \bbhat{\theta}^{(\alpha,\beta),\psi})
        = \beta \left[ \bb{J}_{disc}^{(\alpha,\beta),\psi}(\bb{\theta}^g)\right]^{-1} N_{disc}^{(\alpha,\beta),\psi}(y; \bb{\theta}^g)
        \label{eqn:influence-function-discrete}
    \end{equation}
    \noindent where $\bb{J}_{disc}^{(\alpha,\beta),\psi}(\bb{\theta}^g)$ and $N_{disc}^{(\alpha,\beta),\psi}(y; \bb{\theta}^g)$ are as given in Eq.~\eqref{eqn:J-theta-generic} and~\eqref{eqn:N-generic} respectively but with all the integrals replaced by sums. The subscript ``disc'' indices this replacement.
\end{corollary}

However, for the general case of model families $\Fcal$ with continuous distributions, one applies the Basu-Lindsay approach, and thus, need to consider the kernel-convoluted model densities $\tilde{f}_{\bb{\theta}}$ and kernel-convoluted true density $\tilde{g}$ as in Eq.~\eqref{eqn:kernel-convoluted-model}. The expressions of $\bb{J}(\bb{\theta})$ and $N^{(\alpha,\beta),\psi}(\cdot,\cdot)$ also need to be modified appropriately to take into account the effect of the kernel weights $W(x, y, h)$. The result is stated formally in the following Corollary.

\begin{corollary}\label{cor:influence-function-continuous}
    Let $\beta \neq 1$ and the model family $\Fcal$ contains continuous distributions with density functions $f_{\bb{\theta}}$ for $\bb{\theta} \in \bb{\Theta}$. Then, under the same set of assumptions as Corollary~\ref{thm:normality-continuous}, the MGABD functional satisfying an estimating equation of the form $\nabla \tilde{H}^{(\alpha,\beta),\psi}(\bb{\theta}) = 0$ (functional version of Eq.~\eqref{eqn:H-theta-discrete-empirical}) has an influence function characterized as
    \begin{equation}
        \IF(y; G, \bbhat{\theta}^{(\alpha,\beta),\psi})
        = \beta \left[ \tilde{\bb{J}}^{(\alpha,\beta),\psi}(\bb{\theta}^g)\right]^{-1} \tilde{N}^{(\alpha,\beta),\psi}(y; \bb{\theta}^g)
        \label{eqn:influence-function-cont}
    \end{equation}
    \noindent where $\bbtilde{J}^{(\alpha,\beta),\psi}(\bb{\theta}^g)$ is same as $\bb{J}^{(\alpha,\beta),\psi}$ given in Eq.~\eqref{eqn:J-theta-generic} but with $f_{\bb{\theta}}$ and $g$ replaced by their kernel-convoluted counterparts $\tilde{f}_{\bb{\theta}}$ and $\tilde{g}$ respectively. And,
    \begin{multline}
        \tilde{N}^{(\alpha,\beta),\psi}(y; \bb{\theta}) = \psi''\left(\langle{\tilde{f}_{\bb{\theta}}, \tilde{g} }\rangle_{\alpha,\beta} \right) \tilde{P}_{\alpha,\beta} \left(  \int \tilde{f}_{\bb{\theta}}^\alpha(x) \tilde{g}^{\beta - 1}(x) W(x,y,h) dx - \langle{\tilde{f}_{\bb{\theta}}, \tilde{g}}\rangle_{\alpha,\beta} \right) \\
        + \psi'\left(\langle{\tilde{f}_{\bb{\theta}}, \tilde{g}}\rangle_{\alpha,\beta}\right) \left( \int f_{\bb{\theta}}^\alpha(x) \tilde{g}^{\beta - 1}(x) \tilde{u}_{\bb{\theta}}(x) W(x,y,h)dx - \tilde{P}_{\alpha,\beta} \right).
        \label{eqn:N-continuous}
    \end{multline}
\end{corollary}

\section{Asymptotic Breakdown Point Analysis}\label{sec:bp-analysis}

\subsection{Notion of Asymptotic Breakdown Point}

Although the influence function has gained popularity as a robustness measure in the literature, its capacity to distinguish between robust and non-robust estimators is limited, as it only measures changes under local perturbation. In fact, \cite{basu1994minimum} demonstrated that all minimum power divergence estimators have the same (first-order) influence functions but differ strikingly in their robustness properties. The breakdown point, an alternative measure of robustness, is particularly useful in this scenario as it tracks the global reliability of an estimator. While working on a location estimation problem, \cite{Hodges_1967} defined the finite-sample breakdown point as the maximum proportion of contaminated observations in a sample that an estimator (or a functional) can tolerate before producing an egregiously bad estimate. Inspired by this, there have been various updates to the notion of breakdown point over the years to incorporate different setups; see~\cite{Hampel_1971, Huber_Donoho_1983}, \cite[Definition 3.1]{maronna2019robust} for some of these notions. Recently, \cite{roy2023breakdown} considered a very general notion of asymptotic breakdown point covering all of aforementioned definitions as special cases. Starting with the standard contamination model as in Eq.~\eqref{eqn:contam-model}, they define the asymptotic breakdown point of a functional $T$ estimating a parameter $\bb{\theta} \in \bb{\Theta} \subseteq \R^p$ as
\begin{equation}
    \epsilon^\ast(T) := \sup \left\{ \epsilon : \epsilon \in [0, 1/2], \text{ and, } \inf_{\bb{\theta}_\infty \in \partial\bb{\Theta}}\liminf_{m\rightarrow \infty} \Vert T(G_{\epsilon, m})-   \bb{\theta}_\infty\Vert_2 > 0, \text{ for all } \{ K_m\}_{m=1}^{\infty} \right\}.\label{eqn:bp-working-defn}
\end{equation}
\noindent Here, $\partial\bb{\Theta}$ denotes the boundary of the parameter space $\bb{\Theta}$ and $\Vert \cdot\Vert_2$ denotes the usual Euclidean norm. Typically, the boundary $\partial\bb{\Theta}$ is taken in the extended real number system. For instance, in the case of a univariate location estimator, we have $\bb{\Theta} = (-\infty, \infty)$ and $\partial\bb{\Theta} = \{ -\infty, \infty\}$, and for a univariate scale estimator $\bb{\Theta} = [0, \infty)$ with $\partial\bb{\Theta} = \{ 0, \infty\}$.

Going further, for a sequence of estimators $\{ T_n : n \geq 1\}$, its asymptotic breakdown point is defined as~\citep{roy2023breakdown}
\begin{multline}
    \epsilon^\ast := \sup\left\{ \epsilon: \epsilon \in [0, 1/2], \text{ and, }\right. \\
    \left. \inf_{\bb{\theta}_\infty \in \partial\bb{\Theta}}\liminf_{m \to \infty}\liminf_{n \to \infty} \prob_{G_{\epsilon,m}}( \Vert T_n- \bb{\theta}_\infty\Vert_2 > 0) = 1, \text{ for all } \{ K_m \}_{m=1}^\infty  \right\}.
    \label{eqn:bp-estimator-defn}
\end{multline}
\noindent It turns out that in the case of a consistent sequence of estimators, the asymptotic breakdown point $\epsilon^\ast$ defined in Eq.~\eqref{eqn:bp-estimator-defn} must be greater than or equal to its counterpart $\epsilon^\ast(T)$ defined in Eq.~\eqref{eqn:bp-working-defn}. The following Proposition establishes this fact.

\begin{proposition}\label{prop:bp-consistency}
    Let $\{ T_n := T(G_n) \}_{n \geq 1}$ be a sequence of estimators that are consistent for the corresponding functional $T(G)$ as $G_n \to G$. Then, $\epsilon^\ast \geq \epsilon^\ast(T)$.
\end{proposition}

As a direct consequence of this result and the asymptotic consistency of the MGABDE derived through Proposition~\ref{prop:generic-normality}, it becomes sufficient to study the behavior of the asymptotic breakdown point of the MGABD functional in order to understand the breakdown behavior of MGABDE. There is now a growing literature on studying the asymptotic breakdown point of a minimum divergence functional~\citep{park2004minimum, roy2023breakdown,jana2025asymptotic}.

\subsection{Asymptotic Breakdown Point Analysis for positive $\alpha$ and $\beta$}\label{sec:gab-bp-positive}

To investigate the asymptotic breakdown point of the MGABD functional, we look at how the functional $\bb{\theta}^{(\alpha,\beta),\psi}(G_{\epsilon, m})$ changes as a function of $\epsilon$ as $m$ increases to infinity. Our findings rely on some standard singularity assumptions as given below. These are similar in spirit of the assumptions considered by~\cite{ghosh2013robust} and later by~\cite{roy2023breakdown}. However, before we describe the assumptions, we remind the reader that an implicit assumption is that the true distribution $G$, the contaminating distributions $K_m$ and the model distributions $F_{\bb{\theta}}$ have densities $g$, $k_m$ and $f_{\bb{\theta}}$ respectively, with respect to a common dominating measure $\mu$.

\begin{enumerate}[label = (BP\arabic*), ref = (BP\arabic*)]
    \item\label{assum:bp-g-sig-km} The sequence of contaminating densities $k_m$ becomes asymptotically singular to the true density $g$, i.e., $\int \min\{ g, k_m \}d\mu \rightarrow 0$ as $m \rightarrow \infty$.
    \item\label{assum:bp-f-sig-km} The sequence of contaminating densities $k_m$ is such that for any compact subset $S \subset \bb{\Theta}$ with $S \cap \partial\bb{\Theta} = \emptyset$, we have $\int \min\{ f_{\bb{\theta}}, k_m \} d\mu \rightarrow 0$ as $m \rightarrow \infty$ uniformly on $\bb{\theta} \in S$.
    \item\label{assum:bp-f-sig-g} The model density $f_{\bb{\theta}}$ is such that for any sequence of parameters $\bb{\theta}_m \rightarrow \bb{\theta}_\infty$ where $\bb{\theta}_\infty \in \partial\bb{\Theta}$, the integral $\int \min\{ g, f_{\bb{\theta}_m} \} d\mu \rightarrow 0$ as $m \rightarrow \infty$.
    \item\label{assum:bp-integrable} The sequence of contaminating densities $k_m$ and the model family of densities $f_{\bb{\theta}}$ are uniformly $L^{\alpha+\beta}$-integrable, i.e.,
          \begin{equation*}
              \sup_m \norm{k_m}_{\alpha+\beta} < \infty, \ \text{ and, }
              \sup_{\bb{\theta} \in (\bb{\Theta} \setminus \partial\bb{\Theta})} \norm{f_{\bb{\theta}}}_{\alpha+\beta} < \infty.
          \end{equation*}
          \noindent If $\alpha\beta(\alpha+\beta) = 0$, then we require the densities $f_{\bb{\theta}}$ and $k_m$ to be uniformly $L^{\alpha+\beta+\delta}$-integrable for some $\delta > 0$, and additionally
          \begin{equation*}
              \sup_{\bb{\theta} \in (\bb{\Theta}\setminus \partial\bb{\Theta})} \int \vert f_{\bb{\theta}}^{\alpha+\beta} \log(g)\vert < \infty, \ \text{and, } \
              \sup_{\bb{\theta} \in (\bb{\Theta}\setminus \partial\bb{\Theta})}\sup_{m}  \int \vert f_{\bb{\theta}}^{\alpha+\beta} \log(k_m)\vert < \infty
          \end{equation*}
\end{enumerate}

It is noteworthy to highlight that the dominating measure $\mu$ could be a counting measure, and hence $f_{\bb{\theta}}$ and $g$ may be regarded as mass functions instead of probability density functions in the strict sense. The asymptotic singularity assumptions~\ref{assum:bp-g-sig-km}-\ref{assum:bp-f-sig-g} are possible in a discrete setting when the support of the underlying mass functions is infinite (e.g., Geometric, Poisson, etc.). We now state the formal result which provides an implicit lower bound to the asymptotic breakdown point for the MGABD function when both $\alpha$ and $\beta$ are positive.

\begin{proposition}\label{thm:gab-breakdown-general}
    Suppose Assumptions~\ref{assum:bp-g-sig-km}-\ref{assum:bp-integrable} hold, and consider any $\alpha, \beta > 0$. Furthermore, assume that there exists a threshold $\tilde{\epsilon}\in [0,1/2]$ satisfying the following property: for any $\epsilon < \tilde{\epsilon}$, one can find a corresponding $\bb{\theta}^g \in \bb{\Theta}$ such that for sufficiently large $m$, we have the inequality
    \begin{equation}
        \psi\left( (1-\epsilon)^\beta \inner{ f_{\bb{\theta}^g}, g}_{\alpha,\beta}\right)-\psi\left(\epsilon^\beta \inner{f_{\bb{\theta}_m},k_m}_{\alpha,\beta}\right) \geq \frac{\alpha\left(\psi\left( \norm{ f_{\bb{\theta}^g}}_{\alpha+\beta}^{\alpha+\beta} \right) -\psi\left( \norm{f_{\bb{\theta}_m}}_{\alpha+\beta}^{\alpha+\beta} \right) \right)}{\alpha+\beta},
        \label{cond:gab-breakdown-general}
    \end{equation}
    \noindent whenever $\bb{\theta}_m \rightarrow \bb{\theta}_\infty$ for some $\bb{\theta}_\infty \in \partial\bb{\Theta}$ and for any contaminating densities $\{ k_m \}$. Then, the asymptotic breakdown point of the minimum generalized Alpha-Beta divergence (MGABD) functional with a continuous $\psi$-function is at least $\tilde{\epsilon}$.
\end{proposition}


\cite{roy2023breakdown} have found the asymptotic breakdown point of the MSD functional under similar conditions, for the particular case when the true density $g$ belongs to the model family of densities $\Fcal := \{ f_{\bb{\theta}}: \bb{\theta} \in \bb{\Theta} \}$. This is simply a special case of Proposition~\ref{thm:gab-breakdown-general} with the choice of $\psi(x) = x$, the identity function. It is easy to see that with the choice of the identity function as $\psi(x)$ and $g = f_{\bb{\theta}^g}$, the condition given in~\eqref{cond:gab-breakdown-general} is equivalent to the condition (3.1) of~\cite{roy2023breakdown}; see Appendix~\ref{appendix:equivalence-bp} for details.

The condition~\eqref{cond:gab-breakdown-general} is usually difficult to verify directly in practice. However, there are some alternative stronger conditions, for specific classes of the $\psi$-functions and specific parametric setups, which are sufficient to imply that the condition~\eqref{cond:gab-breakdown-general} holds. In the following discussions, we explore these alternate conditions.

Firstly, we note that due to Assumption~\ref{assum:bp-integrable}, there must exist a large constant $C > 0$ such that $\norm{k_m}_{\alpha+\beta} \leq C$ for all sufficiently large $m$. This is the weakest condition that can be imposed to obtain a lower bound on the asymptotic breakdown point of the MGABD functional. However, this does not produce the lower bound explicitly, but as an implicit solution to a nonlinear equation, depending on the choice of $\psi$-function.

\begin{corollary}\label{cor:gab-bp-km-bound-C}
    Suppose that $\alpha, \beta > 0, (\alpha + \beta) \neq 1$ and let $C = \limsup_{m\to \infty}\norm{k_m}_{\alpha+\beta}$. Under Assumptions~\ref{assum:bp-g-sig-km}-\ref{assum:bp-integrable}, the asymptotic breakdown point of the MGABD functional for any valid continuous $\psi$-function is at least $\min\{ 1/2, \epsilon^\ast \}$, where $\epsilon^\ast$ is a solution of
    \begin{equation}
        \psi\left( \inner{f_{\bb{\theta}^g}, (1-\epsilon) g}_{\alpha,\beta} \right) - \dfrac{\alpha}{\alpha+\beta} \psi\left(\norm{f_{\bb{\theta}^g}}_{\alpha+\beta}^{\alpha+\beta} \right) - \dfrac{\beta}{\alpha+\beta} \psi\left( (C\epsilon)^{\alpha+\beta} \right) = 0,
        \label{cond:gab-bp-km-bound-C}
    \end{equation}
    \noindent in the interval $[0, 1]$ provided it exists. If the solution does not exist, we take $\epsilon^\ast$ as $0$.
\end{corollary}

In general, Eq.~\eqref{cond:gab-bp-km-bound-C} may not have a solution in $[0, 1]$. However, if $\psi$ is continuous, and the range of $\psi$ is a subset of nonnegative real numbers, and, the true density $g$ belongs to the model family $\Fcal$, i.e., $f_{\bb{\theta}^g} = g$, then the equation must have a solution. To see this, note that when $\epsilon = 0$, the left-hand side of the equation is $\beta(\alpha+\beta)^{-1}\psi(\norm{g}_{\alpha+\beta}^{\alpha+\beta}) > 0$. On the other hand, when $\epsilon = 1$, the left-hand side of the equation is $-(\alpha\psi(\norm{g}_{\alpha+\beta}^{\alpha+\beta})  + \beta\psi(C^{\alpha+\beta}))/(\alpha+\beta) < 0$. Now that Eq.~\eqref{cond:gab-bp-km-bound-C} has a solution follows from an application of the intermediate value theorem and the continuity of the $\psi$-function.

Some consequences of Corollary~\ref{cor:gab-bp-km-bound-C} may be of independent interest. For example, when $\psi(x) = \log(x)$, a lower bound for the asymptotic breakdown point of the MLSD functional is given by
\begin{equation}
    \min\left\{ \frac{1}{2}, \frac{\inner{f_{\bb{\theta}^g}, g }_{\alpha,\beta}^{1/\beta} }{\inner{f_{\bb{\theta}^g}, g }_{\alpha,\beta}^{1/\beta} + \norm{f_{\bb{\theta}^g}}_{\alpha+\beta}^{\alpha / \beta} C } \right\}.
\end{equation}
\noindent For the further special case when the true density belongs to the model family of densities, i.e., there exists $\bb{\theta}^g$ lying in the interior of $\bb{\Theta}$ such that $f_{\bb{\theta}^g} = g$, then the asymptotic breakdown point of the MLSD functional has a lower bound $\min\{ 1/2,  (1 + (C / \norm{g}_{\alpha+\beta} ))^{-1} \}$, where $C = \limsup_{m \to \infty} \norm{k_m}_{\alpha+\beta}$. This means, as long as the contaminating densities $k_m$ are sufficiently light tailed so that $\norm{k_m}_{\alpha+\beta} \leq \norm{g}_{\alpha+\beta}$, the MLSD functional achieves the highest possible breakdown point $1/2$.

Fortunately, such high robust performance is not limited to the MLSD functional, but rather can appear for the much larger class of MGABD functionals for the location estimation problem. The following corollary of Proposition~\ref{thm:gab-breakdown-general} formally describes this result.

\begin{corollary}\label{cor:gab-bp-location-model}
    If the model family $\{ f_{\bb{\theta}} \}$ is a location family, $\bb{\theta}$ is the location parameter to be estimated and the true density $g$ and the sequence of contaminating densities $\{ k_m \}$ all belong to the same location family, then under Assumptions~\ref{assum:bp-g-sig-km}-\ref{assum:bp-integrable}, the asymptotic breakdown point of MGABD functional for any valid continuous $\psi$ function is equal to $1/2$, for any $\alpha,\beta > 0$.
\end{corollary}

\begin{remark}\label{remark:gab-bp-affine}
    Suppose that $\alpha, \beta > 0, (\alpha + \beta) \neq 1$ and Assumptions~\ref{assum:bp-g-sig-km}-\ref{assum:bp-integrable} hold. Then the asymptotic breakdown point of the MGABDE must have a lower bound $\min\{ 1/2, \epsilon^\ast\}$ with
    \begin{equation}
        \epsilon^\ast \leq 1 - \left[ \frac{\psi^{-1}(\frac{\alpha}{\alpha+\beta} \psi(\norm{f_{\bb{\theta}^g}}_{\alpha+\beta}^{\alpha+\beta} ) ) }{ \inner{f_{\bb{\theta}^g}, g}_{\alpha,\beta} } \right]^{1/\beta}.
        \label{eqn:max-bp-bound}
    \end{equation}
    \noindent To see this, one can start with the $\epsilon^\ast$ satisfying condition~\eqref{cond:gab-bp-km-bound-C} as in Corollary~\ref{cor:gab-bp-km-bound-C}, and then replace $\psi$ with $\psi_1(x) = \psi(x) - \psi((C\epsilon^\ast)^{\alpha+\beta})$, which due to Remark~\ref{remark:affine-psi} should produce the same estimate. One should note that this line of reasoning works for any choice of $C$, and hence, the inequality~\eqref{eqn:max-bp-bound} always holds.
\end{remark}

Remark~\ref{remark:gab-bp-affine} describes the constraint that a lower bound to the asymptotic breakdown point of the MGABDE must satisfy. In the following corollary, we restrict our attention to the generating function $\psi$ to take on a particular form, namely the ones proposed by~\cite{jones2001comparison}. It turns out that with this specific choice, a lower bound can be obtained that achieves the equality case of the constraint~\eqref{eqn:max-bp-bound}, demonstrating its tightness. A formal depiction of the corresponding result is as follows.

\begin{corollary}\label{cor:gab-bp-special1}
    Suppose that $\alpha, \beta > 0, (\alpha+\beta) \neq 1$ and the Assumptions~\ref{assum:bp-g-sig-km}-\ref{assum:bp-integrable} hold. Additionally, assume that the sequence of contaminating densities is such that $\norm{k_m}_{\alpha+\beta} \leq \norm{f_{\bb{\theta}_m}}_{\alpha+\beta}$ for sufficiently large $m$. If $\psi(x) = \phi^{-1}x^\phi$ for some $\phi > 0$, then the asymptotic breakdown point of the corresponding MGABD functional is at least
    \begin{equation}
        \min\left\{ \left(\dfrac{\alpha}{\alpha+\beta} \right)^{1/(\beta\phi)}, 1-\left(\dfrac{\alpha}{\alpha+\beta} \right)^{1/(\beta\phi)} \dfrac{ \norm{f_{\bb{\theta}^g}}_{\alpha+\beta}^{1 + \alpha / \beta} }{ \inner{f_{\bb{\theta}^g}, g}_{\alpha, \beta}^{1/\beta}}, \dfrac{1}{2} \right\}
        \label{eqn:cor-bp-special1-bound}
    \end{equation}
\end{corollary}

\begin{remark}\label{remark:cor-bp-special1-extension}
    The conclusion of Corollary~\ref{cor:gab-bp-special1} can be extended to any $\psi$-function that satisfies the following requirement: There exists some $\phi > 0$ such that
    \begin{equation}
        \alpha(\alpha+\beta)^{-1}\psi(x) \geq \psi\left( \alpha^{1/\phi}(\alpha+\beta)^{-1/\phi} x\right),
        \label{eqn:cor-bp-special1-growth-requirement}
    \end{equation}
    \noindent for all $x \in [0,\infty)$. For such $\psi$-functions, it is fairly straightforward to verify that     \begin{equation*}
        \left[\frac{\psi^{-1}\left(\frac{\alpha}{\alpha+\beta}\psi( \norm{f_{\bb{\theta}_m}}_{\alpha+\beta}^{\alpha+\beta})\right)}{ \inner{f_{\bb{\theta}_m},k_m}_{\alpha,\beta}}\right]^{1/\beta}
        \geq \left( \dfrac{\alpha}{\alpha+\beta} \right)^{1/(\beta\phi)} \left[ \dfrac{\norm{f_{\bb{\theta}_m}}_{\alpha+\beta}^{(\alpha+\beta)} }{\inner{f_{\bb{\theta}_m}, k_m}_{\alpha,\beta} } \right]^{1/\beta}.
    \end{equation*}
    \noindent This inequality stems from the strictly monotonically increasing property of $\psi$, allowing the subsequent steps in the proof of Corollary~\ref{cor:gab-bp-special1} to be reproduced verbatim. The condition~\eqref{eqn:cor-bp-special1-growth-requirement} is essentially a requirement on the growth rate of $\psi$-function, and is satisfied by any function growing at least as fast as $x^\phi$.
\end{remark}

Unfortunately, for the special case of MLSD functional, Corollary~\ref{cor:gab-bp-special1} only provides a trivial lower bound. To see this, a viable approach would be to consider a sequence of generating functions $\psi_\phi = \phi^{-1}x^\phi$ and take the limit as $\phi \to 0+$, but this yields the first term of Eq.~\eqref{eqn:cor-bp-special1-bound} tends to $0$, reducing to a trivial lower bound.

\subsection{Asymptotic Breakdown Point of MGABD functional for the case $\alpha > 0, \beta = 0$}\label{sec:gab-bp-zero}

A crucial ingredient of the asymptotic breakdown analysis results derived in the previous Section~\ref{sec:gab-bp-positive} is H\"{o}lder's inequality, which does not hold for the case when either $\alpha = 0$ or $\beta = 0$. However, when $\alpha = 0$, usually the asymptotic breakdown point only has a trivial lower bound, at least for the special case $\psi(x) = x$; see Remark 3.1 of~\cite{roy2023breakdown}. Therefore, we restrict our attention to only the case $\alpha > 0, \beta = 0$ in this section.

\begin{proposition}\label{thm:gab-bp-general-zero}
    Let us assume that the Assumptions~\ref{assum:bp-g-sig-km}-\ref{assum:bp-integrable} are satisfied and $\psi \in C^1((0,\infty))$ be a valid generating function for GAB divergence. Suppose, there exists $\tilde{\epsilon} \in [0,1/2]$ satisfying the following property: for all $\epsilon < \tilde{\epsilon}$ there exists $\bb{\theta}^g \in \bb{\Theta}$ depending on $\epsilon$ such that for sufficiently large $m$, we have
    \begin{multline}
        \psi'\left( \norm{f_{\bb{\theta}_m}}_{\alpha}^{\alpha} \right) \int f_{\bb{\theta}_m}^{\alpha} \log\left( f_{\bb{\theta}_m}/\epsilon k_m \right) - \psi'\left( \norm{f_{\bb{\theta}^g}}_{\alpha}^{\alpha} \right) \int f_{\bb{\theta}^g}^{\alpha} \log\left( f_{\bb{\theta}^g} / (1-\epsilon)g \right)\\
        > \alpha^{-1} \left[ \psi\left( \norm{f_{\bb{\theta}_m}}_{\alpha}^{\alpha} \right) - \psi\left( \norm{f_{\bb{\theta}^g}}_{\alpha}^{\alpha} \right) \right]
        \label{cond:gab-bp-general-zero}
    \end{multline}
    \noindent whenever $\bb{\theta}_m \rightarrow \bb{\theta}_\infty$ for some $\bb{\theta}_\infty \in \partial\bb{\Theta}$ and for any sequence of contaminating densities $\{ k_m \}_{m=1}^\infty$. Then, the MGABD functional with generating function $\psi$ for the $\alpha > 0, \beta = 0$ case is at least $\min\{ \tilde{\epsilon}, 1/2 \}$.
\end{proposition}
\noindent Now that we have established the general results, we can obtain corollaries similar to the ones described in Section~\ref{sec:gab-bp-positive}. We begin with a corollary similar to Corollary~\ref{cor:gab-bp-km-bound-C} for the specific case when $\beta = 0$.

\begin{corollary}\label{cor:gab-bp-km-bound-C-zero}
    Suppose that $\beta = 0$ and the Assumptions~\ref{assum:bp-g-sig-km}-\ref{assum:bp-integrable} hold. Let us denote $C = \limsup_{m\to \infty}\norm{k_m}_{\alpha} > 0$. Then, the asymptotic breakdown point of the MGABD functional is at least $\min\{ 1/2, \epsilon^\ast \}$, where $\epsilon^\ast$ is a solution of
    \begin{equation}
        \psi'\left( \norm{f_{\bb{\theta}^g}}_{\alpha}^{\alpha}\right)\int f_{\bb{\theta}^g}^{\alpha}\log\left(\frac{(1-\epsilon)g}{f_{\bb{\theta}^g}}\right) - \dfrac{1}{\alpha} \left[\psi\left( (C\epsilon)^{\alpha} \right)-\psi\left( \norm{f_{\bb{\theta}^g}}_{\alpha}^{\alpha}\right)\right] = 0,
        \label{cond:gab-bp-km-bound-C-zero}
    \end{equation}
    \noindent in the interval $[0, 1]$ provided it exists. If the solution does not exist, we take $\epsilon^\ast$ as $0$.
\end{corollary}

As before, when $\psi \in C^1((0, \infty))$ and $g$ belongs to the model family so that $g = f_{\theta^g}$, the condition~\eqref{cond:gab-bp-km-bound-C-zero} must have a solution. Note that, when $\epsilon = 0$, the left-hand side of condition~\eqref{cond:gab-bp-km-bound-C-zero} is $\alpha^{-1}(\psi(\norm{g}_{\alpha}^{\alpha}) - \psi(0)) > 0$. When $\epsilon = 1$, the left-hand side becomes negative. The existence of the solution is then guaranteed by the intermediate value theorem.

Similar to the Corollary~\ref{cor:gab-bp-location-model}, one can show that for the location estimation problem, the asymptotic breakdown point of the MGABD functional for the $\alpha > 0, \beta = 0$ case also achieves the highest possible breakdown of $1/2$.

\begin{corollary}\label{cor:gab-bp-location-model-zero}
    If $\Fcal := \{ f_{\bb{\theta}}: \bb{\theta} \in \bb{\Theta} \}$ is a location model family of densities where $\bb{\theta}$ is the location parameter and the true density $g$ and the sequence of contaminating densities $\{ k_m \}$ belong to the same location model family, then under the Assumptions~\ref{assum:bp-g-sig-km}-\ref{assum:bp-integrable}, the asymptotic breakdown point of the MGABD functional is $1/2$ for any valid $\psi$-function and hyperparameters $\alpha > 0, \beta = 0$.
\end{corollary}

\begin{remark}\label{remark:gab-bp-affine-zero}
    As in Remark~\ref{remark:gab-bp-affine}, one can consider an affine transformation of $\psi$ as $\psi_1(x) = \psi(x) - \psi((C\epsilon^\ast)^\alpha)$ where $\epsilon^\ast$ is a solution to Eq.~\eqref{cond:gab-bp-km-bound-C-zero}. This results in a lower bound  on the asymptotic breakdown point of MGABDE under $\beta = 0$ case given by $\min\{1/2, \epsilon^\ast\}$ where
    \begin{equation*}
        \epsilon^\ast \leq 1 - \exp\left[ \frac{1}{\norm{f_{\bb{\theta}^g}}_\alpha^\alpha } \left\{ \int f_{\bb{\theta}^g}^\alpha \log\left(\frac{f_{\bb{\theta}^g}}{g} \right) - \frac{\psi(\norm{f_{\bb{\theta}^g}}_\alpha^\alpha )}{\alpha\psi'(\norm{f_{\bb{\theta}^g}}_\alpha^\alpha ) } \right\} \right].
    \end{equation*}
\end{remark}

Under conditions similar to those of Corollary~\ref{cor:gab-bp-special1}, but for the specific $\alpha > 1, \beta = 0$ case, one now obtains the following corollary. 

\begin{corollary}\label{cor:gab-bp-zero-special1}
    Let us assume $\alpha > 1, \beta = 0$ and grant Assumptions~\ref{assum:bp-g-sig-km}-\ref{assum:bp-integrable}. Additionally, suppose that the contaminating densities $\{ k_m \}$ satisfy $\norm{k_m}_{\alpha} \leq \norm{f_{\bb{\theta}_m}}_{\alpha}$ for all sufficiently large $m$. Then, the MGABD functional with generating function $\psi(x) = \phi^{-1}x^\phi$ has an asymptotic breakdown point which is greater than or equal to
    \begin{equation*}
        \min\left\{ \exp\left( -\dfrac{1}{\alpha\phi} \right), 1 - \exp\left( \dfrac{\int f_{\bb{\theta}^g}^{\alpha} \log(f_{\bb{\theta}^g}/g) }{ \norm{f_{\bb{\theta}^g}}_{\alpha}^{\alpha} } -\dfrac{1}{\alpha\phi} \right), \frac{1}{2} \right\}.
    \end{equation*}
    \noindent Furthermore, if the true density $g$ belongs to the model family of densities $\Fcal$, i.e., we have $g = f_{\bb{\theta}^g}$, then the above lower bound vastly simplifies to $\min\{ e^{-1/(\alpha\phi)}, 1 - e^{-1/(\alpha\phi)} \}$.
\end{corollary}

However, for the special case of LSD with $\phi \to 0+$, we find that the first term $e^{-1/(\alpha\phi)} \to 0$, resulting in a trivial lower bound. Unfortunately, this lower bound cannot be improved further. To see this, note that the GAB divergence between $f$ and $g$ in this special case is equivalent to the KL-divergence between densities proportional to $f^\alpha$ and $g^\alpha$; see Eq. (9)-(10) of~\cite{roy2025characterization} for more technical details. Hence, the minimization of GAB divergence in this special case is equivalent to the minimum KL-divergence estimator, i.e., the maximum likelihood estimator (MLE), which is known to be nonrobust having asymptotic breakdown point of $0$.

\section{Optimality of efficiency and robustness}\label{sec:optimality}

We are now ready to tackle the primary objective of this paper, i.e., to find the most statistically efficient choice of the $\psi$-function that also provides a global robustness guarantee via asymptotic breakdown point. We present our discussions in two separate cases depending on whether $\beta$ is positive.


\subsection{Optimality under positive $\beta$}

Consider the special case where both $\alpha, \beta > 0$ and either $\beta = 1$ or the model family $\Fcal$ only contains discrete distributions. In this case, we can use Corollaries~\ref{thm:normality-beta-1}-\ref{thm:normality-discrete} to obtain the form of asymptotic variance of the MGABDE under suitable regularity conditions. Under a different set of assumptions, Corollary~\ref{cor:gab-bp-km-bound-C} also provides a lower bound to the asymptotic breakdown point of the MGABDE when the contaminating densities $\{ k_m \}$ satisfy uniform $L^{\alpha + \beta}$-integrability; namely $C = \limsup_{m \to \infty} \norm{k_m}_{\alpha+\beta}^{\alpha+\beta} < \infty$. Armed with these results, one can now proceed towards developing an optimal choice of the generating function $\psi$ such that the resultant MGABDE has the minimum variance at the model, while also guaranteeing a certain level of robustness properties under a Huber-style contamination. Compactly, the optimization problem is described as
\begin{equation*}
    \text{Minimize } (\psi'(\norm{g}_{\alpha+\beta}^{\alpha+\beta}))^2\text{Trace}\left\{ \bb{J}^{-1} (Q_{0,2(\alpha+\beta)-1} - P_{0,\alpha+\beta}P_{0,\alpha+\beta}\tr ) \bb{J}^{-1} \right\}
\end{equation*}
\noindent where
\begin{equation*}
    \bb{J} = \left[ \psi''(\norm{g}_{\alpha+\beta}^{\alpha+\beta}) P_{0,\alpha+\beta}P_{0,\alpha+\beta}\tr + \psi'(\norm{g}_{\alpha+\beta}^{\alpha+\beta}) Q_{0,\alpha+\beta} \right],
\end{equation*}
\noindent subject to the restriction that
\begin{equation}
    \psi((1-\epsilon^\ast)^\beta \norm{g}_{\alpha+\beta}^{\alpha+\beta} ) - \frac{\alpha}{\alpha+\beta} \psi(\norm{g}_{\alpha+\beta}^{\alpha+\beta}) - \frac{\beta}{\alpha+\beta}\psi(C^{\alpha+\beta} (\epsilon^\ast)^{\alpha+\beta}) \geq 0
    \label{eqn:optimality-bp-constraint-discrete}
\end{equation}
\noindent for some exogenously given $\epsilon^\ast \in (0, 1/2]$. In relation to Corollary~\ref{cor:gab-bp-km-bound-C}, the constraint~\eqref{eqn:optimality-bp-constraint-discrete} ensures that the asymptotic breakdown point of the MGABDE with feasible $\psi$-function is at least $\epsilon^\ast$. A careful reorganization of the objective function reveals that the minimization of the variance amounts to maximizing the sum of the eigenvalues of the matrix
\begin{equation}
    \frac{\psi''(\norm{g}_{\alpha+\beta}^{\alpha+\beta}) }{\psi'(\norm{g}_{\alpha+\beta}^{\alpha+\beta})} P_{0,\alpha+\beta}P_{0,\alpha+\beta}\tr + Q_{0,\alpha+\beta}.
    \label{eqn:optimality-objective-discrete}
\end{equation}
\noindent Unfortunately, this problem is not well-posed without additional constraints. Firstly, as illustrated in Remark~\ref{remark:affine-psi}, the MGABDE remains invariant under an affine transformation of $\psi$ with positive scaling factor, and hence we may assume $\psi(1) = 0$ and $\psi'(1) = 1$. Additionally, since the true density $g$ is unknown, it is reasonable to impose the requirement that the constraint in~\eqref{eqn:optimality-bp-constraint-discrete} to hold for any choice of $\norm{g}_{\alpha+\beta}^{\alpha+\beta}$. This ensures that the asymptotic breakdown point of the MGABDE is greater than or equal to $\epsilon^\ast$ under any setup. However, even under this global constraint, it is not possible to uniformly minimize the expression in Eq.~\eqref{eqn:optimality-objective-discrete}; the following remark produces a concrete counterexample.

\begin{remark}\label{remark:counter-uniform-opt}
    Let $\phi^\ast = -\log(1 + \beta/\alpha) / \beta(1-\epsilon^\ast)$. Take, $\psi(x) = x^{\phi^\ast}(1 + \delta e^{-x})$, for some carefully chosen small $\delta > 0$. It can be shown that $\psi$ is increasing, geometrically convex, and satisfies the constraint
    \begin{equation*}
        \psi((1 - \epsilon^\ast)^\beta x) - \alpha(\alpha+\beta)^{-1} \psi(x) \geq 0,
    \end{equation*}
    \noindent which is equivalent to satisfying Eq.~\eqref{eqn:optimality-bp-constraint-discrete} due to the translation invariance of $\psi$-function. Also, by modifying the choice of $\delta$, we can continuously increase $\psi''(x)/\psi'(x)$ on some particular subset of $\R$, thus minimizing the variance arbitrarily due to the relation~\eqref{eqn:optimality-objective-discrete}. In this case, the specific choice of $\psi$-function produces very efficient estimates on some choices of $g$, while performing significantly poorly for the other choices.
\end{remark}

An alternative approach to finding the optimal estimator would be to consider the worst-case variance of the estimator. Note that,
\begin{equation*}
    P_{0,\alpha+\beta} = \norm{g}_{\alpha+\beta}^{\alpha+\beta} \E_{X \sim g^{[\alpha+\beta]}}(u_{\bb{\theta}}(X)), \
    Q_{0,\alpha+\beta} = \norm{g}_{\alpha+\beta}^{\alpha+\beta} \E_{X \sim g^{[\alpha+\beta]}}(u_{\bb{\theta}}(X) u_{\bb{\theta}}\tr(X) ),
\end{equation*}
\noindent where $g^{[\alpha+\beta]}$ denotes the $(\alpha+\beta)$-escorted density of $g$, i.e., $g^{[\alpha+\beta]} \propto g^{\alpha+\beta}$; see Section 2.1 of~\cite{roy2025characterization} for definition. An inspection of the expression given in Eq.~\eqref{eqn:optimality-objective-discrete} yields that the curvature of the variance of the MDPDE is governed by the quantity
\begin{equation}
    L^\psi(\norm{g}_{\alpha+\beta}^{\alpha+\beta}) = \norm{g}_{\alpha+\beta}^{\alpha+\beta} \psi''(\norm{g}_{\alpha+\beta}^{\alpha+\beta}) \left[ \psi'(\norm{g}_{\alpha+\beta}^{\alpha+\beta})\right]^{-1}.
    \label{eqn:bp-optimality-R-definition}
\end{equation}
\noindent Hence, we restrict our attention to maximizing the infimum $\inf_{x > 0} L^\psi(x)$ in order to minimize the worst-case variance. As we consider the infimum of $L^\psi(x)$ over $x \in (0, \infty)$, the same line of reasoning applies to the continuous case (cf.~Corollary~\ref{thm:normality-continuous}) as well, where instead of $\norm{g}_{\alpha+\beta}^{\alpha+\beta}$, we consider the kernel-convoluted density $\norm{\tilde{g}}_{\alpha+\beta}^{\alpha+\beta}$. The solution to this problem turns out to be a member of the extended $(\phi, \gamma)$-divergence family given in~\eqref{eqn:jhhb-div}, with a specific choice of the exponent $\phi$. This is formally described through the following theorem.

\begin{theorem}\label{thm:gab-optimality-discrete}
    Assume $\alpha, \beta > 0$. Furthermore, assume that all the conditions for Corollaries~\ref{thm:normality-beta-1}, \ref{thm:normality-discrete} and~\ref{thm:normality-continuous} hold depending on the underlying setup, whether $\beta = 1$ or the model family is discrete or continuous. Additionally, suppose that conditions for Corollary~\ref{cor:gab-bp-km-bound-C} hold. Then, the optimal generating function $\psi \in C^2((0, \infty))$ such that $\inf_{x > 0} L^\psi(x)$ is maximized (or equivalently, the worst-case variance given in Eq.~\eqref{eqn:optimality-objective-discrete} is minimized) among those with asymptotic breakdown point at least $\epsilon^\ast$ is given by $\psi^\ast(x) = A + B x^{\phi^\ast}$, for some $A \in \R$ and $B > 0$ and,
    \begin{equation}
        \phi^\ast = -\log(1 + \beta/\alpha)/(\beta \log(1-\epsilon^\ast)).
        \label{eqn:best-gamma-star}
    \end{equation}
    \noindent Additionally, if $\psi(1) = 0$ and $\psi'(1) = 1$, then $\psi^\ast(x) =  (x^{\phi^\ast} - 1)/\phi^\ast$.
\end{theorem}

\begin{figure}[htbp]
    \centering
    \includegraphics[width=\linewidth]{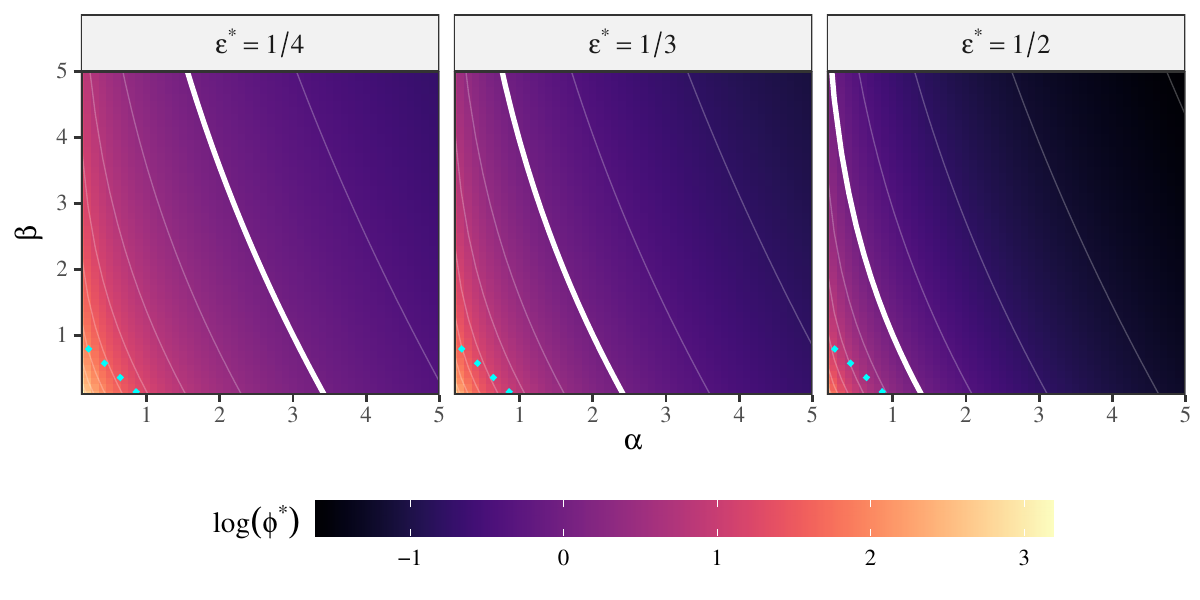}
    \caption{The choice of the optimal $\phi^\ast$ as obtained through Theorem~\ref{thm:gab-optimality-discrete} for varying $\alpha, \beta$ and exogenously given asymptotic breakdown point $\epsilon^\ast$. The bold white contour refers to the choice of $\phi^\ast = 1$, i.e., the identity $\psi(x) = x$ function corresponding to DPD. The green dotted line indicates $\alpha + \beta = 1$ relating to the PD family.}
    \label{fig:optimal-gamma}
\end{figure}

We note that the conclusions of Theorem~\ref{thm:gab-optimality-discrete} are actually free of $C = \limsup_{m \to \infty} \norm{k_m}_{\alpha+\beta}^{\alpha+\beta}$, or in particular, the choice of the contaminating densities $k_m$, as long as they are uniformly $L^{\alpha+\beta}$-integrable.

\begin{remark}\label{remark:nonunique-optimal}
    It is worthwhile to note that Theorem~\ref{thm:gab-optimality-discrete} merely provides a constructive existence of the optimal choice of the divergence, and does not imply uniqueness of this choice. In fact, when $\alpha + \beta = 1$, any valid choice of $\psi$-function leads to a monotonic transformation of the power divergence family (cf.~\cite[Section 4.1]{roy2025characterization}), in turn, producing exactly the same minimum divergence estimator. For example, when $\alpha = \beta = 0.5$, the MGABDE corresponds to the well-known minimum Hellinger distance estimator (MHDE), which is known to be asymptotically fully efficient yet achieves a breakdown point of $1/2$ under specific settings~\citep{tamura1986minimum, toma2007minimum}.
\end{remark}

More generally, if for some setting, $P_{0,\alpha+\beta} = \int g^{\alpha+\beta} u_{\bb{\theta}} = 0$, then from the variance formula of Eq.~\eqref{eqn:optimality-objective-discrete}, we see that all choices of $\psi(\cdot)$ function leads to a constant variance free of $\psi$. Some interesting implications of Theorem~\ref{thm:gab-optimality-discrete} become apparent through Figure~\ref{fig:optimal-gamma}. When $\beta = 1$ and $\alpha \in (0, 1]$, the MDPDE is never the optimal estimator except when $\alpha = 1$, achieving an asymptotic breakdown point of $1/2$. If one only requires a smaller robustness guarantee (e.g., $\epsilon^\ast = 1/4$ or $\epsilon^\ast = 1/3$), the MDPDE can be the most efficient estimator when $\alpha$ is much larger than $1$; but it leads to higher variance as pointed out by several authors~\citep{basu1998robust, ghosh2013robust}. Typically, using a quadratic (or cubic) $\psi$-function improves the performance of minimum divergence estimation in this scenario. A table containing the expressions of asymptotic variance for this optimal estimator under different settings is provided in Appendix~\ref{appendix:av-table}.


\subsection{Optimality under $\beta = 0$}

Let us now shift our focus to the case when $\beta = 0$. Suppose we are exogenously given $\epsilon^\ast \in (0, 1/2]$, and we would like the asymptotic breakdown point of the MGABDE estimator to be at least $\epsilon^\ast$. In view of Corollary~\ref{cor:gab-bp-km-bound-C-zero}, the feasible set of generating functions $\psi$ needs to satisfy
\begin{equation}
    \psi'(x) x \log(1-\epsilon^\ast) - \alpha^{-1}\left[ \psi((C\epsilon^\ast)^\alpha) - \psi(x) \right] \geq 0,
    \label{eqn:optimality-constraint-beta0}
\end{equation}
\noindent for all $x \in \R$, where $C = \limsup_{m \to \infty} \Vert k_m\Vert_\alpha$. This condition, as before, restricts the growth rate of the $\psi$-function. To see this, note that by translation of the $\psi$-function, the inequality~\eqref{eqn:optimality-constraint-beta0} is equivalent to
\begin{equation}
    \frac{x\psi'(x)}{\psi(x)} = \frac{\Psi'(x)}{\Psi(x)} \leq -\frac{1}{\alpha \log(1-\epsilon^\ast)},
    \label{eqn:optimality-growth-bound-beta0}
\end{equation}
\noindent where $\Psi(x) = \psi(e^x)$, which is further equivalent to saying, $\psi(x) \leq C' x^{-1/\alpha \log(1-\epsilon^\ast)}$ for some $C' > 0$ and for all sufficiently large $x \in \R$.

On the other hand, in view of Corollaries~\ref{thm:normality-discrete-beta0} and~\ref{thm:normality-continuous-beta0}, the asymptotic variance of the MGABDE has the form
\begin{equation}
    V^\psi(x) := \bb{B}(x)^{-1} \left( \bb{A}_1\xi_1^2(x) + 2\bb{A}_2\xi_1(x) + \bb{A}_3 \right) \bb{B}(x)^{-1},
    \label{eqn:optimality-variance-beta0}
\end{equation}
\noindent where $x = \norm{g}_{\alpha+\beta}^{\alpha+\beta}$, $\bb{B}(x) = \left( \bb{B}_1 \xi_2(x) + \bb{B}_2 \xi_1(x) - \bb{B}_2 \right)$, $\bb{A}_i$s and $\bb{B}_i$s are some appropriate constants, depending on the underlying setup, and $\xi_1(x) = \psi''(x) / \psi'(x)$ and $\xi_2(x) = \psi'''(x) / \psi'(x)$. A statistician would then aim to minimize this expression of the variance, if possible, uniformly over $x$, while maintaining a guarantee on the breakdown point via inequality~\eqref{eqn:optimality-constraint-beta0}. Unfortunately, such a uniform optimality is unattainable, as shown in the following remark.

\begin{remark}\label{remark:counter-uniform-opt-beta0}
    Pick any $x_0 \in \R$ and $M > 0$. Consider a function $\psi$ such that, $\psi'(x_0) = 1$, $\psi''(x) = 0$ except in the interval $[x_0 - 1/M, x_0 + 1/M]$, with $\psi''(x_0) = 1$ and $\psi'''(x_0) = M$. Note that, since for sufficiently large $x$, $\psi''(x) = 0$, it grows slower than the rate $Cx^{-1/(\alpha \log(1-\epsilon^\ast))}$, and hence satisfy inequality~\eqref{eqn:optimality-constraint-beta0}. However, the denominator of the asymptotic variance given in Eq.~\eqref{eqn:optimality-variance-beta0} at $x = x_0$ scales as $M^2$, hence the variance can be made arbitrarily small as $M > 0$ is arbitrary.
\end{remark}

Therefore, we restrict our attention to deriving the most efficient estimator by minimizing the worst-case variance behavior, i.e., minimizing $\sup_{x \geq 0} V^\psi(x) / x^2$ with respect to the choice of $\psi$-function. The division by $x^2$ ensures that the supremum leads to a nontrivial stable quantity. The optimal function again turns out to belong to the family of extended $(\phi,\gamma)$-divergences; this is formalized into the following result.

\begin{theorem}\label{thm:gab-optimality-beta0}
    Assume $\alpha > 0, \beta = 0$. Furthermore, assume that all the conditions for Corollaries~\ref{thm:normality-discrete-beta0} and~\ref{thm:normality-continuous-beta0} hold, depending on the underlying setup, whether the model family is discrete or continuous. Additionally, suppose that conditions for Corollary~\ref{cor:gab-bp-km-bound-C-zero} hold. Then, for any exogenously given $\epsilon^\ast$, the optimal generating function $\psi \in C^3((0, \infty))$ that minimizes the worst case asymptotic variance while also ensuring that the asymptotic breakdown point of the corresponding MGABDE is at least $\epsilon^\ast$ is given by $\psi^\ast(x) = A + B x^{\phi^\ast}$, for some $A \in \R$ and $B > 0$ and, $\phi^\ast = -1 / (\alpha \log(1-\epsilon^\ast))$. Additionally, if $\psi(1) = 0$ and $\psi'(1) = 1$, then $\psi^\ast(x) =  (x^{\phi^\ast} - 1)/\phi^\ast$.
\end{theorem}

Theorem~\ref{thm:gab-optimality-beta0} can be viewed as a limiting version of Theorem~\ref{thm:gab-optimality-discrete}. It is evident that,
\begin{equation*}
    \lim_{\beta \to 0} \frac{\log(1 + \beta/\alpha)}{\beta \log(1-\epsilon^\ast)}
    = \frac{1}{\alpha\log(1-\epsilon^\ast)} \lim_{\beta \to 0} \frac{\log(1 + \beta / \alpha)}{(\beta / \alpha)} = \frac{1}{\alpha \log(1-\epsilon^\ast)},
\end{equation*}
\noindent which illustrates that under the $\beta = 0$ case, the optimal choice of the $\psi$-function is again governed by a similar interplay between $\alpha$ and the desired breakdown point $\epsilon^\ast$.

\subsection{Hyperparameter tuning}\label{sec:hyperparameter}

Theorems~\ref{thm:gab-optimality-discrete}-\ref{thm:gab-optimality-beta0} reveal that once a statistician equips herself with an appropriate choice of $\alpha$ and $\beta$, and the exogenously given amount of contamination $\epsilon^\ast$ to guard against, the choice of the best $\psi$-function is automatic. However, a careful choice of $\alpha$ and $\beta$ ultimately affects the asymptotic variance under the model. Often, a particularly useful choice is $\beta = 1$, which presents opportunities to perform inference free of any nonparametric smoothing~\citep{jana2019characterization}. In the existing literature, the choice of $\alpha$ is typically motivated through the amount of robustness required as demonstrated in~\cite{basak2021optimal}. However, the results of this paper ensure that with $\beta = 1$ and any choice of $\alpha > 0$, one should be able to obtain a $\phi^\ast$ such that its asymptotic breakdown point is at least $\epsilon^\ast$. However, too low $\alpha$ results in larger $\phi^\ast$, which leads to numerically unstable estimates. Therefore, the choice of $\alpha > 0$ may be strictly motivated by either minimizing the variance or from a practical point of view. To obtain a choice of $\alpha$, an iterative procedure as in~\cite{basak2021optimal} is possible, where one starts with an initial choice of $\alpha = \alpha^{(0)}$, obtains the MGABDE estimate by using the optimal choice of $\psi$-function, uses the estimated parameter to compute the asymptotic variance for a range of choices of $\alpha$, and picks the best one $\alpha^{(1)}$ achieving the minimum possible variance. These iterations may be repeated several times until convergence, though additional investigation is necessary to ensure the convergence and correctness of this procedure, which we leave to future research.


\section{Empirical Studies}\label{sec:empirical}

As a brief demonstration of our theoretical findings, we present two empirical illustrations in this section. Additional numerical experiments are available in Section~\ref{appendix:empirical} of the Supplementary Material.

\subsection{Dimension independent breakdown guarantee}\label{sec:dim-bp-normalloc}

\begin{figure}[tbp]
    \centering
    \includegraphics[width=0.8\linewidth]{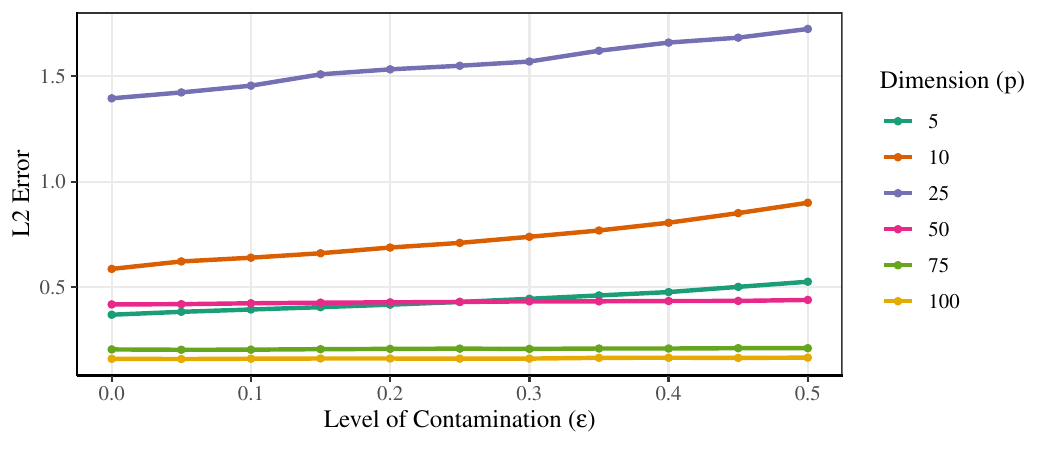}
    \caption{Trajectories of estimation error across varying data dimension and increasing contamination level}
    \label{fig:normalloc-estimate}
\end{figure}

Corollary~\ref{cor:gab-bp-location-model} ensures that the asymptotic breakdown point of the MGABDE for a location estimation problem is exactly $1/2$, independent of the data dimension. To validate this empirically, we consider a multivariate setting where the clean observations are drawn as $X_1, \dots, X_n \sim \normdist(\bb{0}_p, \bb{\Sigma})$. The covariance matrix $\bb{\Sigma}$ is defined with ones on the main diagonal, $0.9$ on the first super- and sub-diagonals, and zeros elsewhere. For a given dimension $p$, we introduce contamination drawn from $\normdist(100\bb{1}_p/\sqrt{p}, \bb{\Sigma})$. This specific scaling ensures that the $L^2$-distance between the true and contaminating mean vectors remains exactly $100$, rendering the contamination severity the same across values of $p$. The trajectories of the $L^2$-error for the MGABDE under varying contamination proportions and dimensions are illustrated in Figure~\ref{fig:normalloc-estimate}. Notably, for the location estimation problem, the minimizers of the MGABDE objective functions~\eqref{eqn:H-theta-beta1-empirical}--\eqref{eqn:H-theta-BL-empirical} are invariant to the choice of the generating function $\psi$; therefore, for this specific experiment, we utilize $\psi(x) = x$ with $\alpha = 0.5$ and $\beta = 1$. As demonstrated in Figure~\ref{fig:normalloc-estimate}, the $L^2$-error of the estimate remains highly stable across all dimensions for contamination levels $\epsilon \in [0, 1/2]$, never approaching the contamination distance of $100$.

\subsection{Refinement of existing minimum divergence estimators}

To illustrate our results on the robustness-efficiency trade-off, we evaluate the robustness-efficiency trade-off in the following experimental setting. Consider the problem of estimating the scale parameter $\theta$ of a Gamma distribution with a known shape parameter of $2$. The true scale parameter is $\theta^g = 1$, and the data are subjected to a Huber-style contamination model $g_\epsilon = (1-\epsilon)g + \epsilon k$, where the contaminating density $k$ is a Gamma distribution with a massive scale parameter of $100$.

For a given set of hyperparameters $\alpha$ and $\beta$, we first evaluate the performance of several popular, existing generating functions $\psi(x)$, serving as baselines. By varying the contamination proportion $\epsilon$, we identify the empirical breakdown point of these baseline estimators (denoted as $\epsilon_{\mathrm{base}}$) and calculate their corresponding asymptotic variance ($v_{\mathrm{base}}$). Next, to demonstrate the power of the optimality results of Theorems~\ref{thm:gab-optimality-discrete} and~\ref{thm:gab-optimality-beta0}, we consider the optimal generating function $\psi^\ast$ that minimizes the asymptotic variance while guaranteeing a target breakdown point $\epsilon^\ast$. To ensure a fair comparison, we set our target robustness constraint to be at least as strong as the baseline estimator ($\epsilon^\ast \ge \epsilon_{\mathrm{base}}$). Using the inputs $\alpha, \beta$ and $\epsilon_{\mathrm{base}}$, we obtain this optimal estimator and record its empirical breakdown point ($\epsilon_{\mathrm{opt}}$) and asymptotic variance ($v_{\mathrm{opt}}$). As shown in Table~\ref{tab:optimal-improvement}, the theoretical predictions align with experimental evidence. The optimal $\psi^\ast$-function strictly improves the efficiency (lower variance) of the estimator while maintaining, and frequently improving, the empirical robustness guarantee. Furthermore, because $\alpha$ and $\beta$ remain unchanged, this optimal construction retains the computational tractability of the baseline approaches.

\begin{table}[tbp]
    \centering
    \begin{tabular}{cc|rrr|rrr}
        \toprule
        $\alpha$ & $\beta$ & $\psi_{\mathrm{base}}(x)$ & $\epsilon_{\mathrm{base}}$ & $v_{\mathrm{base}}$ & $\phi^\ast$ & $\epsilon_{\mathrm{opt}}$ & $v_{\mathrm{opt}}$ \\
        \midrule
        $0.5$    & $0.5$   & $x$                       & $0.05$                     & $0.5$               & $27.027$    & $0.05$                    & $0.499$            \\
        $0.5$    & $1$     & $x$                       & $0.25$                     & $0.645$             & $1.409$     & $0.4$                     & $0.599$            \\
        $1$      & $1$     & $x$                       & $0.35$                     & $0.803$             & $1.609$     & $0.5$                     & $0.605$            \\
        \midrule
        $0.5$    & $0.5$   & $\log(x)$                 & $0.05$                     & $0.5$               & $27.027$    & $0.5$                     & $0.499$            \\
        $0.5$    & $1$     & $\log(x)$                 & $0.05$                     & $0.78$              & $7.905$     & $0.5$                     & $0.243$            \\
        $1$      & $1$     & $\log(x)$                 & $0.05$                     & $1.428$             & $13.513$    & $0.5$                     & $0.047$            \\
        \midrule
        $0.5$    & $0.5$   & $\log(0.7 + 0.3x)$        & $0.05$                     & $0.5$               & $27.027$    & $0.05$                    & $0.499$            \\
        $0.5$    & $1$     & $\log(0.7+0.3x)$          & $0.15$                     & $0.665$             & $2.495$     & $0.5$                     & $0.501$            \\
        $1$      & $1$     & $\log(0.7 + 0.3x)$        & $0.15$                     & $0.844$             & $4.265$     & $0.5$                     & $0.244$            \\
        \bottomrule
    \end{tabular}
    \caption{Comparison of baseline minimum divergence estimators against their theoretically optimal counterparts under a Gamma scale-estimation problem. For fixed hyperparameters $\alpha$ and $\beta$, the baseline estimators (using standard generating functions $\psi$) yield an empirical breakdown point $\epsilon_{\mathrm{base}}$ and asymptotic variance $v_{\mathrm{base}}$. Using $\alpha$, $\beta$, and a target breakdown constraint as $\epsilon_{\mathrm{base}}$, we derive the optimal tuning parameter $\phi^\ast$. The resulting optimal estimator, yields the new metrics $\epsilon_{\mathrm{opt}}$ and $v_{\mathrm{opt}}$, with lower variance while maintaining the same or improved asymptotic breakdown point. $\alpha + \beta = 1$ always lead to same minimum divergence estimator irrespective of $\psi$.}
    \label{tab:optimal-improvement}
\end{table}

\vspace*{-0.2cm}
\section{Conclusion}\label{sec:conclusion}

In this paper, we have provided a resolution to the ambiguity surrounding the selection of divergence measures in robust statistical inference. By unifying existing discrepancy measures under the Generalized Alpha-Beta Divergence (GABD) class, we mathematically characterized the trade-off between estimation efficiency under the model and global robustness under deviations from the model. The variance-minimizing estimator, universally belongs to the extended $(\phi,\gamma)$-divergence family, agnostic to the underlying parametric model and the structure of data contamination. As discussed in Section~\ref{sec:hyperparameter}, a natural avenue for future research would be to extend this optimality result over the choice of hyperparameters $\alpha$ and $\beta$, with a mathematical guarantee.

\appendix

\section{Additional Technical Details}\label{appendix:additional-calculation}

\subsection{Equivalence of asymptotic breakdown point condition for GAB divergence and S-divergence}\label{appendix:equivalence-bp}

Note that, if $\psi(x) = x$ and $g = f_{\bb{\theta}^g}$, then the breakdown condition~\eqref{cond:gab-breakdown-general} reduces to
\begin{equation*}
    (1-\epsilon)^\beta \int g^{\alpha+\beta} - \epsilon^\beta \int f_{\bb{\theta}_m}^\alpha k_m^\beta > \dfrac{\alpha}{\alpha+\beta} \int g^{\alpha+\beta} - \dfrac{\alpha}{\alpha+\beta} \int f_{\bb{\theta}_m}^{\alpha+\beta}.
\end{equation*}
\noindent Letting $A = \beta$ and $B = \alpha = (1 + a - A)$, the S-divergence between $\epsilon k_m$ and $f_{\bb{\theta}_m}$ satisfy
\begin{align*}
        & S_{(a,\lambda)}(\epsilon k_m, f_{\bb{\theta}_m})                                                                                                                                                                 \\
    ={} & \frac{1}{A}\int f_{\bb{\theta}_m}^{1+a} - \frac{1+a}{AB} \epsilon^A \int f_{\bb{\theta}_m}^B k_m^A + \frac{1}{B} \int k_m^{1+a}                                                                                  \\
    >{} & \dfrac{1}{A}\int f_{\bb{\theta}_m}^{1+a} + \dfrac{1}{B} \int k_m^{1+a} + \dfrac{1+a}{AB}\left[ \left( \dfrac{B}{A+B} - (1-\epsilon)^A \right) \int g^{A+B} - \dfrac{B}{A+B} \int f_{\bb{\theta}_m}^{A+B} \right] \\
    ={} & \dfrac{1}{B} \int k_m^{1+a} + \left[ \dfrac{1}{A} - \dfrac{1+a}{AB}(1-\epsilon)^A \right]\int g^{1+a}.
\end{align*}
\noindent Taking the limit inferior to both sides of this inequality with respect to $m$ translates to the condition (3.1) of~\cite{roy2023breakdown}.

\subsection{Form of Optimal Asymptotic Variance in Different Estimation Problems}\label{appendix:av-table}

Table~\ref{tab:optimal-asymptotic-variance} shows the form of the asymptotic variance for the optimal choice of $\psi$ function i.e., $\psi(x)=\phi^{-1}(x^\phi-1)$ with $\phi^\ast = -\log(1+\beta/\alpha)/(\beta\log(1-\epsilon^*))$ for some different parametric families. Towards enhancing readability, the following notations have been adopted.
\begin{equation*}
    a=\alpha+\beta,~t_{a}=a(t-1)+1,~t_{2a-1}=(2a-1)(t-1)+1.
\end{equation*}
\noindent Furthermore, $\Gamma_1(\cdot)$ and $\Gamma_2(\cdot)$ represent the digamma and the trigamma functions respectively i.e., the first and the second derivatives of $\log \Gamma(\cdot)$, and the constants $C_{r}(t,\theta)$ and $C_{r}(t)$ stand for
\begin{equation}
    C_{r}(t,\theta)=\frac{\theta^{r-1}}{r^{r(t-1)+1}}\frac{\Gamma(r(t-1)+1)}{\Gamma(t)^r}, \
    C_r(t) = C_r(t, 1).
    \label{eqn:C-r-function}
\end{equation}
It is apparent that for the location family of the normal distribution, $\phi$ has no effect on the asymptotic variance, whereas for all the other cases, the asymptotic variance decreases with an increase in $\phi$, which itself is a decreasing function of the breakdown point $\epsilon^\ast$. Thus, higher robustness in the form of breakdown point comes with the cost of higher asymptotic variance and hence, lower asymptotic efficiency. Such a trade-off is ubiquitous in robust estimation literature.

\begin{table}[htbp]
    \centering
    \begin{tabular}{ll}
        \toprule
        \textbf{Family of distributions} & \textbf{Optimal Asymptotic Variance}                                                                                                                                                                                                                      \\
        \midrule
        Normal Location:
        $\normdist(\theta,1)$
                                         &
        $
            \frac{1}{\beta^2}
            \left(\frac{a^2}{2a-1}\right)^{3/2}
        $                                                                                                                                                                                                                                                       \\
        \midrule
        Normal Scale:
        $\normdist(0,\sigma^2)$
                                         &
        $
            \frac{\sigma^2}{\beta^2} \left(\phi^\ast\left(1-\frac{1}{a}\right)^2+2\right)^{-2} \left(\frac{a}{\sqrt{2a-1}}\left[\left(1-\frac{1}{a}\right)^2+2\right]-\left(1-\frac{1}{a}\right)^2\right)
        $                                                                                             \\
        \midrule
        Gamma Scale:
        $\gamma(t,\theta)$
                                         &
        $\frac{\theta^2}{\beta^2}\frac{\left(\frac{C_{2a-1}(t,\theta)}{C_a(t,\theta)^2}\left(\left(1-\frac{1}{2a-1}\right)^2+\frac{(t-1)}{(2a-1)}+\frac{1}{(2a-1)^2}\right)-\left(1-\frac{1}{a}\right)^2\right)}{\left(\phi^\ast \left(1-\frac{1}{a}\right)^2+\frac{(t-1)}{a}+\frac{1}{a^2}\right)^2}
        $ \\
        \midrule
        Gamma Shape:
        $\gamma(t,1)$
                                         &
        $\frac{1}{\beta^2}\frac{\left[
                    \frac{C_{2a-1}(t)}{C_a(t)^2}
                    \left(
                    \Gamma_2(t_{2a-1})
                    +
                    \left(
                    \Gamma_1(t_{2a-1}) - \Gamma_1(t) - \ln (2a-1)
                    \right)^2
                    \right)
                    -
                    \left(
                    \Gamma_1(t_a) - \Gamma_1(t) - \ln a
                    \right)^2
                    \right]}{
                \left[
                    \phi^\ast \left(
                    \Gamma_1(t_a) - \Gamma_1(t) - \ln a
                    \right)^2
                    + \Gamma_2(t_a)
                    \right]^{2}}$
        \\
        \bottomrule
    \end{tabular}
    \caption{Optimal asymptotic variance for several parametric families. $a = (\alpha+\beta), t_a = a(t-1)+1$, $C_a(t,\theta)$ and $C_a(t)$ are as in Eq.~\eqref{eqn:C-r-function}.}
    \label{tab:optimal-asymptotic-variance}
\end{table}

\pagebreak
\section{Proofs of the results}\label{appendix:proofs}

\subsection{Proof of Proposition~\ref{prop:generic-normality}}\label{appendix-proof:generic-normality}

\begin{proof}
    We begin with the proof for consistency first. The proof uses similar ideas borrowed from~\cite{ghosh2015asymptotic} and~\cite{ghosh2017minimum}. We will show that for every small $a > 0$, with probability tending to $1$,
    \begin{equation*}
        H_n^{(\alpha,\beta),\psi}(\bb{\theta}) > H_n^{(\alpha,\beta),\psi}(\bb{\theta}^g),
    \end{equation*}
    \noindent for all $\bb{\theta}$ lying on the surface of a hypersphere $Q_a$ centered at $\bb{\theta}^g$ with radius $a$. This shows that the objective function $H_n^{(\alpha,\beta),\psi}$ has a local minimum within the hypersphere $Q_a$. At these local minima, the estimating equation $\nabla H_n^{(\alpha,\beta),\psi}(\bb{\theta}) = \bb{0}_p$ must be satisfied. Since the above holds for every small $a > 0$, this will show the existence of the sequence of solutions $\bbhat{\theta}_n^{(\alpha, \beta), \psi}$ converging to $\bb{\theta}^g$.

    Let us consider the Taylor series expansion of $H_n^{(\alpha,\beta),\psi}$ around $\bb{\theta}^g$ for a $\bb{\theta} \in Q_a$. The smoothness of $\psi$ ensures that the expansion is valid.
    \begin{align}
            & H_n^{(\alpha,\beta),\psi}(\bb{\theta})
        - H_n^{(\alpha,\beta),\psi}(\bb{\theta}^g) \nonumber                                                                                                                                                                                                                                              \\
        ={} & (\bb{\theta} - \bb{\theta}^g)\tr \nabla H_n^{(\alpha,\beta),\psi}(\bb{\theta})\vert_{\bb{\theta} = \bb{\theta}^g} + \frac{1}{2} (\bb{\theta} - \bb{\theta}^g)\tr \nabla^2 H_n^{(\alpha,\beta),\psi}(\bb{\theta})\vert_{\bb{\theta} = \bb{\theta}^g} (\bb{\theta} - \bb{\theta}^g) \nonumber \\
            & \qquad + \dfrac{1}{6} \sum_{i,j,k} (\theta_i - \theta_i^g)(\theta_j - \theta_j^g)(\theta_k - \theta_k^g) \frac{\partial^3}{\partial \theta_i\theta_j\theta_k} H_n^{(\alpha,\beta),\psi}(\bb{\theta})\vert_{\bb{\theta} = \bb{\theta}^\ast} \nonumber                                        \\
            & = S_1 + \frac{1}{2}S_2 + \frac{1}{6}S_3, \label{eqn:consistency-proof-taylor-series}
    \end{align}
    \noindent where $\bb{\theta}^\ast$ lies on the line joining $\bb{\theta}$ and $\bb{\theta}^g$, and $S_1, S_2$ and $S_3$ are the appropriate sums they are replacing. Note that due to Assumption~\ref{assum:H-first-diff},
    \begin{equation*}
        \nabla H_n^{(\alpha,\beta),\psi}(\bb{\theta})\vert_{\bb{\theta} = \bb{\theta}^g} \to 0,
    \end{equation*}
    \noindent as $n \to \infty$ in probability. For any given $a > 0$, therefore, for all sufficiently large $n$ with probability tending to $1$, we can make this quantity less than $a^2$. Since $\bb{\theta} \in Q_a$, an application of Cauchy-Schwartz inequality now yields
    \begin{equation*}
        \vert (\bb{\theta} - \bb{\theta}^g)\tr \nabla H_n^{(\alpha,\beta),\psi}(\bb{\theta})\vert_{\bb{\theta} = \bb{\theta}^g} \vert
        \leq \Vert \bb{\theta} - \bb{\theta}^g \Vert \Vert \nabla H_n^{(\alpha,\beta),\psi}(\bb{\theta})\vert_{\bb{\theta} = \bb{\theta}^g} \Vert < \sqrt{p}a^3
    \end{equation*}

    For the second term, note that due to Assumption~\ref{assum:H-second-diff}, the minimum eigenvalue of $\bb{J}(\bb{\theta}^g)$ is strictly positive. Let us denote this as $\delta > 0$. Then, we have
    \begin{align}
        S_2
         & = (\bb{\theta} - \bb{\theta}^g)\tr \nabla^2 H_n^{(\alpha,\beta),\psi}(\bb{\theta})\vert_{\bb{\theta} = \bb{\theta}^g} (\bb{\theta} - \bb{\theta}^g)\nonumber                                                                                                                                  \\
         & \geq (\bb{\theta} - \bb{\theta}^g)\tr \bb{J}(\bb{\theta}^g) (\bb{\theta} - \bb{\theta}^g) - \vert (\bb{\theta} - \bb{\theta}^g)\tr (\nabla^2 H_n^{(\alpha,\beta),\psi}(\bb{\theta})\vert_{\bb{\theta} = \bb{\theta}^g} - \bb{J}(\bb{\theta}^g)) (\bb{\theta} - \bb{\theta}^g) \vert \nonumber \\
         & \ge \delta \Vert \bb{\theta} - \bb{\theta}^g\Vert^2 - \Vert \nabla^2 H_n^{(\alpha,\beta),\psi}(\bb{\theta})\vert_{\bb{\theta} = \bb{\theta}^g} - \bb{J}(\bb{\theta}^g) \Vert \Vert \bb{\theta}-\bb{\theta}^g\Vert^2\label{eqn:consistency-proof-1}                                            \\
         & = \delta a^2 - p^2 a^3 \label{eqn:consistency-proof-2}
    \end{align}
    \noindent Here, in Eq.~\eqref{eqn:consistency-proof-1}, we use the Triangle Inequality. In Eq.~\eqref{eqn:consistency-proof-2}, we make use of the fact that the quadratic form is bounded by the maximum eigenvalue of $\nabla^2 H_n^{(\alpha,\beta),\psi}(\bb{\theta})\vert_{\bb{\theta} = \bb{\theta}^g} - \bb{J}(\bb{\theta}^g)$, which is further bounded by its Frobenius norm and the Assumption~\ref{assum:H-second-diff} to bound each element of the matrix $\nabla^2 H_n^{(\alpha,\beta),\psi}(\bb{\theta})\vert_{\bb{\theta} = \bb{\theta}^g} - \bb{J}(\bb{\theta}^g)$ by $a$ by choosing sufficiently large $n$.

    Finally, for the third term $S_3$, note that $\frac{\partial^3}{\partial \theta_i\theta_j\theta_k}H_n^{(\alpha,\beta),\psi}(\bb{\theta})$ is bounded due to Assumption~\ref{assum:H-third-diff}. Hence, an application of DCT yields,
    \begin{align*}
        \vert S_3\vert
         & = \sum_{i,j,k} \vert (\theta_i - \theta_i^g)(\theta_j - \theta_j^g)(\theta_k - \theta_k^g)\vert  \left\vert \frac{\partial^3}{\partial \theta_i\theta_j\theta_k}H_n^{(\alpha,\beta),\psi}(\bb{\theta})\right\vert \\
         & \leq \sum_{i,j,k} C\vert (\theta_i - \theta_i^g)(\theta_j - \theta_j^g)(\theta_k - \theta_k^g)\vert \leq Cp^3 a^3
    \end{align*}
    \noindent for some constant $C > 0$.

    Putting everything back into Eq.~\eqref{eqn:consistency-proof-taylor-series}, we get that as $n \to \infty$, with probability tending to $1$, we have
    \begin{equation*}
        H_n^{(\alpha,\beta),\psi}(\bb{\theta})
        - H_n^{(\alpha,\beta),\psi}(\bb{\theta}^g)
        \geq -\sqrt{p}a^3 + \delta a^2 - p^2a^3 - Cp^3 a^3
        = \delta a^2 \left( 1 - a\frac{(Cp^3 + p^2 + \sqrt{p})}{\delta} \right).
    \end{equation*}
    \noindent Therefore, choosing any $a < \delta/(Cp^3 + p^2 + \sqrt{p})$ ensures that $H_n^{(\alpha,\beta),\psi}(\bb{\theta}) - H_n^{(\alpha,\beta),\psi}(\bb{\theta}^g) > 0$ for all $\bb{\theta}$ on the surface of $Q_a$. This completes the proof of consistency.

    \medskip
    To prove asymptotic normality, we consider a Taylor series approximation of the estimating equation, namely for every $i = 1, 2, \dots, p$,
    \begin{multline*}
        \nabla_i H_n^{(\alpha,\beta),\psi}(\bb{\theta})
        = \nabla_i H_n^{(\alpha,\beta),\psi}(\bb{\theta}^g) + \sum_{j} (\theta_j - \theta_j^g) \nabla^2_{ij} H_n^{(\alpha,\beta),\psi}(\bb{\theta})\vert_{\bb{\theta} = \bb{\theta}^g} \\
        + \dfrac{1}{2}\sum_{jk} (\theta_j - \theta_j^g)(\theta_k - \theta_k^g) \nabla^3_{ijk} H_n^{(\alpha,\beta),\psi}(\bb{\theta})\vert_{\bb{\theta} = \bb{\theta}'}
    \end{multline*}
    \noindent for some $\bb{\theta}'$ lying on the line joining $\bb{\theta}$ and $\bb{\theta}^g$. Since $\bbhat{\theta}_n$ is a solution of the estimating equation, plugging $\bb{\theta} = \bbhat{\theta}_n$ in the above equation and multiplying both sides by $\sqrt{n}$ yields
    \begin{multline}
        -\sqrt{n}\nabla_i H_n^{(\alpha,\beta),\psi}(\bb{\theta}^g)\\
        = \sqrt{n} \sum_{j} (\hat{\theta}_{nj} - \theta_j^g) \left[ \nabla^2_{ij} H_n^{(\alpha,\beta),\psi}(\bb{\theta})\vert_{\bb{\theta} = \bb{\theta}^g} + \sum_{k} \frac{( \hat{\theta}_{nk} - \theta_k^g)}{2} \nabla^3_{ijk} H_n^{(\alpha,\beta),\psi}(\bb{\theta})\vert_{\bb{\theta} = \bb{\theta}'} \right]
        \label{eqn:estimating-eqn-taylor}
    \end{multline}
    \noindent We can also rewrite the above as a system of linear equations given by
    \begin{equation*}
        Y_{in} = \sum_{j=1}^p A_{ijn} T_{jn}
    \end{equation*}
    \noindent where
    \begin{align*}
        Y_{in}  & := -\sqrt{n}\nabla_i H_n^{(\alpha,\beta),\psi}(\bb{\theta}^g)                                                                                                                                                                                             \\
        A_{ijn} & :=  \left[ \nabla^2_{ij} H_n^{(\alpha,\beta),\psi}(\bb{\theta})\vert_{\bb{\theta} = \bb{\theta}^g} + \sum_{k} \frac{( \hat{\theta}_{nk} - \theta_k^g)}{2} \nabla^3_{ijk} H_n^{(\alpha,\beta),\psi}(\bb{\theta})\vert_{\bb{\theta} = \bb{\theta}'} \right] \\
        T_{jn}  & := \sqrt{n}(\hat{\theta}_{nj} - \theta_j^g)
    \end{align*}
    \noindent As $n \to \infty$, the first term of the sum in $A_{ijn}$ converges to the $(i,j)$-th entry of $\bb{J}(\bb{\theta}^g)$ due to Assumption~\ref{assum:H-second-diff}, and the second term of the sum is $o_p(1)$ due to the consistency proved above.

    Also, due to Assumption~\ref{assum:H-first-diff}, $\sqrt{n}\nabla H_n^{(\alpha,\beta),\psi}(\bb{\theta}^g)$ converges to a $p$-dimensional normal random variable with mean $\bb{0}$ and variance-covariance matrix given by $\bb{K}(\bb{\theta}^g)$. The asymptotic normality of $\sqrt{n}(\bbhat{\theta}_n - \bb{\theta}^g)$ now follows from an application of Theorem 5.3 of~\cite{Lehmann2006tpe}.
\end{proof}

\subsection{Proof of Corollary~\ref{thm:normality-beta-1}}\label{appendix-proof:normality-beta-1}

\begin{proof}
    In view of Proposition~\ref{prop:generic-normality}, it is enough to show that Assumptions~\ref{assum:H-first-diff} and \ref{assum:H-second-diff} hold. To this end, we begin by finding the expressions for $\nabla H_n^{(\alpha,1),\psi}(\bb{\theta})$ and $\nabla^2 H_n^{(\alpha,1),\psi}(\bb{\theta})$ first. Let $\bb{u}_{\bb{\theta}}(x) := \frac{\partial}{\partial \bb{\theta}} \log(f_{\bb{\theta}}(x))$ be the score function. Then, some standard calculations yield,
    \begin{equation}
        \nabla H_n^{(\alpha,\beta),\psi}(\bb{\theta})
        = \psi'(\norm{f_{\bb{\theta}}}_{\alpha+1}^{\alpha+1}) P_{\alpha+1,0} - \psi'( \bar{f}^{(\alpha)}_{\bb{\theta}}  ) \hat{P}_{\alpha,1},
        \label{eqn:H-beta1-empirical-diff}
    \end{equation}
    \noindent and,
    \begin{multline}
        \nabla^2 H_n^{(\alpha,1),\psi}(\bb{\theta})
        = (\alpha + 1) \psi''(\norm{f_{\bb{\theta}}}_{\alpha+1}^{\alpha+1})  P_{\alpha+1,0}P_{\alpha+1,0}\tr - \alpha  \psi''(\bar{f}^{(\alpha)}_{\bb{\theta}} ) \hat{P}_{\alpha,1}\hat{P}_{\alpha,1}\tr\\
        + (\alpha + 1) \psi'(\norm{f_\theta}_{\alpha+1}^{\alpha+1}) Q_{\alpha+1, 0} - \alpha\psi'(\bar{f}^{(\alpha)}_{\bb{\theta}} ) \hat{Q}_{\alpha,1} \\
        - \psi'(\norm{f_\theta}_{\alpha+1}^{\alpha+1}) R_{\alpha+1,0} + \psi'(\bar{f}^{(\alpha)}_{\bb{\theta}} ) \hat{R}_{\alpha,1},
        \label{eqn:H-beta1-empirical-diff-2}
    \end{multline}
    \noindent where $\bar{f}^{(\alpha)}_{\bb{\theta}} = n^{-1} \sum_{i=1}^n f_{\bb{\theta}}^\alpha(X_i)$. An application of the Law of Large Numbers (LLN) and the continuous mapping theorem shows that
    \begin{equation*}
        \psi'(\bar{f}^{(\alpha)}_{\bb{\theta}} ) \to \psi'(\inner{f_{\bb{\theta}},g}_{\alpha,1}), \ \text{ and, } \psi''(\bar{f}^{(\alpha)}_{\bb{\theta}} ) \to \psi''(\inner{f_{\bb{\theta}}, g}_{\alpha, 1}),
    \end{equation*}
    \noindent as $n \to \infty$, where the above convergence is in probability. Similarly, one can also verify that as $n \to \infty$, for any $\alpha \in [0, 1]$, $\hat{P}_{\alpha,1} \to P_{\alpha,1}$, $\hat{Q}_{\alpha,1} \to Q_{\alpha,1}$ and $\hat{R}_{\alpha,1} \to R_{\alpha, 1}$. Therefore, it follows that $\nabla^2 H_n^{(\alpha,1),\psi}(\bb{\theta})\vert_{\bb{\theta} = \bb{\theta}^g} \to \bb{J}^{(\alpha,1),\psi}(\bb{\theta}^g)$ in probability as $n \to \infty$, establishing Assumption~\ref{assum:H-second-diff}.

    To see how Assumption~\ref{assum:H-first-diff} follows through, note that an application of the standard Levy-Lindeberg CLT implies that
    \begin{equation*}
        \dfrac{1}{\sqrt{n}}\sum_{i=1}^n f_{\bb{\theta}}^\alpha (X_i) u_{\bb{\theta}}(X_i) \to_d \normdist\left( P_{\alpha,1}, Q_{2\alpha,1} - P_{\alpha,1}P_{\alpha,1}\tr \right).
    \end{equation*}
    \noindent Since, $\psi'(\bar{f}^{(\alpha)}_{\bb{\theta}} ) \to \psi'(\inner{f_{\bb{\theta}},g}_{\alpha,1})$, an application of Slutsky's theorem now yields that $\nabla H_n^{(\alpha,1),\psi}(\bb{\theta})$ has an asymptotic normal distribution as $n \to \infty$ with mean
    \begin{equation*}
        \psi'(\norm{f_{\bb{\theta}}}_{\alpha+1}^{\alpha+1}) P_{\alpha+1, 0} - \psi'(\inner{f_{\bb{\theta}}, g}_{\alpha,1})P_{\alpha, 1} =: \nabla H^{(\alpha,1)}(\bb{\theta}),
    \end{equation*}
    \noindent and variance
    \begin{equation*}
        (\psi'(\inner{f_{\bb{\theta}}, g}_{\alpha,1}))^2 (Q_{2\alpha, 1} - P_{\alpha,1}P_{\alpha,1}\tr).
    \end{equation*}
    \noindent As $\bb{\theta}^g$ is the true parameter, it satisfies $\nabla H^{(\alpha,1),\psi}(\bb{\theta})\vert_{\bb{\theta} = \bb{\theta}^g} = \bb{0}$. Putting $\bb{\theta} = \bb{\theta}^g$ in the above expression for variance yields $\bb{K}^{(\alpha,1),\psi}(\bb{\theta}^g)$. Hence, Assumption~\ref{assum:H-first-diff} follows.
\end{proof}

\subsection{Proof of Corollary~\ref{thm:normality-discrete}}\label{appendix-proof:normality-discrete}

\begin{proof}
    The proof follows along the lines of~\cite{ghosh2015asymptotic} and~\cite{maji2016logarithmic}. We only present the key steps of the proof, and refer the interested reader to consult these references for the intricate technical details. Essentially, we show that Assumptions~\ref{assum:discrete-J-pd}-\ref{assum:discrete-ratio-bound} together imply Assumptions~\ref{assum:H-first-diff}-\ref{assum:H-second-diff} for the discrete setup, and then one can make use of Proposition~\ref{prop:generic-normality} to establish the consistency and asymptotic normality results.

    Given the Assumptions~\ref{assum:identifiable}-\ref{assum:H-third-diff} and \ref{assum:discrete-J-pd}-\ref{assum:discrete-ratio-bound}, an application of Lemma 4 of~\cite{ghosh2015asymptotic} yields that
    \begin{equation}
        \sqrt{n}(\hat{P}_{\alpha,\beta} - P_{\alpha,\beta}) \to_d \normdist\left( 0, \beta^2 \left( Q_{2\alpha,2\beta - 1} - P_{\alpha,\beta} P_{\alpha,\beta}\tr \right) \right)
        \label{eqn:discrete-normal-proof-1}
    \end{equation}
    \noindent where $P_{\alpha,\beta}$ and $\hat{P}_{\alpha,\beta}$ are as given in Eq.~\eqref{eqn:P-Q-R-defn} and~\eqref{eqn:P-Q-R-hat-defn} respectively.

    Also, due to Assumption~\ref{assum:discrete-hellinger-bound}, we have $\sum_{x} g^{1/2}(x) f_{\bb{\theta}}^{\alpha+\beta - 1} < \infty$, and hence minor modifications of Lemmas 3 and 4 of~\cite{ghosh2015asymptotic} establishes that
    \begin{equation*}
        \sqrt{n}\left( \sum_{x} f_{\bb{\theta}}^{\alpha + \beta}(x) \delta_n^{\beta}(x) - \sum_{x} f_{\bb{\theta}}^{\alpha + \beta}(x) \delta_g^{\beta}(x) \right) \to_d \normdist\left( 0, V \right),
    \end{equation*}
    \noindent where $V$ is an approximate variance term free of the sample size $n$. However, this also means,
    \begin{equation*}
        \inner{f_{\bb{\theta}}, \delta_n}_{\alpha+\beta, \beta} - \inner{f_{\bb{\theta}}, \delta_g}_{\alpha+\beta, \beta} \to_\prob 0,
    \end{equation*}
    \noindent as $n \to \infty$. Since, $\psi \in C^2((0, \infty))$, $\psi'$ is continuous on $(0, \infty)$, and hence by the continuous mapping theorem, it follows that as $n \to \infty$,
    \begin{equation}
        \psi'\left( \inner{f_{\bb{\theta}}, \delta_n}_{\alpha+\beta, \beta}\right) - \psi'\left( \inner{f_{\bb{\theta}}, \delta_g}_{\alpha+\beta, \beta}\right) \to_\prob 0.
        \label{eqn:discrete-normal-proof-2}
    \end{equation}
    \noindent Now, consider the expression for $\nabla H_{n,disc}^{(\alpha,\beta),\psi}(\bb{\theta})$ from Eq.~\eqref{eqn:discrete-H-diff}, and decompose it as,
    \begin{equation*}
        \sqrt{n}\nabla H_{n,disc}^{(\alpha,\beta),\psi}(\bb{\theta})
        = \sqrt{n} \nabla H^{(\alpha,\beta),\psi}(\bb{\theta}) - \dfrac{\sqrt{n}}{\beta} \left( \psi'(\inner{f_{\bb{\theta}}, \delta_n}_{\alpha+\beta, \beta}) \hat{P}_{\alpha,\beta} - \psi'(\inner{f_{\bb{\theta}}, \delta_g}_{\alpha+\beta, \beta}) P_{\alpha,\beta} \right).
    \end{equation*}
    \noindent The first term on the right-hand side is equal to $0$ at $\bb{\theta} = \bb{\theta}^g$, by definition of the best fitting parameter $\bb{\theta}^g$. To deal with the second term, we combine the limit results from Eq.~\eqref{eqn:discrete-normal-proof-1} and~\eqref{eqn:discrete-normal-proof-2}. Putting everything together, we get
    \begin{equation*}
        \sqrt{n}\nabla H_{n,disc}^{(\alpha,\beta),\psi}(\bb{\theta})\vert_{\bb{\theta} = \bb{\theta}^g}
        \to_d \normdist\left( 0, \left( \psi'\left( \inner{f_{\bb{\theta}^g}, \delta_g}_{\alpha+\beta, \beta} \right)\right)^2 (Q_{2\alpha, 2\beta - 1} - P_{\alpha,\beta} P_{\alpha,\beta}\tr) \right),
    \end{equation*}
    \noindent where the terms $P_{\alpha,\beta}$ and $Q_{\alpha,\beta}$ are as defined in Eq.~\eqref{eqn:P-Q-R-defn} but evaluated at $\bb{\theta} = \bb{\theta}^g$. This delivers Assumption~\ref{assum:H-first-diff}.

    Now, we move on to establishing Assumption~\ref{assum:H-second-diff}. For this, note that
    \begin{multline*}
        \beta \nabla^2 H_{n,disc}^{(\alpha,\beta),\psi}(\bb{\theta})
        = (\alpha + \beta)   \psi''(\norm{f_{\bb{\theta}}}_{\alpha+\beta}^{\alpha+\beta})  P_{\alpha+\beta,0}P_{\alpha+\beta,0}\tr + \psi'(\norm{f_\theta}_{\alpha+\beta}^{\alpha+\beta}) ((\alpha+\beta) Q_{\alpha+\beta, 0} - R_{\alpha+\beta, 0}) \\
        - \alpha  \psi''( \inner{f_{\bb{\theta}}, r_n}_{\alpha,\beta}  ) \hat{P}_{\alpha,\beta}\hat{P}_{\alpha,\beta}\tr - \psi'(\inner{f_{\bb{\theta}}, r_n}_{\alpha,\beta} ) (\alpha \hat{Q}_{\alpha,\beta} - \hat{R}_{\alpha,\beta})
    \end{multline*}

    \noindent As before, due to Assumptions~\ref{assum:H-third-diff} and~\ref{assum:discrete-ratio-bound}, we have the in-probability convergence of various terms in the above expression due to a combination of the Law of Large Numbers and Dominated Convergence Theorem (DCT). Namely, we have
    \begin{align*}
        \inner{f_{\bb{\theta}}, r_n}_{\alpha,\beta} & = \inner{f_{\bb{\theta}}, \delta_n}_{\alpha+\beta,\beta} \to_\prob \inner{f_{\bb{\theta}}, \delta_g}_{\alpha+\beta,\beta} = \inner{f_{\bb{\theta}}, g}_{\alpha,\beta} \\
        \hat{P}_{\alpha,\beta}                      & \to_\prob P_{\alpha,\beta}, \
        \hat{Q}_{\alpha,\beta} \to_\prob Q_{\alpha,\beta}, \
        \hat{R}_{\alpha,\beta} \to_\prob R_{\alpha,\beta}.
    \end{align*}
    \noindent Interested readers may consult the proof of Theorem 1 of~\cite{ghosh2015asymptotic} for the technical details. Again, since $\psi \in C^2((0, \infty))$, both $\psi'(x)$ and $\psi''(x)$ are continuous. Hence, by an application of the continuous mapping theorem, we get that
    \begin{equation*}
        \nabla^2 H_{n,disc}^{(\alpha,\beta),\psi}(\bb{\theta}) \to_\prob \frac{1}{\beta}\bb{J}^{(\alpha,\beta)}(\bb{\theta}^g),
    \end{equation*}
    \noindent as $n \to \infty$. Corollary~\ref{thm:normality-discrete} is now immediate due to Proposition~\ref{prop:generic-normality}.
\end{proof}

\subsection{Proof of Corollary~\ref{thm:normality-discrete-beta0}}\label{appendix-proof:normality-discrete-beta0}

\begin{proof}
    The proof follows an almost identical line of reasoning as the proof of Corollary~\ref{thm:normality-discrete} as illustrated in Appendix~\ref{appendix-proof:normality-discrete}. We only show a few key steps.

    First note that, the objective function when $\beta = 0$ is given by
    \begin{equation*}
        H_{n,disc}^{(\alpha,0),\psi}(\bb{\theta}) = \frac{1}{\alpha^2} \left( \psi'(\norm{f_{\bb{\theta}}}_\alpha^\alpha) \sum_{x \in \chi} f_{\bb{\theta}}^\alpha(x) \log(f_{\bb{\theta}}^\alpha(x) / r_n^\alpha(x) ) - \psi(\norm{f_{\bb{\theta}}}_\alpha^\alpha) \right).
    \end{equation*}
    \noindent As a result, the first-order derivative of the objective function turns out to be
    \begin{multline*}
        \nabla_{\bb{\theta}} H_{n,disc}^{(\alpha,0),\psi}(\bb{\theta}) = \psi''(\norm{f_{\bb{\theta}}}_\alpha^\alpha) P_{\alpha,0} \int f_{\bb{\theta}}^\alpha \log(f_{\bb{\theta}}) + \psi'(\norm{f_{\bb{\theta}}}_\alpha^\alpha) ( \alpha^{-1} P_{\alpha, 0} +  \int f_{\bb{\theta}}^\alpha \log(f_{\bb{\theta}}) u_{\bb{\theta}}) \\
        - \psi''(\norm{f_{\bb{\theta}}}_\alpha^\alpha) P_{\alpha,0} \sum_{x\in \chi} f_{\bb{\theta}}^\alpha(x) \log(r_n(x)) -  \psi'(\norm{f_{\bb{\theta}}}_\alpha^\alpha) \sum_{x \in \chi} f_{\bb{\theta}}^\alpha(x) u_{\bb{\theta}}(x) \log(r_n(x))\\
        - \alpha^{-1} \psi''(\norm{f_{\bb{\theta}}}_\alpha^\alpha) P_{\alpha, 0}
    \end{multline*}
    \noindent To obtain $\bb{K}$-matrix, we simply need to restrict our attention to the terms involving the relative frequencies $r_n$. By an application of the delta method, it follows that
    \begin{align*}
        \sqrt{n}\begin{bmatrix}
                    \sum_{x \in \chi} f_{\bb{\theta}}^\alpha \log(r_n) - \int f_{\bb{\theta}}^\alpha \log(g) \\
                    \int f_{\bb{\theta}}^\alpha u_{\bb{\theta}} \log(r_n) - \int f_{\bb{\theta}}^\alpha u_{\bb{\theta}} \log(g)
                \end{bmatrix} \to_d N\left(\bb{0}, \begin{bmatrix}
                                                       \int f_{\bb{\theta}}^{2\alpha} g^{-1} - \norm{f_{\bb{\theta}}}_\alpha^{2\alpha}
                                                        & P_{2\alpha, -1} - P_{\alpha,0}\norm{f_{\bb{\theta}}}_\alpha^\alpha \\
                                                       P_{2\alpha, -1} - P_{\alpha,0}\norm{f_{\bb{\theta}}}_\alpha^\alpha
                                                        & Q_{2\alpha,-1} - P_{\alpha,0}P_{\alpha,0}\tr
                                                   \end{bmatrix} \right)
    \end{align*}
    \noindent By putting in this asymptotic variance form into the expression of $\nabla_{\bb{\theta}} H_{n,disc}^{(\alpha,0),\psi}(\bb{\theta})$ yields the form of $\bb{K}$-matrix as given in Eq.~\eqref{eqn:K-theta-beta0}. The form of the $\bb{J}$-matrix can be obtained by finding $\nabla^2_{\bb{\theta}} H_{n,disc}^{(\alpha,0),\psi}(\bb{\theta})$ by standard calculations. The rest of the argument is similar to Appendix~\ref{appendix-proof:normality-discrete}.
\end{proof}

\subsection{Proof of Corollary~\ref{thm:normality-continuous}}\label{appendix-proof:normality-continuous}

\begin{proof}
    The proof follows along the lines of the proof of Theorem 2 of~\cite{ghosh2017minimum}, and is very similar to the proof of Corollary~\ref{thm:normality-discrete}.

    To achieve Assumption~\ref{assum:H-first-diff}, we begin by noting that
    \begin{equation*}
        \nabla \tilde{H}_n^{(\alpha,\beta),\psi}(\bb{\theta}) = \frac{1}{\beta} \psi'\left( \norm{\tilde{f}_{\bb{\theta}}}_{\alpha+\beta}^{\alpha+\beta} \right) \int \tilde{f}_{\bb{\theta}}^{\alpha+\beta} \tilde{u}_{\bb{\theta}} - \frac{1}{\beta} \psi'\left( \int \tilde{f}_{\bb{\theta}}^\alpha \hat{g}_n^\beta \right) \int \tilde{f}_{\bb{\theta}}^{\alpha} \hat{g}_n^\beta \tilde{u}_{\bb{\theta}},
    \end{equation*}
    \noindent and, only the second term is affected by the sample observations $X_i$. Focusing on this second term, from a direct application of Lemma 6 of~\cite{ghosh2017minimum}, we obtain that
    \begin{equation*}
        \sqrt{n}\left( \int \tilde{f}_{\bb{\theta}}^{\alpha} \hat{g}_n^\beta \tilde{u}_{\bb{\theta}} - \int \tilde{f}_{\bb{\theta}}^{\alpha} \tilde{g}^\beta \tilde{u}_{\bb{\theta}}  \right) \to_d \normdist\left( 0, V(\bb{\theta}^g) \right),
    \end{equation*}
    \noindent where
    \begin{align*}
        V(\bb{\theta}) & = \var_{g}\left( \beta \int W(X, y, h) \tilde{f}_{\bb{\theta}}^{\alpha}(y) \tilde{g}^{\beta - 1}(y) \tilde{u}_{\bb{\theta}}(y) dy \right)                                                                                     \\
                       & = \beta^2 \left[ \int W^2(X, y, h) \tilde{f}_{\bb{\theta}}^{2\alpha}(y) \tilde{g}^{2\beta - 2}(y) \tilde{u}_{\bb{\theta}}(y) \tilde{u}_{\bb{\theta}}\tr (y) dy - \tilde{P}_{\alpha,\beta} \tilde{P}_{\alpha,\beta}\tr \right] \\
                       & = \beta^2 \left[ \int \tilde{f}_{\bb{\theta}}^{2\alpha}\tilde{g}^{2\beta - 2} \tilde{u}_{\bb{\theta}}(y) \tilde{u}_{\bb{\theta}}\tr (y) \nu(y) dy - \tilde{P}_{\alpha,\beta} \tilde{P}_{\alpha,\beta}\tr \right],
    \end{align*}
    \noindent where $\nu(y) = \E_g(W(X, y, h)) = \int W^2(x,y,h) g(x)dx$. Also, by the Law of Large Numbers and DCT (see Theorem 2 of~\cite{ghosh2017minimum} for more details), we have $\int \tilde{f}_{\bb{\theta}}^\alpha \hat{g}_n^\beta \to_\prob \int \tilde{f}_{\bb{\theta}}^\alpha \tilde{g}^\beta$ as $n \to \infty$. Continuity of $\psi'$ now establishes Assumption~\ref{assum:H-first-diff} for this specific setup. Assumption~\ref{assum:H-second-diff} can be verified in a similar manner, by repeatedly using the Law of Large Numbers and the continuity of $\psi'$ and $\psi''$. Turning to Proposition~\ref{prop:generic-normality} now completes the proof.
\end{proof}

\subsection{Proof of Proposition~\ref{prop:influence-function}}\label{appendix-proof:influence-function}

\begin{proof}
    Let $G_{\epsilon,y} = (1-\epsilon)G + \epsilon \Delta_y$. Let us express the population objective function given in Eq.~\eqref{eqn:H-theta-popn} as $H^{(\alpha,\beta),\psi}(\bb{\theta}, G)$ where we explicitly mention the data distribution as the second argument. By Fisher consistency of the functional, we have
    \begin{equation*}
        \nabla H^{(\alpha,\beta),\psi}(\bbhat{\theta}^{(\alpha,\beta),\psi}(G_{\epsilon, y}), G_{\epsilon, y}) = 0.
    \end{equation*}
    \noindent Now, we take total derivatives of both sides with respect to $\epsilon$ and set $\epsilon = 0$. This yields,
    \begin{multline}
        \nabla^2_{\bb{\theta}} H^{(\alpha,\beta),\psi}(\bbhat{\theta}^{(\alpha,\beta),\psi}(G), G) \IF(y; G, \bbhat{\theta}^{(\alpha,\beta),\psi}) \\
        + \langle \nabla_{\bb{G}} \nabla_{\bb{\theta}}  H^{(\alpha,\beta),\psi}(\bbhat{\theta}^{(\alpha,\beta),\psi}(G), G) , \Delta_y - G \rangle = 0.
        \label{eqn:influence-function-generic-1}
    \end{multline}
    \noindent Here, in the second term $\nabla_{\bb{G}} \nabla_{\bb{\theta}}  H^{(\alpha,\beta),\psi}$ refers to the linear functional operator acting on the tangent space of the probability distributions. The inner product $\langle \cdot, \cdot \rangle$ is the corresponding inner product defined in that tangent space. Since, $\bbhat{\theta}^{(\alpha,\beta),\psi}(G) = \bb{\theta}^g$ is the ``best fitting parameter'', in view of Assumption~\ref{assum:H-second-diff}, we have $\nabla^2_{\bb{\theta}} H^{(\alpha,\beta),\psi}(\bbhat{\theta}^{(\alpha,\beta),\psi}(G), G) = \bb{J}(\bb{\theta}^g) / \beta$, where $\bb{J}(\bb{\theta})$ is as given in Eq.~\eqref{eqn:J-theta-generic}. This enables one to use Eq.~\eqref{eqn:influence-function-generic-1} to solve for the influence curve, i.e.,
    \begin{equation}
        \IF(y; G, \bbhat{\theta}^{(\alpha,\beta),\psi})
        = -\beta \bb{J}^{-1}(\bb{\theta}^g) \langle \nabla_{\bb{G}} \nabla_{\bb{\theta}}  H^{(\alpha,\beta),\psi}(\bbhat{\theta}^{(\alpha,\beta),\psi}(G), G) , \Delta_y - G \rangle
        \label{eqn:influence-function-generic-2}
    \end{equation}
    \noindent Now, a re-examination of the estimating equation given in Eq.~\eqref{eqn:H-theta-popn} reveals that the first term is free of the data distribution $G$, hence the above linear operator only acts on the second term alone. A few lines of careful calculations now yield
    \begin{align*}
            & \langle \nabla_{\bb{G}} \nabla_{\bb{\theta}}  H^{(\alpha,\beta),\psi}(\bbhat{\theta}^{(\alpha,\beta),\psi}(G), G) , \Delta_y - G \rangle                                                            \\
        ={} & -\frac{1}{\beta} \left\langle \nabla_{\bb{G}}\left( \psi'\left( \int f_{\bb{\theta}}^\alpha g^\beta \right) \int f_{\bb{\theta}}^\alpha g^\beta u_{\bb{\theta}} \right), \Delta_y - G \right\rangle \\
        ={} & - \left[ \psi''(\inner{f_{\bb{\theta}}, g}_{\alpha,\beta}) P_{\alpha,\beta} (f_{\bb{\theta}}^\alpha(y) g^{\beta - 1}(y) - \inner{f_{\bb{\theta}}, g}_{\alpha,\beta}) \right.\\
                    & \qquad \left. +{} \psi'(\inner{f_{\bb{\theta}}, g}_{\alpha,\beta}) (f_{\bb{\theta}}^\alpha(y) g^{\beta - 1}(y) u_{\bb{\theta}}(y) - P_{\alpha,\beta}) \right]
    \end{align*}
    \noindent Plugging this expression back into Eq.~\eqref{eqn:influence-function-generic-2}, we obtain the intended expression for the influence function.
\end{proof}

\subsection{Proof of Proposition~\ref{prop:bp-consistency}}\label{appendix-proof:bp-consistency}

\begin{proof}
    For the sake of contradiction, assume that $\epsilon^\ast < \epsilon^\ast(T)$. Pick $\epsilon_0 \in (\epsilon^\ast, \epsilon^\ast(T))$. Then, by definition of the asymptotic breakdown point as given above, the sequence of estimators $\{ T_n \}_{ n \geq 1}$ must break down at $\epsilon_0$. Mathematically, there exists a sequence of contaminating distributions $\{ K_m \}_{m = 1}^\infty$ such that
    \begin{equation*}
        \inf_{\bb{\theta}_\infty \in \partial\bb{\Theta}} \liminf_{m \to \infty} \liminf_{n \to \infty} \prob_{G_{\epsilon, m}}(d(T_n,  \bb{\theta}_\infty) > 0) = 1 - \delta,
    \end{equation*}
    \noindent for some suitably chosen small $\delta > 0$. Now pick any small $\xi > 0$. Then, there exists a $\bb{\theta}_\infty, M_1$ and a sequence $n_{1m}$ such that for all $m \geq M_1$ and for all $n \geq n_{1m}$, $d(T(G_{\epsilon, m, n}), \bb{\theta}_\infty) \geq \xi/2$ with probability $(1 - \delta)$, or
    \begin{equation}
        d(T(G_{\epsilon, m, n}), \bb{\theta}_\infty) < \xi/2
        \label{eqn:bp-consistency-proof-eq1}
    \end{equation}
    \noindent with probability $\delta$.

    Also, as $\{ T_n \}_{n \geq 1}$ is consistent for the corresponding functional $T$, it follows that
    \begin{equation*}
        d(T(G_{\epsilon, m, n}), T(G_{\epsilon, m})) \to 0,
    \end{equation*}
    \noindent for any fixed $\epsilon \in [0, 1/2]$ as
    \begin{equation*}
        \vert G_{\epsilon, m, n} - G_{\epsilon, m}\vert
        = \vert (1 - \epsilon)G_n + \epsilon K_m - (1 - \epsilon)G - \epsilon K_m \vert
        \leq \vert G_n - G\vert \to 0,
    \end{equation*}
    \noindent as $n \to \infty$. Therefore, for the fixed $\epsilon = \epsilon_0$, there exists $M_2$ and a sequence $n_{2m}$ such that for all $m \geq M_2$ and $n \geq n_{2m}$,
    \begin{equation}
        d(T(G_{\epsilon_0, m, n}), T(G_{\epsilon_0, m})) < \xi / 2
        \label{eqn:bp-consistency-proof-eq2}
    \end{equation}
    \noindent with probability at least $(1 - \delta/2)$.

    By an application of Bonferroni's inequality, both Eq.~\eqref{eqn:bp-consistency-proof-eq1} and Eq.~\eqref{eqn:bp-consistency-proof-eq2} holds with probability at least $\delta/2 > 0$, whenever $m \geq \max\{ M_1, M_2\}$ and $n \geq n_m := \max\{ n_{1m}, n_{2m}\}$. Let us pick one such sample $\omega$ such that both these inequalities hold. Then, by triangle inequality, it follows that
    \begin{equation*}
        d(T(G_{\epsilon_0, m}), \bb{\theta}_\infty)
        \leq d(T(G_{\epsilon_0, m}), T(G_{\epsilon_0, m, n_m})(\omega)) + d(T(G_{\epsilon_0, m, n_m})(\omega), \bb{\theta}_\infty)
        < \xi,
    \end{equation*}
    \noindent for all $m \geq \max\{ M_1, M_2 \}$, and the particular choice of $\bb{\theta}_\infty$ and the contaminating sequence $\{ K_m \}_{m \geq 1}$. This means the asymptotic breakdown point $\epsilon^\ast(T)$ must be less than $\epsilon_0$ (by Definition given in Eq.~\eqref{eqn:bp-consistency-proof-eq2}). This contradicts the choice of $\epsilon_0$ as we picked $\epsilon_0 \in (\epsilon^\ast, \epsilon^\ast(T))$.
\end{proof}

\subsection{Proof of Proposition~\ref{thm:gab-breakdown-general}}\label{appendix-proof:breakdown-general}

\begin{proof}
    Let $\epsilon < \tilde{\epsilon}$ be a level of contamination where the breakdown occurs. This means, there exists a sequence of contaminating densities $\{ k_m\}$ such that for the corresponding $\epsilon$-contaminated densities $g_{\epsilon,m}=(1-\epsilon)g+\epsilon k_m$, the MGABD functional $\bb{\theta}_m = \bbhat{\theta}^{(\alpha,\beta),\psi}(G_{\epsilon,m})$ satisfies $\bb{\theta}_m \rightarrow \bb{\theta}_\infty$ for some $\bb{\theta}_\infty \in \partial\bb{\Theta}$.

    The proof now follows in three steps, similar to those of Theorem 3.1 of~\cite{roy2023breakdown}. In Step 1, we find an asymptotic lower bound of the GABD between $f_{\bb{\theta}_m}$ and $g_{\epsilon,m}$. Next, for a fixed $\bb{\theta}^g$ as provided by condition~\eqref{cond:gab-breakdown-general}, we find an asymptotic upper bound of the GABD between $f_{\bb{\theta}^g}$ and $g_{\epsilon, m}$. In the final step, we compare these bounds by using the minimality of GABD at $\theta_m$ and aim to recover a bound on the $\epsilon$ from it.

    \noindent\textbf{\underline{Step 1}:} Consider the set $A_m=\left\{ x: g(x)> \max\{k_m(x),f_{\theta_m}(x)\} \right\}$. Due to Assumption~\ref{assum:bp-f-sig-g}, we have $f_{\bb{\theta}_m}(x) \ind{A_m}(x) \to 0$ pointwise. Similarly, Assumption~\ref{assum:bp-g-sig-km} entails that for each $x \in \R$, we have $g_{\epsilon, m}(x)\ind{A_m^c}(x) \to \epsilon k_m(x)\ind{A_m^c}(x)$ pointwise. An application of Lemma 3.1 of~\cite{roy2023breakdown} establishes that
    \begin{align*}
        \int_{A_m} f_{\bb{\theta}_m}^\alpha g_{\epsilon,m}^\beta   & \asymp 0,                                                                                                                                  \\
        \int_{A_m^c} f_{\bb{\theta}_m}^\alpha g_{\epsilon,m}^\beta & \asymp \epsilon^\beta \int_{A_m^c} f_{\bb{\theta}_m}^\alpha k_m^\beta \asymp \epsilon^\beta \inner{f_{\bb{\theta}_m}, k_m}_{\alpha,\beta},
    \end{align*}
    \noindent as $m \rightarrow \infty$. Due to the continuity of $\psi$-function, it now follows that
    \begin{align}
               & d_{GAB}^{(\alpha,\beta),\psi}(f_{\bb{\theta}_m}, g_{\epsilon,m}) \nonumber                                                                                                                                                                                                                                                    \\
        =      & \dfrac{1}{\beta(\alpha+\beta)}\psi\left( \norm{f_{\bb{\theta}_m}}_{\alpha+\beta}^{\alpha+\beta} \right) - \dfrac{1}{\alpha\beta}\psi\left( \inner{f_{\bb{\theta}_m}, g_{\epsilon,m}}_{\alpha,\beta} \right) + \dfrac{1}{\alpha(\alpha+\beta)}\psi\left( \norm{g_{\epsilon,m}}_{\alpha+\beta}^{\alpha+\beta} \right) \nonumber \\
        \asymp & \dfrac{\psi\left( \norm{f_{\bb{\theta}_m}}_{\alpha+\beta}^{\alpha+\beta} \right)}{\beta(\alpha+\beta)} - \dfrac{\psi\left( \epsilon^\beta\inner{f_{\bb{\theta}_m}, k_m}_{\alpha,\beta} \right)}{\alpha\beta} + \dfrac{\psi\left( \norm{g_{\epsilon,m}}_{\alpha+\beta}^{\alpha+\beta} \right)}{\alpha(\alpha+\beta)}.
        \label{eqn:gab-bp-proof-step1}
    \end{align}
    \noindent We denote the right-hand side of Eq.~\eqref{eqn:gab-bp-proof-step1} by $\delta_m(\epsilon)$.

    \noindent\textbf{\underline{Step 2}:} Since $\beta > 0$, we have the inequality
    \begin{equation*}
        \inner{f_{\bb{\theta}^g}, g_{\epsilon, m}}_{\alpha,\beta}
        = \int f_{\bb{\theta}^g}^\alpha ((1 - \epsilon)g + \epsilon k_m)^\beta
        \geq (1-\epsilon)^\beta \int f_{\bb{\theta}^g}^\alpha g^\beta = (1-\epsilon)^\beta \inner{f_{\bb{\theta}^g}, g}_{\alpha,\beta}.
    \end{equation*}
    \noindent Therefore, we have
    \begin{align}
               & d_{GAB}^{(\alpha,\beta),\psi}(f_{\bb{\theta}^g}, g_{\epsilon,m})\nonumber                                                                                                                                                                                                                                                    \\
        ={}    & \dfrac{1}{\beta(\alpha+\beta)}\psi\left( \norm{f_{\bb{\theta}^g}}_{\alpha+\beta}^{\alpha+\beta} \right) - \dfrac{1}{\alpha\beta}\psi\left( \inner{f_{\bb{\theta}^g}, g_{\epsilon,m}}_{\alpha,\beta} \right) + \dfrac{1}{\alpha(\alpha+\beta)}\psi\left( \norm{g_{\epsilon,m}}_{\alpha+\beta}^{\alpha+\beta} \right)\nonumber \\
        \leq{} & \dfrac{1}{\beta(\alpha+\beta)}\psi\left( \norm{f_{\bb{\theta}^g}}_{\alpha+\beta}^{\alpha+\beta} \right) - \dfrac{1}{\alpha\beta}\psi\left( (1-\epsilon)^\beta\inner{f_{\bb{\theta}^g}, g}_{\alpha,\beta} \right) + \dfrac{1}{\alpha(\alpha+\beta)}\psi\left( \norm{g_{\epsilon,m}}_{\alpha+\beta}^{\alpha+\beta} \right).
        \label{eqn:gab-bp-proof-step2}
    \end{align}
    \noindent Here, we make use of the strictly increasing behavior of the $\psi$ function, which is a necessary condition for $\psi$ to be a valid characterizing function for the GAB divergence; see~\cite{roy2025characterization} for a proof. We denote the right-hand side of Eq.~\eqref{eqn:gab-bp-proof-step2} by $\Delta_m(\epsilon)$.

    \noindent\textbf{\underline{Step 3}:} Now, we will have a contradiction to the assumption that the MGABD functional breaks down at $\epsilon$, if, for all sufficiently large $m$, we have
    \begin{equation*}
        \Delta_m(\epsilon) < \delta_m(\epsilon).
    \end{equation*}
    \noindent This comparison leads to the inequality
    \begin{align*}
             & \frac{1}{\beta(\alpha+\beta)}\psi\left( \norm{f_{\bb{\theta}^g}}_{\alpha+\beta}^{\alpha+\beta}\right)-\frac{1}{\alpha\beta}\psi\left( (1-\epsilon)^\beta \inner{ f_{\bb{\theta}^g}, g}_{\alpha,\beta}\right)         \\
             & \qquad < \frac{1}{\beta(\alpha+\beta)}\psi\left( \norm{f_{\bb{\theta}_m}}_{\alpha+\beta}^{\alpha+\beta} \right) - \frac{1}{\alpha\beta}\psi\left( \epsilon^\beta \inner{f_{\bb{\theta}_m},k_m}_{\alpha,\beta}\right) \\
        \iff & \psi\left( (1-\epsilon)^\beta \inner{ f_{\bb{\theta}^g}, g}_{\alpha,\beta}\right)-\psi\left(\epsilon^\beta \inner{f_{\bb{\theta}_m},k_m}_{\alpha,\beta}\right)                                                       \\
             & \qquad > \frac{\alpha}{\alpha+\beta}\left(\psi\left( \norm{ f_{\bb{\theta}^g}}_{\alpha+\beta}^{\alpha+\beta} \right) -\psi\left( \norm{f_{\bb{\theta}_m}}_{\alpha+\beta}^{\alpha+\beta} \right) \right).
    \end{align*}
    \noindent However, according to the condition~\eqref{cond:gab-breakdown-general}, the above inequality is satisfied for the specific choice of $\epsilon$ (since $\epsilon < \tilde{\epsilon}$) and the choice of $\bb{\theta}^g$ depending on $\epsilon$. Therefore, we have a contradiction, which means the MGABD functional does not asymptotically break down at $\epsilon$. Since $\epsilon < \tilde{\epsilon}$ is arbitrary, the result follows.
\end{proof}

\subsection{Proof of Corollary~\ref{cor:gab-bp-km-bound-C}}\label{appendix-proof:bp-km-bound-C}

\begin{proof}
    It is evident from Theorem 3.2 of~\cite{roy2025characterization} that $\psi$ is increasing and geometrically-convex (i.e., $\Psi(x) := \psi(e^x)$ is convex). As a result,
    \begin{align*}
        \psi\left( \inner{f_{\bb{\theta}_m}, \epsilon k_m}_{\alpha,\beta} \right)
         & \leq \psi\left( \norm{f_{\bb{\theta}_m}}_{\alpha+\beta}^{\alpha} \norm{\epsilon k_m}_{\alpha+\beta}^{\beta} \right), \text{by H\"{o}lder's inequality}                                                                                          \\
         & \leq \dfrac{\alpha}{\alpha+\beta} \psi\left( \norm{f_{\bb{\theta}_m}}_{\alpha+\beta}^{\alpha+\beta} \right) + \dfrac{\beta}{\alpha+\beta} \psi\left( \norm{\epsilon k_m}_{\alpha+\beta}^{\alpha+\beta} \right), \text{ by geometric convexity.}
    \end{align*}
    \noindent Continuing, we get
    \begin{equation*}
        \psi\left( \inner{f_{\bb{\theta}_m}, \epsilon k_m}_{\alpha,\beta} \right) -  \dfrac{\alpha}{\alpha+\beta} \psi\left( \norm{f_{\bb{\theta}_m}}_{\alpha+\beta}^{\alpha+\beta} \right) \leq \dfrac{\beta}{\alpha+\beta} \psi\left( \norm{\epsilon k_m}_{\alpha+\beta}^{\alpha+\beta} \right) \leq \dfrac{\beta}{\alpha+\beta} \psi( (C\epsilon)^{\alpha+\beta} ),
    \end{equation*}
    \noindent where the last inequality follows from the increasing nature of $\psi$ function.

    \noindent In view of the condition~\eqref{cond:gab-breakdown-general}, it is therefore enough to show that for all $\epsilon < \epsilon^\ast_{(\alpha, \beta)}$, the following inequality holds
    \begin{equation*}
        \psi\left( \inner{f_{\bb{\theta}^g}, (1-\epsilon) g}_{\alpha,\beta} \right) - \dfrac{\alpha}{\alpha+\beta} \psi\left(\norm{f_{\bb{\theta}^g}}_{\alpha+\beta}^{\alpha+\beta} \right) - \dfrac{\beta}{\alpha+\beta} \psi\left( (C\epsilon)^{\alpha+\beta} \right) > 0.
    \end{equation*}
    \noindent Clearly, when $\epsilon = \epsilon^\ast_{(\alpha, \beta)}$, the left-hand side of the above inequality is $0$. Since $\psi$ is strictly increasing and $\inner{f_{\bb{\theta}^g}, (1-\epsilon)g}_{\alpha,\beta} = (1-\epsilon)^\beta \inner{f_{\bb{\theta}^g}, g}_{\alpha,\beta}$ is decreasing function of $\epsilon$, the first term in left-hand side is strictly decreasing in $\epsilon$ in $[0, 1]$. Since $C \geq 0$, $\psi\left( (C\epsilon)^{\alpha+\beta} \right)$ is increasing in $\epsilon$. Hence, the left-hand side is strictly decreasing in $\epsilon$. Hence, for all $\epsilon < \epsilon^\ast_{(\alpha,\beta)}$ the left-hand side of the above inequality must be positive.
\end{proof}

\subsection{Proof of Corollary~\ref{cor:gab-bp-location-model}}\label{appendix-proof:bp-location-model}

\begin{proof}
    We start by noting that due to the consideration of the location family, $\norm{f_{\bb{\theta}}}_{\alpha+\beta}^{\alpha+\beta} = c$ is independent of $\bb{\theta} \in \bb{\Theta}$. Also, since $g \in \{ f_{\bb{\theta}} : \bb{\theta} \in \bb{\Theta} \}$, the density at the best fitting parameter satisfy $f_{\bb{\theta}^g} = g$. Clearly, $\norm{k_m}_{\alpha+\beta}^{\alpha+\beta}$ also equals to $c$ in this case. As a result of these, Assumption~\ref{assum:bp-integrable} is automatically satisfied.

    Hence, in condition~\eqref{cond:gab-breakdown-general}, the RHS of the inequality becomes equal to $0$. Also, the strictly increasing nature of $\psi$ along with its continuity implies that $\psi$ is one-one, and its inverse function is also strictly increasing. Hence, the inequality in~\eqref{cond:gab-breakdown-general} reduces to
    \begin{equation*}
        (1 - \epsilon)^\beta \int f_{\bb{\theta}_g}^\alpha g^\beta \geq \epsilon^\beta \int f_{\bb{\theta}_m}^\alpha k_m^\beta.
    \end{equation*}
    \noindent Now, it follows by H\"{o}lder's inequality that
    \begin{equation*}
        \inner{f_{\bb{\theta}_m}, k_m}_{\alpha,\beta} \leq \norm{f_{\bb{\theta}_m}}^{\alpha}_{\alpha+\beta} \norm{k_m}^{\beta}_{\alpha+\beta} = c = \norm{f_{\bb{\theta}^g}}^{\alpha+\beta}_{\alpha+\beta} = \inner{f_{\bb{\theta}^g}, g}_{\alpha,\beta}.
    \end{equation*}
    \noindent Therefore, to ensure the inequality in condition~\eqref{cond:gab-breakdown-general}, it is enough to ensure $(1 - \epsilon)^\beta \geq \epsilon^\beta$, i.e, $\epsilon < 1/2$. Thus, choosing $\tilde{\epsilon} = 1/2$ satisfies condition~\eqref{cond:gab-breakdown-general}, and as a result, the MGABD functional in this case has an asymptotic breakdown point equal to $1/2$.
\end{proof}

\subsection{Proof of Corollary~\ref{cor:gab-bp-special1}}\label{appendix-proof:bp-special-1}

\begin{proof}
    Let
    \begin{align*}
        \tilde{\epsilon}
         & = \min \left\{ \liminf_{m\to \infty}\left[\frac{\psi^{-1}\left(\frac{\alpha}{\alpha+\beta}\psi( \norm{f_{\bb{\theta}_m}}_{\alpha+\beta}^{\alpha+\beta})\right)}{ \inner{f_{\bb{\theta}_m},k_m}_{\alpha,\beta}}\right]^{1/\beta}, 1-\left[\frac{\psi^{-1}\left(\frac{\alpha}{\alpha+\beta} \psi( \norm{f_{\theta_g}}_{\alpha+\beta}^{\alpha+\beta})\right)}{ \inner{ f_{\theta_g},g}_{\alpha,\beta}}\right]^{1/\beta}\right\} \\
         & = \liminf_{m\to \infty} \min \left\{\left[\frac{\psi^{-1}\left(\frac{\alpha}{\alpha+\beta}\psi( \norm{f_{\theta_m}}_{\alpha+\beta}^{\alpha+\beta})\right)}{ \inner{ f_{\theta_m},k_m}_{\alpha,\beta}}\right]^{1/\beta}, 1-\left[ \frac{\psi^{-1}\left(\frac{\alpha}{\alpha+\beta} \psi(\norm{f_{\theta_g}}_{\alpha+\beta}^{\alpha+\beta})\right)}{ \inner{f_{\theta_g},g}_{\alpha,\beta}}\right]^{1/\beta}\right\}
    \end{align*}
    \noindent The exchangeability of the limit inferior and minimum can be verified by standard calculations~\citep{rudin1976principles}. Also, the above expression is valid only if for $\alpha, \beta > 0$, the quantities $\frac{\alpha}{\alpha+\beta}\psi( \norm{f_{\bb{\theta}}}_{\alpha+\beta}^{\alpha+\beta})$ for $\bb{\theta} = \bb{\theta}_m$ and $\bb{\theta} = \bb{\theta}^g$ remain in the range of $\psi$, so that the corresponding terms involving $\psi^{-1}$ are well-defined. One can verify that this is indeed the case for $\psi(x) = \phi^{-1}x^\phi$ with $\phi > 0$. Now, by definition of $\tilde{\epsilon}$, for every $\epsilon < \tilde{\epsilon}$, for sufficiently large enough $m$, we must have
    \begin{equation*}
        \epsilon< \min \left\{ \left[ \frac{\psi^{-1}\left(\frac{\alpha}{\alpha+\beta}\psi( \norm{f_{\bb{\theta}_m}}_{\alpha+\beta}^{\alpha+\beta})\right)}{ \inner{f_{\bb{\theta}_m},k_m}_{\alpha,\beta}}\right]^{1/\beta}, 1-\left[\frac{\psi^{-1}\left(\frac{\alpha}{\alpha+\beta} \psi( \norm{f_{\bb{\theta}_g}}_{\alpha+\beta}^{\alpha+\beta}) \right)}{ \inner{f_{\bb{\theta}_g},g}_{\alpha,\beta}}\right]^{1/\beta}\right\}.
    \end{equation*}
    \noindent As a result, $\epsilon$ is less than the first term, which can be rephrased as
    \begin{equation}
        \psi(\epsilon^\beta  \inner{f_{\bb{\theta}_m},k_m}_{\alpha,\beta}) < \frac{\alpha}{\alpha+\beta} \psi(\norm{f_{\bb{\theta}_m}}_{\alpha+\beta}^{\alpha+\beta}),
        \label{eqn:gab-bp-special1-ineq1}
    \end{equation}
    \noindent where we use the monotonically increasing property of $\psi$. Similarly, since $\epsilon$ is also less than the second term, we obtain
    \begin{equation}
        \psi((1-\epsilon)^\beta \inner{f_{\bb{\theta}_g},g}_{\alpha,\beta}) > \frac{\alpha}{\alpha+\beta} \psi( \norm{f_{\bb{\theta}_g}}_{\alpha+\beta}^{\alpha+\beta}).
        \label{eqn:gab-bp-special1-ineq2}
    \end{equation}
    \noindent Combining the inequalities in~\eqref{eqn:gab-bp-special1-ineq1} and~\eqref{eqn:gab-bp-special1-ineq2}, we get the condition~\eqref{cond:gab-breakdown-general}. An application of Proposition~\ref{thm:gab-breakdown-general} now ensures that the asymptotic breakdown point is at least $\tilde{\epsilon}$.

    However, since $\psi(x) = \frac{1}{\phi} x^\phi$, we can further obtain a lower bound to $\tilde{\epsilon}$ by noting that
    \begin{align}
        \left[\frac{\psi^{-1}\left(\frac{\alpha}{\alpha+\beta}\psi(\norm{f_{\theta_m}}_{\alpha+\beta}^{\alpha+\beta})\right)}{\inner{f_{\theta_m},k_m}_{\alpha,\beta}}\right]^{1/\beta}
        ={}    & \left( \dfrac{\alpha}{\alpha+\beta} \right)^{1/\beta\phi} \left[ \dfrac{\norm{f_{\theta_m}}_{\alpha+\beta}^{(\alpha+\beta)} }{\inner{f_{\theta_m}, k_m}_{\alpha,\beta} } \right]^{1/\beta}\nonumber                                                            \\
        \geq{} & \left( \dfrac{\alpha}{\alpha+\beta} \right)^{1/\beta\phi} \left[ \dfrac{\norm{f_{\theta_m}}_{\alpha+\beta}^{(\alpha+\beta)} }{ \norm{f_{\theta_m}}_{\alpha+\beta}^\alpha \norm{k_m}_{\alpha+\beta}^\beta } \right]^{1/\beta} \label{eqn:gab-bp-special1-ineq3} \\
        \geq{} & \left( \dfrac{\alpha}{\alpha+\beta} \right)^{1/\beta\phi}.
        \label{eqn:gab-bp-special1-ineq4}
    \end{align}
    \noindent The inequality in~\eqref{eqn:gab-bp-special1-ineq3} follows from H\"{o}lder's inequality while the inequality in~\eqref{eqn:gab-bp-special1-ineq4} is a consequence of the assumption that $\norm{f_{\theta_m}}_{\alpha+\beta} \geq \norm{k_m}_{\alpha+\beta}$ for sufficiently large $m$. Therefore,
    \begin{equation*}
        \tilde{\epsilon} \geq \min\left\{ \left(\dfrac{\alpha}{\alpha+\beta} \right)^{1/\beta\phi}, 1-\left(\dfrac{\alpha}{\alpha+\beta} \right)^{1/\beta\phi} \dfrac{ \norm{f_{\theta^g}}_{\alpha+\beta}^{1 + \alpha/\beta} }{ \inner{f_{\theta^g}, g}_{\alpha, \beta}^{1/\beta}} \right\}.
    \end{equation*}
\end{proof}

\subsection{Proof of Proposition~\ref{thm:gab-bp-general-zero}}\label{appendix-proof:bp-general-zero}

\begin{proof}
    The proof is very similar to the proof of Proposition~\ref{thm:gab-breakdown-general}, with two small differences. First, the H\"{o}lder's inequality is not applicable, due to which some convergence results need to be modified appropriately. Secondly, one has to work with the limiting form of the GAB divergence for the $\beta = 0$ case. We only indicate these differences in the proof below.

    \noindent\textbf{\underline{Step 1}:}

    By choosing the set $A_m$ as before, due to Assumptions~\ref{assum:bp-g-sig-km}-\ref{assum:bp-integrable}, an application of Lemma A.2 of~\cite{roy2023breakdown} implies that, as $m \to \infty$,
    \begin{equation*}
        \int f_{\bb{\theta}_m}^\alpha \log(f_{\bb{\theta}_m}/g_{\epsilon,m}) \asymp \int_{A_m^c} f_{\bb{\theta}_m}^\alpha \log(f_{\bb{\theta}_m}/\epsilon k_m).
    \end{equation*}
    \noindent Since, $\psi \in C^1((0,\infty))$, by using continuity of $\psi$, we get
    \begin{equation}
        \alpha^2d_{GAB}^{(\alpha,0),\psi}(f_{\bb{\theta}_m}, g_{\epsilon,m})
        \asymp \alpha \psi\left(\norm{f_{\bb{\theta}_m}}_{\alpha}^\alpha \right) \int f_{\bb{\theta}_m}^\alpha \log(f_{\bb{\theta}_m}/\epsilon k_m) - \psi\left(\norm{f_{\bb{\theta}_m}}_{\alpha}^\alpha \right) + \psi\left( \norm{g_{\epsilon, m}}_{\alpha}^\alpha \right),
        \label{eqn:gab-bp-general-zero-eqn1}
    \end{equation}
    \noindent as $m \rightarrow \infty$.

    \noindent\textbf{\underline{Step 2}:}

    Since $g_{\epsilon,m} = (1-\epsilon)g + \epsilon k_m$, we clearly have the inequality
    \begin{equation*}
        \int f_{\bb{\theta}^g}^\alpha \log(g_{\epsilon, m}/f_{\bb{\theta}^g}) \geq \int f_{\bb{\theta}^g}^\alpha \log((1-\epsilon)g  / f_{\bb{\theta}^g}),
    \end{equation*}
    \noindent for any $m \geq 1$ and any $\epsilon \in (0, 1/2]$. Since for $\alpha = 0$, the necessary conditions for GABD to be a valid divergence measure is that $\psi$ is a strictly increasing function (see~\cite{roy2025characterization}), it follows that $\psi'(\norm{f_{\bb{\theta}^g}}_{\alpha}^{\alpha}) > 0$. Therefore, we have the inequality
    \begin{equation}
        \alpha^2 d_{GAB}^{(\alpha,0),\psi}(f_{\bb{\theta}^g}, g_{\epsilon, m})
        \leq \alpha \psi'\left( \norm{f_{\bb{\theta}^g}}_{\alpha}^\alpha \right) \int f_{\bb{\theta}^g}^\alpha \log( f_{\bb{\theta}^g}/(1-\epsilon)g) - \psi\left(\norm{f_{\bb{\theta}^g}}_{\alpha}^\alpha \right) + \psi\left( \norm{g_{\epsilon, m}}_{\alpha}^\alpha \right).
        \label{eqn:gab-bp-general-zero-eqn2}
    \end{equation}

    \noindent\textbf{\underline{Step 3}:}

    In step 3, we compare the right-hand sides of the asymptotic equivalence relation~\eqref{eqn:gab-bp-general-zero-eqn1} and the inequality~\eqref{eqn:gab-bp-general-zero-eqn2} to yield the condition~\eqref{cond:gab-bp-general-zero}, completing the proof.
\end{proof}

\subsection{Proof of Corollary~\ref{cor:gab-bp-km-bound-C-zero}}\label{appendix-proof:km-bound-C-zero}

\begin{proof}
    We first note that, the GAB divergence between two non-normalized density functions $f_{\bb{\theta}_m}$ and $\epsilon k_m$ is nonnegative for valid choices of $\psi$-function. In mathematical terms, this means
    \begin{equation*}
        \alpha^2 d_{GAB}^{(\alpha,0)}(f_{\bb{\theta}_m}, \epsilon k_m)
        = \alpha \psi'\left( \norm{f_{\bb{\theta}_m}}_\alpha^\alpha \right) \int f_{\bb{\theta}_m}^\alpha \log(f_{\bb{\theta}_m}/\epsilon k_m) - \psi\left( \norm{f_{\bb{\theta}_m}}_\alpha^\alpha \right) + \psi\left( \norm{\epsilon k_m}_\alpha^\alpha \right) \geq 0.
    \end{equation*}
    \noindent Since $\psi$ must be increasing, it follows that
    \begin{equation*}
        \alpha \psi'\left( \norm{f_{\bb{\theta}_m}}_\alpha^\alpha \right) \int f_{\bb{\theta}_m}^\alpha \log(f_{\bb{\theta}_m}/\epsilon k_m) - \psi\left( \norm{f_{\bb{\theta}_m}}_\alpha^\alpha \right) + \psi\left( (C\epsilon)^\alpha \right) \geq 0.
    \end{equation*}
    \noindent In view of condition~\eqref{cond:gab-bp-general-zero}, it is thus enough to show that for any $\epsilon < \epsilon^\ast$, we have the inequality
    \begin{align*}
             & -\dfrac{1}{\alpha}\psi\left( (C\epsilon)^\alpha \right) > \psi\left( \norm{f_{\bb{\theta}^g}}_\alpha^\alpha \right)\int f_{\bb{\theta}^g}\ln\left( \dfrac{f_{\bb{\theta}^g}}{(1-\epsilon)g} \right) - \dfrac{1}{\alpha}\psi\left( \norm{f_{\bb{\theta}^g}}_\alpha^\alpha \right),               \\
        \iff & \psi'\left( \norm{f_{\bb{\theta}^g}}_{\alpha}^{\alpha}\right)\int f_{\bb{\theta}^g}^{\alpha}\log\left(\frac{(1-\epsilon)g}{f_{\bb{\theta}^g}}\right) - \dfrac{1}{\alpha} \left[\psi\left( (C\epsilon)^{\alpha} \right)-\psi\left( \norm{f_{\bb{\theta}^g}}_{\alpha}^{\alpha}\right)\right] > 0.
    \end{align*}
    \noindent Clearly, when $\epsilon = \epsilon^\ast$, the left-hand side of the above inequality is $0$. The first term in left-hand side is strictly decreasing in $\epsilon$ in $[0, 1]$. Since $C \geq 0$, $\psi\left( (C\epsilon)^{1+\alpha} \right)$ is increasing in $\epsilon$. Hence, the left-hand side is strictly decreasing in $\epsilon$. Hence, for all $\epsilon < \epsilon^\ast$ the left-hand side of the above inequality must be positive.

    Now an application of Proposition~\ref{thm:gab-bp-general-zero} establishes the corollary.
\end{proof}

\subsection{Proof of Corollary~\ref{cor:gab-bp-location-model-zero}}\label{appendix-proof:bp-location-model-zero}

\begin{proof}
    Since we consider the location family, the quantities $\norm{f_{\bb{\theta}}}_{\alpha}^{\alpha} = c$, free of $\bb{\theta}$. Also, as the true density $g$ belong to the same location family, we have $f_{\bb{\theta}^g} = g$ and $\norm{g}_\alpha^\alpha = c$. Hence, the right-hand side of the inequality~\eqref{cond:gab-bp-general-zero} reduces to $0$. In fact, the condition~\eqref{cond:gab-bp-general-zero} reduces to the inequality
    \begin{equation}
        \psi'(c) \left[ \int f_{\bb{\theta}_m}^\alpha \ln\left( \dfrac{f_{\bb{\theta}_m}}{k_m} \right) - c\log(\epsilon) + c\log(1-\epsilon) \right] > 0.
        \label{eqn:gab-bp-location-model-zero-eq1}
    \end{equation}
    \noindent For the corresponding GABD functional to be properly defined, $\psi$ must be a strictly increasing function (see~\cite{roy2025characterization}), and hence $\psi'(c) > 0$. Also note that, since both $f_{\bb{\theta}_m}$ and $k_m$ belong to the same location family, we have
    \begin{equation*}
        \int f_{\bb{\theta}_m}^\alpha \ln\left( \dfrac{f_{\bb{\theta}_m}}{k_m} \right)
        = \frac{1}{\alpha} \int f_{\bb{\theta}_m}^\alpha \ln\left( \dfrac{f_{\bb{\theta}_m}^\alpha / c }{k_m^\alpha / c} \right)
        = \frac{c}{\alpha} d_{KL}(f_{\bb{\theta}_m}^{[\alpha]}, k_m^{[\alpha]}),
    \end{equation*}
    \noindent where $d_{KL}(f_{\bb{\theta}_m}^{[\alpha]}, k_m^{[\alpha]})$ denotes the Kullback-Leibler (KL) divergence between the densities proportional to $f_{\bb{\theta}_m}^\alpha$ and $k_m^\alpha$. As a result, condition~\eqref{eqn:gab-bp-location-model-zero-eq1} reduces to
    \begin{equation*}
        d_{KL}(f_{\theta_m}^{[\alpha]}, k_m^{[\alpha]}) > \ln\left( \dfrac{\epsilon}{1-\epsilon} \right).
    \end{equation*}
    \noindent Clearly, the left-hand side is nonnegative, and hence it is enough to choose $\epsilon$ so that the right-hand side is negative, which happens for any $\epsilon < (1-\epsilon)$, i.e., $\epsilon < 1/2$. In other words, for the location parameter estimation setup, the condition~\eqref{cond:gab-bp-general-zero} remains true for any $\epsilon < 1/2$. The result now follows due to Proposition~\ref{thm:gab-bp-general-zero}.
\end{proof}

\subsection{Proof of Corollary~\ref{cor:gab-bp-zero-special1}}\label{appendix-proof:bp-special1-zero}

\begin{proof}
    We shall show that the asymptotic breakdown point is at least
    \begin{equation*}
        \min\left\{ \exp\left( -\dfrac{1}{\alpha} \sup_{y \in \R} \dfrac{\Psi(y)}{\Psi'(y)} \right), 1 - \exp\left[ \dfrac{\int f_{\bb{\theta}^g}^{\alpha} \log(f_{\bb{\theta}^g}/g) }{ \norm{f_{\bb{\theta}^g}}_{\alpha}^{\alpha} } - \dfrac{\Psi\left( \norm{f_{\bb{\theta}^g}}_{\alpha}^{\alpha} \right) }{\alpha \Psi'\left( \norm{f_{\bb{\theta}^g}}_{\alpha}^{\alpha} \right)} \right], \dfrac{1}{2} \right\},
    \end{equation*}
    \noindent where $\Psi(x) = \psi(e^x)$. The rest will follow from the specific choice of $\psi(x) = \phi^{-1}x^\phi$.

    Using H\"{o}lder's inequality and that fact that $\log(1+x) < x$, it follows that
    \begin{align*}
        \int f_{\bb{\theta}_m}^{\alpha} \log\left( \dfrac{\epsilon k_m}{f_{\bb{\theta}_m}} \right)
         & \leq \norm{f_{\bb{\theta}_m}}^{\alpha}_{\alpha} \log(\epsilon) + \int f_{\bb{\theta}_m}^{\alpha-1} k_m - \norm{f_{\bb{\theta}_m}}^{\alpha}_{\alpha} \\
         & \leq \norm{f_{\bb{\theta}_m}}^{\alpha}_{\alpha} (\log(\epsilon) - 1) + \norm{f_{\bb{\theta}_m}}^{\alpha-1}_{\alpha} \norm{k_m}_{\alpha}             \\
         & \leq \norm{f_{\bb{\theta}_m}}^{\alpha}_{\alpha} \log(\epsilon),
    \end{align*}
    \noindent for all sufficiently large $m$. Here, we also use the fact that $\norm{k_m}_{\alpha} \leq \norm{f_{\bb{\theta}_m}}_{\alpha}$ at the last step of the inequality.  In view of the condition~\eqref{cond:gab-bp-general-zero}, it is now enough to show that
    \begin{equation}
        -\psi'\left( \norm{f_{\bb{\theta}_m}}^{\alpha}_{\alpha} \right)\norm{f_{\bb{\theta}_m}}^{\alpha}_{\alpha}  \log(\epsilon) \geq \dfrac{1}{\alpha}\psi\left( \norm{f_{\bb{\theta}_m}}^{\alpha}_{\alpha} \right),
        \label{eqn:gab-bp-zero-special1-cond1}
    \end{equation}
    \noindent and,
    \begin{equation}
        \psi'\left( \norm{f_{\bb{\theta}^g}}_{\alpha}^{\alpha} \right) \int f_{\bb{\theta}^g}^{1+\alpha} \log\left( \dfrac{f_{\bb{\theta}^g}}{(1-\epsilon)g} \right) < \dfrac{1}{\alpha} \psi\left(  \norm{f_{\bb{\theta}^g}}_{\alpha}^{\alpha}  \right).
        \label{eqn:gab-bp-zero-special1-cond2}
    \end{equation}

    The first inequality given in~\eqref{eqn:gab-bp-zero-special1-cond1} translates to
    \begin{equation*}
        \epsilon < \exp\left( -\dfrac{\psi(x)}{\alpha x\psi'(x)} \right) = \exp\left( - \dfrac{\Psi(\log(x))}{\Psi'(\log(x))} \right),
    \end{equation*}
    \noindent where $x = \norm{f_{\theta_m}}_{\alpha}^{\alpha}$. The division is permissible here since $\psi(x)$ is strictly increasing for all $x > 0$, resulting in $\psi'(x) > 0$. This now means, we can take any $\epsilon$ such that
    \begin{equation*}
        \epsilon < \exp\left( -\dfrac{1}{\alpha} \sup_{y} \dfrac{\Psi(y)}{\Psi'(y)} \right),
    \end{equation*}
    \noindent and the inequality in~\eqref{eqn:gab-bp-zero-special1-cond1} will follow.

    On the other hand, the second inequality~\eqref{eqn:gab-bp-zero-special1-cond2} translates to
    \begin{align*}
             & \psi'\left( \norm{f_{\bb{\theta}^g}}_{\alpha}^{\alpha} \right) \int f_{\bb{\theta}^g}^{\alpha} \log(1-\epsilon) > \psi'\left( \norm{f_{\bb{\theta}^g}}_{\alpha}^{\alpha} \right) \int f_{\bb{\theta}^g}^{\alpha} \log(f_{\bb{\theta}^g}/g) - \dfrac{1}{\alpha} \psi\left( \norm{f_{\bb{\theta}^g}}_{\alpha}^{\alpha} \right) \\
        \iff & \log(1-\epsilon) > \dfrac{\int f_{\bb{\theta}^g}^{\alpha} \log(f_{\bb{\theta}^g}/g) }{ \norm{f_{\bb{\theta}^g}}_{\alpha}^{\alpha} } - \dfrac{\Psi\left( \log(\norm{f_{\bb{\theta}^g}}_{\alpha}^{\alpha}) \right) }{\alpha \Psi'\left( \log(\norm{f_{\bb{\theta}^g}}_{\alpha}^{\alpha}) \right)}                              \\
        \iff & \epsilon < 1 - \exp\left[ \dfrac{\int f_{\bb{\theta}^g}^{\alpha} \log(f_{\bb{\theta}^g}/g) }{z} - \dfrac{\Psi\left( \log(z) \right) }{\alpha \Psi'\left( \log(z) \right)} \right],
    \end{align*}
    \noindent where $z = \norm{f_{\bb{\theta}^g}}_{\alpha}^{\alpha}$.

    Since the choice of $\epsilon$ given in the statement of the result satisfies both the inequalities~\eqref{eqn:gab-bp-zero-special1-cond1} and~\eqref{eqn:gab-bp-zero-special1-cond2}, an application of Proposition~\ref{thm:gab-bp-general-zero} now completes the proof.
\end{proof}

\subsection{Proof of Theorem~\ref{thm:gab-optimality-discrete}}\label{appendix-proof:optimality-discrete}

\begin{proof}

    The objective is to maximize
    \begin{equation*}
        \inf_{x > 0} L(x) = \inf_{x > 0} x\dfrac{\psi''(x)}{\psi'(x)},
    \end{equation*}
    \noindent where $x = \norm{g}_{\alpha+\beta}^{\alpha+\beta}$, subject to the constraint
    \begin{equation*}
        \psi(\eta x) - \lambda_\alpha \psi(x) - (1-\lambda_\alpha) \psi(d) \geq 0,
    \end{equation*}
    \noindent where $d = C^{\alpha+\beta} (\epsilon^\ast)^{\alpha+\beta}$, $\eta = (1-\epsilon^\ast)^\beta$ and $\lambda_\alpha = \frac{\alpha}{\alpha+\beta}$. Let $\psi_1(x) = \psi(x) - \psi(d), \eta = e^{-\delta}$ for some $\delta > 0$, $y = \log(x)$, and define $\Psi_1(y) = \psi_1(e^y) = \psi_1(x)$. Note that,
    \begin{align*}
        \psi'(x)  & = \psi'_1(x) = \frac{1}{x} \Psi'_1(\log(x)),                                       \\
        \psi''(x) & = \psi''_1(x) = \frac{1}{x^2} \left[ \Psi''_1(\log(x)) - \Psi'_1(\log(x)) \right].
    \end{align*}
    \noindent With the help of these notations, we can rewrite the objective function as
    \begin{equation*}
        L(x) = \dfrac{\Psi''_1(\log(x)) - \Psi'_1(\log(x))}{\Psi'_1(\log(x))} = \dfrac{\Psi''_1(y)}{\Psi'_1(y)} - 1 = \frac{d}{dy}\left( \log\Psi'_1(y) \right) - 1.
    \end{equation*}
    \noindent and the constraint inequality as
    \begin{equation}
        \Psi_1(y - \delta) \geq \frac{\alpha}{\alpha+\beta} \Psi_1(y).
        \label{eqn:gab-optimality-discrete-proof2}
    \end{equation}
    \noindent Since any feasible $\psi$ must be increasing and geometrically convex, it follows that $\Psi_1$ is increasing and convex. The constraint in Eq.~\eqref{eqn:gab-optimality-discrete-proof2} effectively provides an upper bound to its growth rate. To see this, by repeatedly applying inequality~\eqref{eqn:gab-optimality-discrete-proof2}, we obtain
    \begin{equation*}
        \Psi_1(y - n\delta) \geq \left( \frac{\alpha}{\alpha+\beta} \right)^n \Psi_1(y),
    \end{equation*}
    \noindent for any $n \geq 1$. Choose $n = [y/\delta]$, where $[a]$ denotes the largest integer less than or equal to $a$. Since $y - [y/\delta]\delta \in [0, 1]$, by appealing to the increasing property of $\Psi_1$, we have
    \begin{equation*}
        \Psi_1(1) \geq \Psi_1(y - [y/\delta] \delta) \geq \Psi_1(y) \left( \dfrac{\alpha}{\alpha+\beta} \right)^{[y/\delta]},
    \end{equation*}
    \noindent i.e.,
    \begin{equation}
        \Psi_1(y) \leq \Psi_1(1) \left(1 + \frac{\beta}{\alpha} \right)^{y/\delta} = \Psi_1(1) \exp\left( \frac{y}{\delta}\log\left(1+\frac{\beta}{\alpha} \right) \right).
        \label{eqn:gab-optimality-discrete-proof3}
    \end{equation}

    On the other hand, if $\inf_{x > 0}L(x) = c$, then
    \begin{equation*}
        \dfrac{d}{dy}(\log \Psi'_1(y)) = 1 + L(x) \geq (c+1).
    \end{equation*}
    \noindent This provides a lower bound to the growth rate of $\Psi_1$. Namely, by the fundamental theorem of calculus (FTC), we have for any fixed $y_0 \in \R$,
    \begin{equation*}
        \Psi'_1(y) \geq \Psi'_1(y_0) e^{(c + 1)(y - y_0)},
    \end{equation*}
    \noindent for any $y \geq y_0$. Applying the FTC again, we obtain
    \begin{equation*}
        \Psi_1(y) \geq \Psi_1(y_0) + \int_{y_0}^y \Psi'_1(y_0) e^{(c+1)(t - y_0)} dt = \Psi_1(y_0) + \Psi'_1(y_0) \frac{e^{(c+1)(y - y_0)} - 1}{c+1} \geq c_2 e^{(c+1)y},
    \end{equation*}
    \noindent for some sufficiently small constant $c_2 > 0$. This suggests that the growth rate of $\Psi_1(y)$ is bounded below by an exponential function with exponent $(c+1)y$. Clearly, to maximize $c = \inf_{x > 0} L(x)$ within the feasible set constrained by Eq.~\eqref{eqn:gab-optimality-discrete-proof3}, we choose the maximum $c$ satisfying the inequality,
    \begin{equation*}
        cy \leq \frac{y}{\delta} \log\left( 1 + \frac{\beta}{\alpha} \right).
    \end{equation*}
    \noindent Obvious choice of the optimal $c$, therefore, leads to the optimal choice
    \begin{equation*}
        \Psi_1^\ast(y) = B\exp\left( -y \dfrac{\log(1 + \beta / \alpha)}{\beta \log(1 - \epsilon^\ast)}  \right),
    \end{equation*}
    \noindent for some $B > 0$, i.e.,
    \begin{equation*}
        \psi^\ast(x) = A + B \exp\left( - \log(x) \frac{\log(1 + \beta / \alpha)}{\beta \log(1 - \epsilon^\ast)} \right) = A + B x^{-\frac{\log(1 + \beta/\alpha)}{\beta \log(1-\epsilon^\ast)}},
    \end{equation*}
    \noindent for some $A > 0$ and $B \in \R$.

    The second part of the proof follows from standard calculations by enforcing the above class of functions to satisfy $\psi(1) = 0$ and $\psi'(1) = 1$.

\end{proof}

\subsection{Proof of Theorem~\ref{thm:gab-optimality-beta0}}\label{appendix-proof:gab-optimality-beta0}

\begin{proof}
    Let $\Scal$ be the class of monotonically increasing, geometrically convex $\psi$-functions such that their growth rate is bounded as shown in the inequality~\eqref{eqn:optimality-growth-bound-beta0}. We shall show that a function minimizing $\sup_{x > 0} V^\psi(x)/x^2$ must satisfy a constant elasticity condition, i.e.,
    \begin{equation}
        \frac{x\psi'(x)}{\psi(x)} = \phi^\ast,
        \label{eqn:optimality-constant-elasticity-beta0}
    \end{equation}
    \noindent for some constant $c > 0$.

    We proceed by contradiction. Let us pick a function $\psi(x)$ such that for some $x_0$, $x_0\psi'(x_0)/\psi(x_0) < \phi^\ast$. Given such a $\psi(x)$, consider a transformation $\tilde{\psi}'(x) = \psi'(x) e^{\int \delta(x)dx}$ for some smooth function $\delta(x)$. This means,
    \begin{align*}
                 & \log(\tilde{\psi}'(x)) = \log(\psi'(x)) + \int \delta(x)dx                           \\
        \implies & \frac{\tilde{\psi}''(x)}{\tilde{\psi}'(x)} = \frac{\psi''(x)}{\psi'(x)} + \delta(x).
    \end{align*}
    \noindent Note that, this means,
    \begin{align*}
        \tilde{\psi}(x)    & = \psi(x)e^{\int \delta(x)dx} - \int e^{\int \delta(x)dx} \delta(x)\psi(x)dx,                               \\
        \tilde{\psi}''(x)  & = e^{\int \delta(x)dx} \left( \psi'(x) \delta(x) + \delta^2(x) \right)                                      \\
        \tilde{\psi}'''(x) & = e^{\int \delta(x)dx} \left( \psi'''(x) + 2\psi''(x)\delta(x) + \psi'(x)(\delta'(x) + \delta^2(x)  \right)
    \end{align*}
    \noindent where the first line follows from integration by parts. As a consequence, we have
    \begin{align*}
        \frac{\tilde{\psi}'''(x)}{\tilde{\psi}'(x)} & = \frac{\psi'''(x)}{\psi'(x)} + 2\delta(x)\frac{\psi''(x)}{\psi'(x)} + (\delta'(x) + \delta^2(x)) \\
        \dfrac{\tilde{\psi}(x)}{x\tilde{\psi}'(x)}  & = \frac{\psi(x)}{x\psi'(x)} - \int \delta(x) \frac{\psi(x)}{x\psi'(x)}dx.
    \end{align*}
    \noindent Note that, since $\psi$ is a feasible function, it must satisfy~\eqref{eqn:optimality-growth-bound-beta0}, i.e., $\psi(x)/x\psi'(x) \geq 1/\phi^\ast$. But by the assumption for the sake of contradiction, we have $\psi(x_0)/x\psi'(x_0) > 1/\phi^\ast > 0$. Since, $\psi \in C^3((0, \infty))$, there exists a sufficiently small $\delta(x)$ such that $\delta(x) = 0$ everywhere but a sufficiently small neighborhood of $x_0$ and $\delta(x_0) > 1$, satisfying
    \begin{equation*}
        \int \delta(x)\frac{\psi(x)}{x\psi'(x)}dx = \frac{\psi(x_0)}{x_0 \psi'(x_0)} - \frac{1}{\phi^\ast},
    \end{equation*}
    \noindent i.e., $\tilde{\psi}(x_0) / x_0 \tilde{\psi}'(x_0) = 1/\phi^\ast$. Outside the neighborhood of $x_0$, denoted by $N(x_0)$, as $\delta(x) = 0$, there is no change in the variance. However, when $x \in N(x_0)$, the asymptotic variance form given in Eq.~\eqref{eqn:optimality-variance-beta0} is reduced. This is because the numerator of the asymptotic variance only grows by $O(\delta^2(x))$, but the denominator grows by $O(\delta^4(x))$, which together reduces the variance slightly at $x = x_0$ as $\delta(x_0) > 1$. This provides a contradiction.

    As a result, the optimal $\psi$-function must satisfy the equality~\eqref{eqn:optimality-constant-elasticity-beta0}. Solving this differential equation yields the form, $\psi(x) = Ax^{\phi^\ast} + B$ for some constants $A$ and $B$ as we wanted.
\end{proof}

\section{Additional Empirical Studies}\label{appendix:empirical}

\subsection{Details of Figure~\ref{fig:frontier-motivation}}\label{appendix:frontier-fig}

To obtain Figure~\ref{fig:frontier-motivation}, we consider a setting where the parametric model family is
\begin{equation*}
    f_{\theta}(x) = \frac{1}{\theta} e^{-\theta x}, \ x > 0.
\end{equation*}
\noindent The true value of the $\theta$ is $1$, which can be used to calculate $\norm{g}_{\alpha+\beta}$ for any $\alpha, \beta$. Let us indicate the choice of any specific family of divergence measures by $\Dcal$, which provides restrictions on the choice of the triplets $(\alpha, \beta, \psi)$. For example, the family of density power divergence is characterized by
\begin{equation*}
    \Dcal_{\text{DPD}} = \{ (\alpha, 1, \psi): \alpha \in (0, 1), \psi(x) = x \}.
\end{equation*}
\noindent For each such triplet $(\alpha,\beta, \psi)$, let $\epsilon^{(\alpha,\beta),\psi}$ be the asymptotic breakdown point and $v^{(\alpha,\beta),\psi}$ be the asymptotic variance (without the $\sqrt{n}$-factor) for the corresponding minimum divergence estimator. Then, for a specific family of divergence described by $\Dcal$, the line curve draws $(\epsilon^\ast, v(\epsilon^\ast))$ where
\begin{equation}
    v(\epsilon^\ast) = \inf_{(\alpha, \beta, \psi) \in \Dcal, \epsilon^{(\alpha,\beta),\psi} > \epsilon^\ast} v^{(\alpha,\beta),\psi},
    \label{eqn:motivation-optimization}
\end{equation}
\noindent i.e., the minimum possible variance achievable by a member of this family by suitably tuning the parameter $\alpha$ and $\beta$, subject to the breakdown guarantee. For the optimal Pareto frontier, we specifically restrict $\psi(x) = (x^{\phi^\ast} - 1)/\phi^\ast$ with $\phi^\ast$ as given by~\eqref{eqn:best-gamma-star}. The minimization given in Eq.~\eqref{eqn:motivation-optimization} is performed numerically using a grid search over $\alpha \in [0, 1]$ and $\beta \in [0, 3]$.

\subsection{Beyond location and scale parameter estimation}

The results described in Sections~\ref{sec:asymp-normality}-\ref{sec:bp-analysis} are very general so that they are applicable to various settings beyond the typical location-scale families. As an illustration, consider the setting where the model family is the class of negative binomial distributions with unknown number of successes $r$ and unknown probability of success $p$. The uncontaminated data are generated from a negative binomial distribution with $r = 5$ and $p = 0.5$, while the contamination appears through a Poisson distribution with mean $\lambda = 100$. It is clear that the mean of the negative binomial distribution $r(1-p)/p = 5$ is much smaller than the mean of outliers.


\begin{figure}[htbp]
    \centering
    \includegraphics[width=\linewidth]{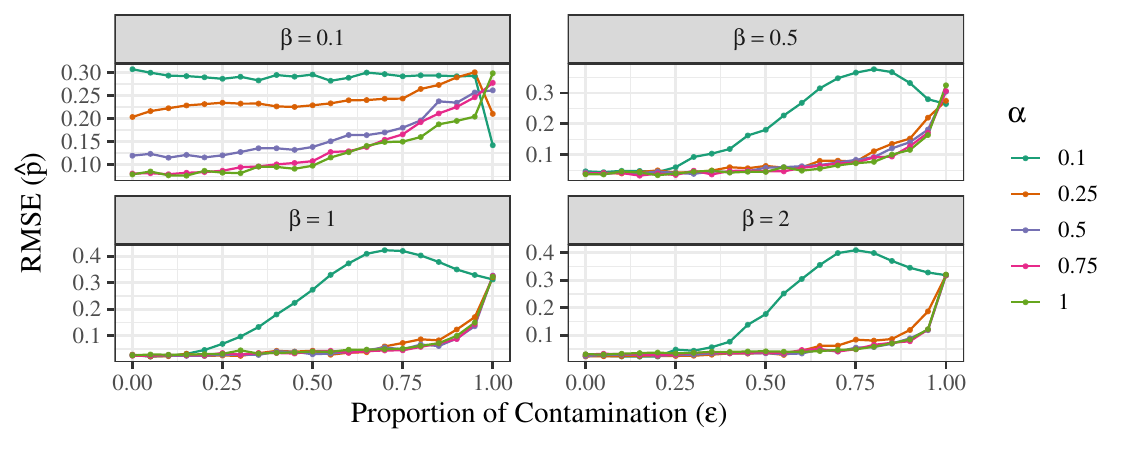}
    \caption{Empirical RMSE of estimated probability of success ($\hat{p}$) for varying contamination, estimated via MGABDE with $\psi(x) = \log(x)$.}
    \label{fig:nb-log}
\end{figure}

\begin{figure}[htbp]
    \centering
    \includegraphics[width=\linewidth]{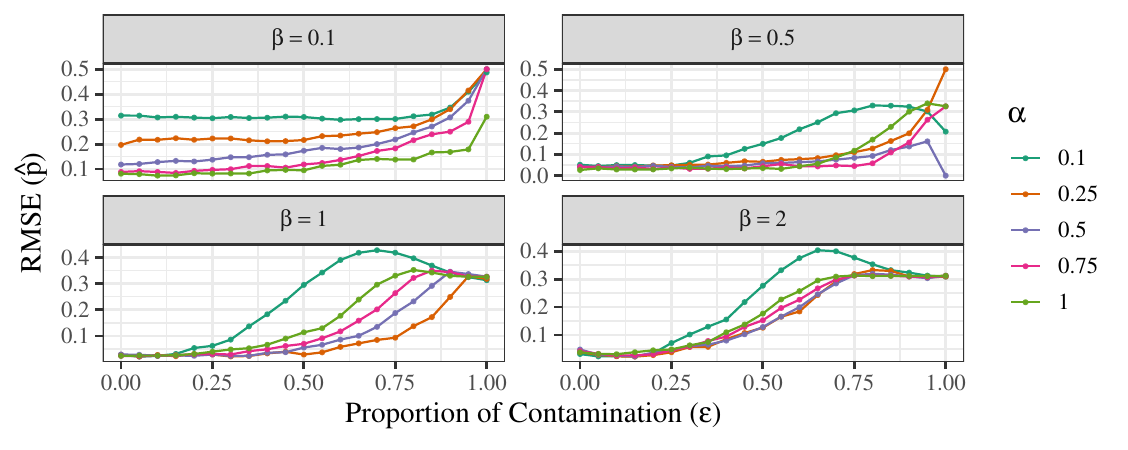}
    \caption{Empirical RMSE of estimated probability of success ($\hat{p}$) for varying contamination, estimated via MGABDE with $\psi(x) = \log((1+x)/2)$.}
    \label{fig:nb-bridge}
\end{figure}

\begin{figure}[htbp]
    \centering
    \includegraphics[width=\linewidth]{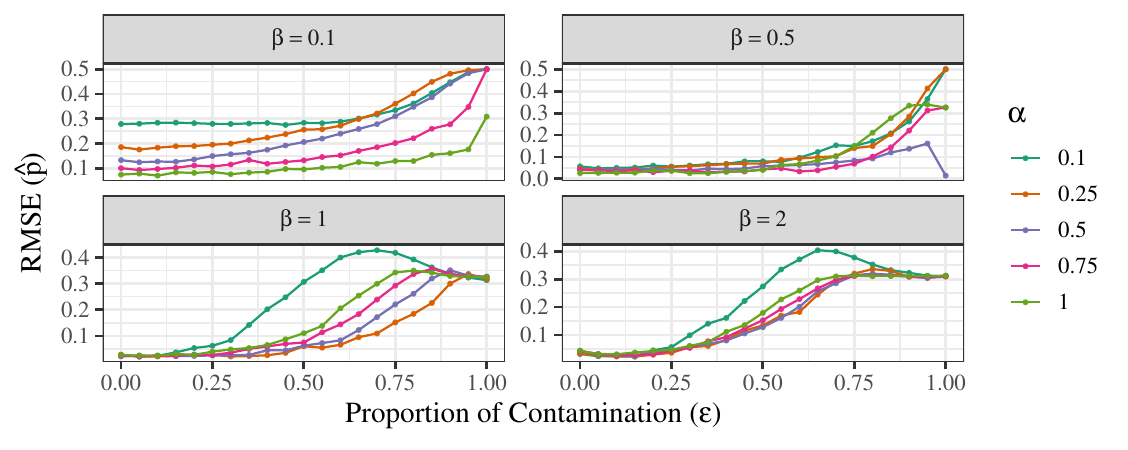}
    \caption{Empirical RMSE of estimated probability of success ($\hat{p}$) for varying contamination, estimated via MGABDE with $\psi(x) = x$.}
    \label{fig:nb-identity}
\end{figure}

\begin{figure}[htbp]
    \centering
    \includegraphics[width=\linewidth]{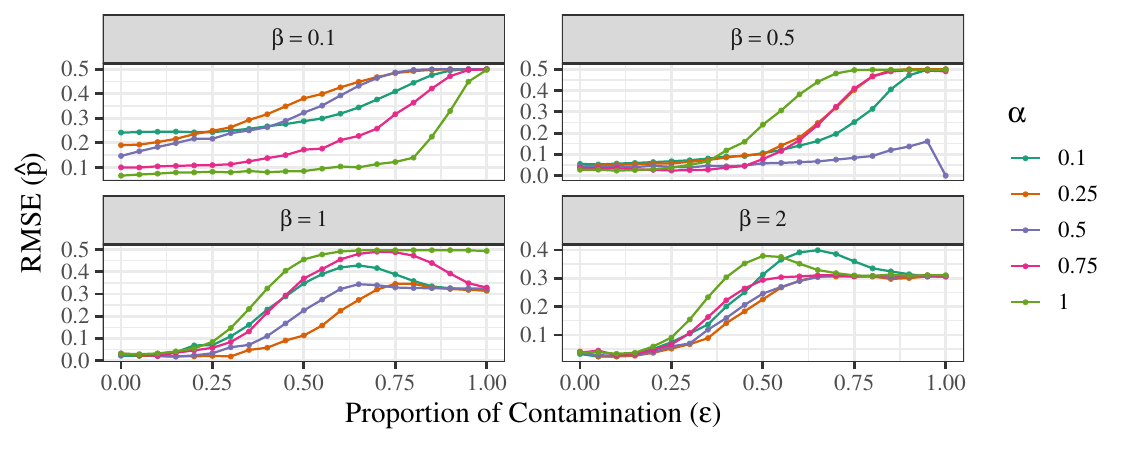}
    \caption{Empirical RMSE of estimated probability of success ($\hat{p}$) for varying contamination, estimated via MGABDE with $\psi(x) = (x^2 - 1)/2$.}
    \label{fig:nb-quad}
\end{figure}

Our first experiment in this particular setting is set to demonstrate the breakdown behavior of the MGABDE for varying $\alpha, \beta$ and different choices of $\psi$ functions. For each choice of the MGABDE and a contamination proportion $\epsilon$, we generate $n = 50$ observations from the mixture distribution $(1-\epsilon)\text{NB}(r = 5, p = 0.5) + \epsilon \text{Poisson}(\lambda = 100)$. Based on $500$ repetitions, for each sample, we compute the MGABDE using the objective function given in Eq.~\eqref{eqn:H-theta-discrete-empirical}, and obtain an empirical root mean squared error (RMSE). Figures~\ref{fig:nb-log}-\ref{fig:nb-quad} respectively depict these empirical RMSE curvature as a function of the proportion of contamination $\epsilon$. As evident from these plots, for a fixed $\psi$-function, increasing the parameter $\alpha$ or $\beta$ generally increases the robustness behavior. This follows from a consequence of Corollary~\ref{cor:gab-bp-special1}, as the lower bound to the asymptotic breakdown point takes the form
\begin{equation*}
    \epsilon^{(\alpha,\beta), \psi} = \left( \frac{\alpha}{\alpha + \beta} \right)^{1/\beta \phi},
\end{equation*}
\noindent where $\psi(x) = (x^\phi - 1)/\phi$. For any fixed $\beta > 0$, $\epsilon^{(\alpha,\beta), \psi}$ is increasing in $\beta$, while for any fixed $\alpha > 0$, $\epsilon^{(\alpha,\beta), \psi}$ is also increasing in $\beta$. Also, increasing the $\beta$ moves the curves for varying $\alpha$'s closer, i.e., for a high $\beta$, the breakdown behavior of the MGABDE changes slowly with respect to changes in $\alpha$. This can be also theoretically be justified by the fact that for $\alpha, \beta > 0$,
\begin{equation*}
    \frac{\partial (\log(\epsilon^{(\alpha,\beta),\psi}))}{\partial\alpha \partial\beta} = -\frac{1}{\alpha \phi(\alpha+\beta)^2} < 0.
\end{equation*}

Our second experiment seeks to understand the relationship between the asymptotic breakdown point and the asymptotic variance. To this end, we fix $\beta = 1$ and compute the MGABDEs for both $r$ and $p$ across various choices of $\alpha \in (0, 1]$, utilizing the optimal generating function $\psi^*$ established in Theorem~\ref{thm:gab-optimality-discrete}. We anticipate that the optimal estimator will exhibit robust and consistent performance even as the contamination proportion $\epsilon$ increases from 0 to $1/2$. Figure~\ref{fig:nb-optimal} validates this expectation: the empirical RMSE for $\hat{p}$ remains extremely low (never exceeding $0.04$), while the RMSE for $\hat{r}$ remains strictly below $1$. Interestingly, selecting a relatively small tuning parameter ($\alpha = 0.1$) results in a marginal increase in error as contamination grows. A similar trend is observed at the opposite extreme ($\alpha = 1$), which is likely attributable to the well-documented variance inflation associated with increasing $\alpha$~\citep{basu1998robust}. Other intermediate values of $\alpha$ lead to approximately consistent RMSE across all contamination levels within this setting.

\begin{figure}[htbp]
    \centering
    \includegraphics[width=0.49\linewidth]{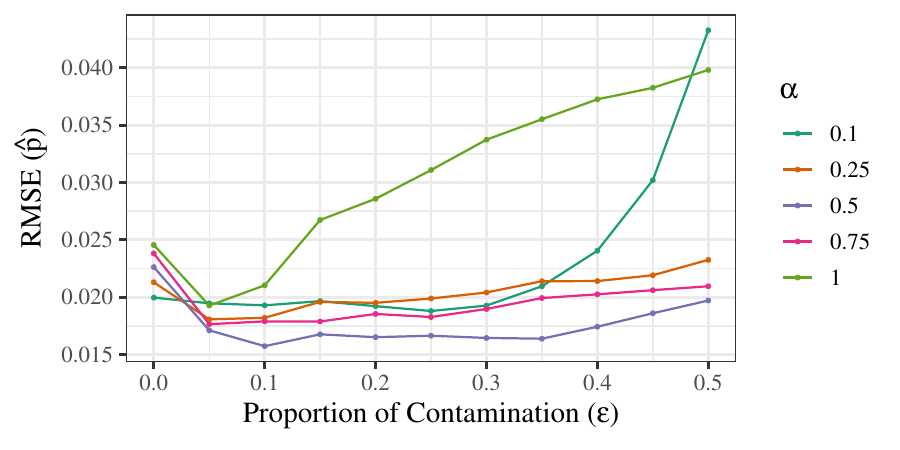}
    \includegraphics[width=0.49\linewidth]{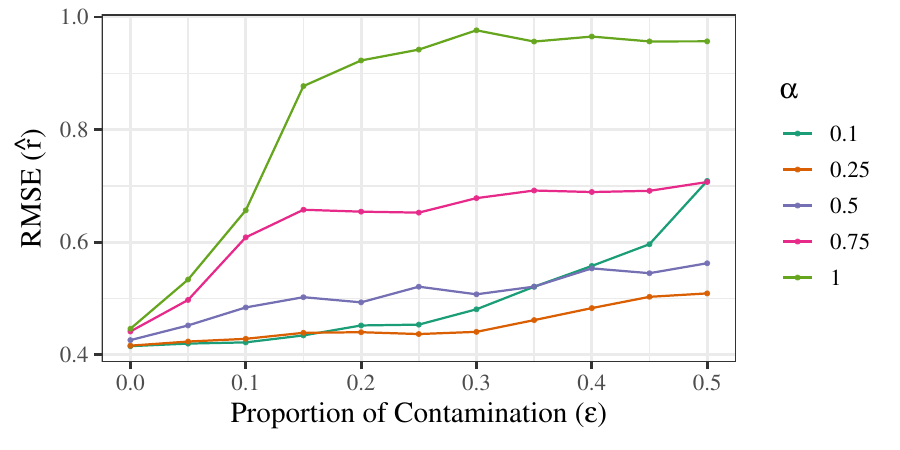}
    \caption{Empirical RMSE of estimated $\hat{p}$ (left) and $\hat{r}$ (right) via the optimal $\psi$-divergence for different choices of $\alpha$ with $\beta = 1$.}
    \label{fig:nb-optimal}
\end{figure}

\subsection{Relation between Asymptotic Variance and Breakdown Point}
In order to assess the relationship between the asymptotic variance and the asymptotic breakdown point of the optimal MGABD estimators, we have plotted heatmaps of the asymptotic variance as a function of the parameter $\alpha$ and the asymptotic breakdown point $\epsilon^*$ keeping $\beta$ fixed at $1$ for different families of distributions. Three univariate distributions were used for this - the normal scale family $\normdist(0,\sigma^2)$, the gamma scale family $\phi(t,\theta)$ with known $t$ and $\theta$ being the scale parameter in question, and the gamma shape family $\phi(t,1)$ with unknown $t$.

\begin{figure}[htbp]
    \centering
    \includegraphics[width=0.32\linewidth]{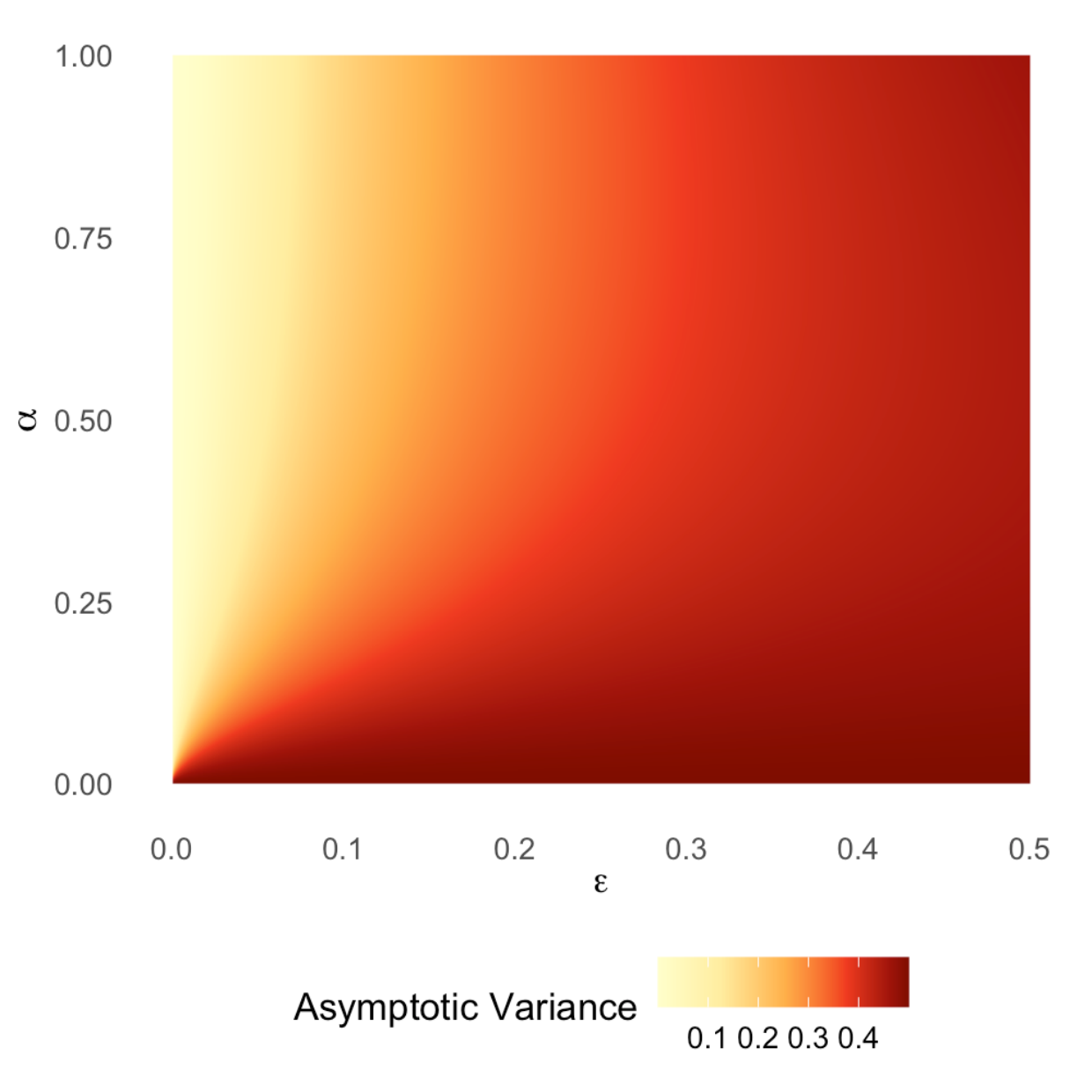}
    \includegraphics[width=0.32\linewidth]{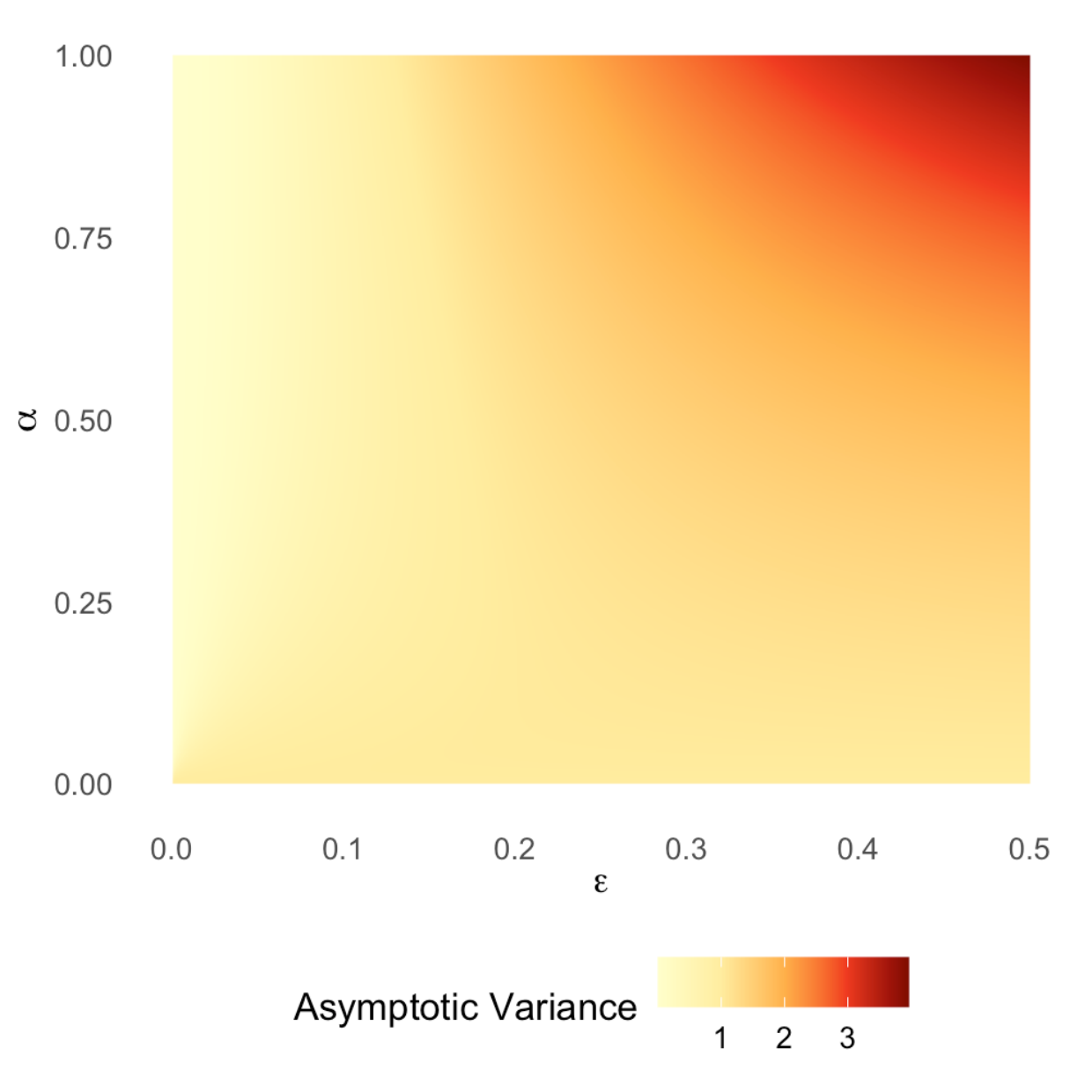}
    \includegraphics[width=0.32\linewidth]{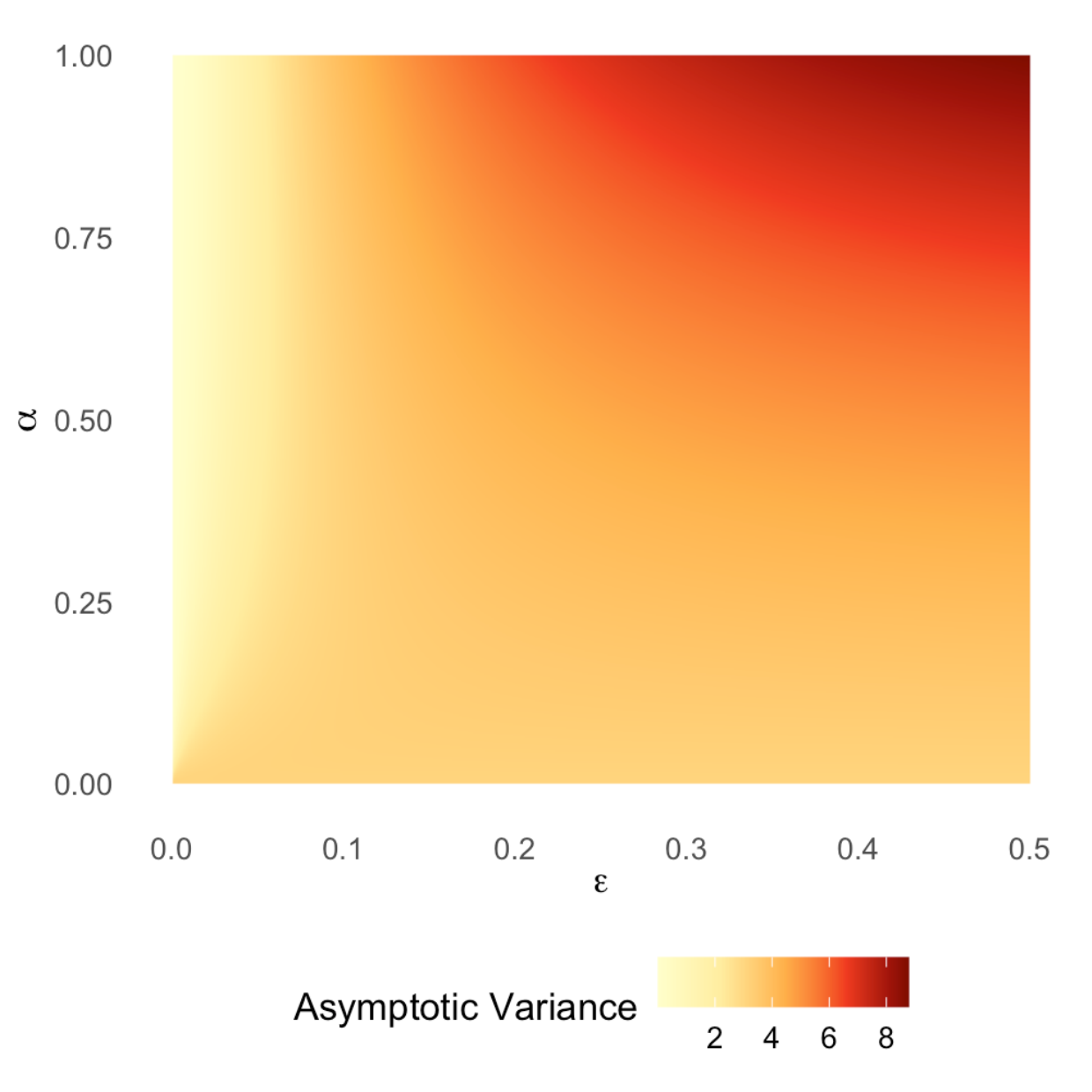}
    \caption{Heatmap showing the variation of  Asymptotic variance with breakdown point $\epsilon$ and $\alpha$ for (left to right) normal scale family, gamma scale family, and gamma shape family.}
    \label{fig:heatmap-av-optimal}
\end{figure}
Figure \ref{fig:heatmap-av-optimal} illustrates the relationship that the asymptotic variance of the optimal MGABD estimators have with the breakdown point and $\alpha$. The optimal asymptotic variance is the same as that mentioned in Section \ref{sec:optimality} for $\psi(x)=(x^{\phi}-1)/\phi$ with $\phi$ chosen to be
\begin{equation*}
    \phi^*=-\frac{\log(1+1/\alpha)}{\log(1-\epsilon^*)}
\end{equation*}
\noindent For the normal scale model, the true value of the parameter $\sigma$ is chosen to be $1$, for the gamma scale family, $t$ is set to $2$, and the true value of $\theta$ is taken to be $1$, while for the gamma shape family, the true value $t=2$ was selected. It is observed that the asymptotic variance decreases with an increase in the value of $\alpha$ for the normal scale model, while for the gamma models it seems to be an increasing function of $\alpha$. It is also observed that with an increase in $\epsilon$, the asymptotic variance also increases. This effect is more pronounced for higher values of $\alpha$. This is consistent with the fact that the optimal asymptotic variance is a decreasing function of $\phi^*$, which again is a decreasing function of $\epsilon^*$. Furthermore, this affirms the belief that there is a tradeoff between the asymptotic variance and the asymptotic breakdown point, implying that a loss in asymptotic efficiency is the price to be paid for a higher degree of robustness, even when one uses the minimax optimal estimators.

\bibliographystyle{plainnat}
\bibliography{references}       

\end{document}